\documentclass{article}

\usepackage{arxiv}

\usepackage[utf8]{inputenc} 
\usepackage[T1]{fontenc}    
\usepackage{hyperref}       
\usepackage{url}            
\usepackage{booktabs}       
\usepackage{amsfonts}       
\usepackage{nicefrac}       
\usepackage{microtype}      
\usepackage{lipsum}		
\usepackage{graphicx}
\usepackage[numbers]{natbib}
\usepackage{doi}

\usepackage{amsmath,amssymb,a4wide}
\usepackage{amsthm}
\newtheorem{theorem}{Theorem}
\newtheorem{remark}[theorem]{Remark}
\newtheorem{lemma}[theorem]{Lemma}
\newtheorem{hypo}[theorem]{Hypothesis}
\usepackage{color,graphicx,epstopdf}
\usepackage{lineno}
\usepackage{color}
\usepackage{tikz}
\usetikzlibrary{patterns}
\usepackage{diagbox}
\usepackage{pgfplots}
\usepackage{pgfplotstable}
\usepackage{enumitem}
\usepackage{hyperref}
\usepackage{fp-eval}
\usepackage{esint}
\usepackage{subfigure}
\usepackage{caption}
\usepackage{mathabx}
\usepackage{booktabs}
\usepackage{subdepth}
\usepackage{nicefrac}
\usepackage{xspace}
\usepackage{units}

\newcommand{\bb}{\boldsymbol b}
\newcommand{\bx}{\boldsymbol x}
\newcommand{\bu}{\boldsymbol u}
\newcommand{\bv}{\boldsymbol v}

\newcommand{\bV}{\boldsymbol V}
\newcommand{\bee}{\boldsymbol e}
\newcommand{\bH}{\boldsymbol H}

\newcommand{\bg}{\boldsymbol g}
\newcommand{\bL}{\boldsymbol L}
\newcommand{\bn}{\boldsymbol n}
\newcommand{\bs}{\boldsymbol s}
\newcommand{\by}{\boldsymbol y}
\newcommand{\bff}{\boldsymbol f}

\newcommand{\bzero}{\boldsymbol 0}
\newcommand{\bI}{\boldsymbol I}
\newcommand{\gD}{\Gamma_{\mathrm{D}}}
\newcommand{\gN}{\Gamma_{\mathrm{N}}}
\newcommand{\di}{\mathrm d}
\newcommand{\bR}{\boldsymbol R}
\newcommand{\bQ}{\boldsymbol Q}
\newcommand{\bP}{\boldsymbol P}
\newcommand{\br}{\boldsymbol r}
\newcommand{\dive}{\nabla\cdot}
\newcommand{\foralls}{\forall\;}
\newcommand{\Ihv}{\bI_h^{\bV}}
\newcommand{\Ihq}{I_h^{Q}}
\newcommand{\cinv}{c_{\mathrm{inv}}}
\renewcommand{\(}{\left(}
\renewcommand{\)}{\right)}
\newcommand{\jp}[1]{\left[\!\!\; \left|#1 \right|\!\!\;\right]_F}
\newcommand{\no}[2]{\left\| #1 \right\|_{#2}}
\newcommand{\blist}{\begin{list}{}{\itemsep0.0ex\parsep0.1ex\topsep0.2ex\leftmargin1.6em
\labelwidth1.3em}}

\title{A robust a posteriori error estimator for the Oseen problem}

\author{ 
	Muhammad Afzal  \\
	Center for Applied Mathematics and Bioinformatics\\ Department of Mathematics and Natural Sciences,\\ Gulf University for Science and Technology \\Mubarak Al-Abdullah Area/West Mishref, Kuwait \\
	\texttt{afzal.m@gust.edu.kw} \\
	\And 
	\href{https://orcid.org/0000-0002-9322-0373}{\includegraphics[scale=0.06]{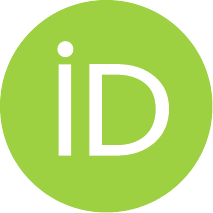}\hspace{1mm} Naveed Ahmed }\\
	Center for Applied Mathematics and Bioinformatics\\ Department of Mathematics and Natural Sciences,\\ Gulf University for Science and Technology \\Mubarak Al-Abdullah Area/West Mishref, Kuwait \\
	\texttt{ahmed.n@gust.edu.kw}\\
	\And
	\href{https://orcid.org/0000-0002-2711-4409}{\includegraphics[scale=0.06]{orcid.pdf}\hspace{1mm}Volker John} \\
	Weierstrass Institute for Applied Analysis and Stochastics (WIAS)\\ Anton-Wilhelm-Amo Str. 39\\ 
	10117 Berlin, Germany \\
	\texttt{john@wias-berlin.de} \\
}

\renewcommand{\shorttitle}{A robust a posteriori error estimator for the Oseen problem}

\begin{document}
\maketitle

\begin{abstract} A residual-based a posteriori error estimator is proposed for the incompressible Oseen problem 
in the convection-dominated regime. The SUPG/PSPG/grad-div stabilized finite element method
is used as discretization. The error estimator estimates the global error in a norm that is used 
in the a priori error analysis of the method. Based on several hypotheses concerning the error and 
interpolation errors, the robustness of the estimator in the convection-dominated regime is proved.
Numerical studies support the analytic results. Finally, the extension of the a posteriori 
error estimator to the steady-state Navier--Stokes equations is discussed. 
\end{abstract}


\section{Introduction}

This paper considers the stationary incompressible Oseen problem, already with dimensionless quantities, 
for the velocity $\bu\ :\ \Omega\to\mathbb{R}^d$ and the 
pressure $p\ :\ \Omega\to\mathbb{R}$ that is of the form
\begin{equation}\label{eq:oseen-strong}
\begin{array}{rcll}
-\nu\,\Delta \bu + (\bb\!\cdot\!\nabla)\bu + \sigma \bu + \nabla p &=& \bff & \mbox{in }\Omega, \\
\nabla \cdot \bu &=& 0 & \mbox{in }\Omega,\\
\bu &=& \bzero & \text{on }\gD,\\
(\nu \nabla\bu - p\mathbb I)\cdot\bn &= & \bg  & \mbox{on }\gN,
\end{array}
\end{equation}
where $\Omega\subset\mathbb{R}^d$, $d\in \{2,3\}$, is a polygonal (polyhedral) domain with Lipschitz boundary 
$\partial\Omega=\gD\cup\gN$, $\gD\cap\gN=\emptyset$. The unit outer normal vector at 
$\partial\Omega$ is denoted by $\bn$ and $\mathbb I$ is the unit tensor. 
The Dirichlet part $\gD$ has positive 
$(d-1)$-dimensional measure and contains the inflow boundary
\begin{equation}\label{eq:inflow_bdry}
\partial\Omega^-:=\left\{ \bx\in\partial\Omega\ :\ \ \bb(\bx)\cdot\bn(\bx)<0 \right\}\subset\gD.
\end{equation}
Given data comprise the (kinematic) viscosity $0<\nu\ll 1$, the  convection field
$\bb \in W^{1,\infty}(\Omega)^d= {\boldsymbol W}^{1,\infty}(\Omega)$, which is assumed to be divergence-free $\nabla\cdot\bb=0$ a.e. in $\Omega$, 
a scalar reaction field
$\sigma\in L^\infty(\Omega)$ with $\sigma\ge \sigma_0\ge 0$, $\sigma_0 \in \mathbb R$,  a.e. in $\Omega$, acting componentwise on $\bu$, a body force
$\bff\in L^2(\Omega)^d = \bL^2(\Omega)$, and boundary data
$\bg\in L^2(\gN)^d=\bL^2(\gN)$. The consideration of homogeneous data at the Dirichlet boundary is just 
for the sake of simplifying the presentation of the numerical analysis. Inhomogeneous 
Dirichlet boundary data could be incorporated with a standard lifting argument. 

Oseen problems are standard model problems in numerical analysis. They can be considered as 
a linearized incompressible Navier--Stokes problem. In practice, Oseen-type problems have to be solved in each 
step of certain iterative methods in each time step for time-dependent Navier--Stokes problems, where $\sigma$ is 
proportional to the inverse of the length of the time step. 

A situation of much interest in practice is the case that $\nu \ll \|\bb\|_{L^\infty(\Omega)}$, the 
so-called convection-dominated case. In this case, where the convective transport dominates the viscous one, 
the solution might possess very small scales, e.g., 
layers, that cannot be resolved on affordable grids. 
Thus, from the numerical point of view, one 
faces a multiscale problem. It is well known that the numerical simulation of multiscale problems 
requires to use special numerical methods to perform stable simulations. There are various proposals
of stabilized finite element methods for Oseen problems, which include the streamline-upwind Petrov--Galerkin (SUPG)/pressure-stabilization Petrov--Galerkin (PSPG) 
method, e.g., see \cite{HS90,TL91,FF92} for some of the early applications of this method to incompressible flow problems,
the local projection stabilization (LPS) method, e.g. see \cite{BB06} for its application to the Oseen problem, or the continuous interior 
penalty (CIP) method, studied for the Oseen problem in \cite{BFH06}. The most popular one is probably the first method. Another 
stabilization, which is popular for finite element methods applied to all kinds of incompressible flow problems, 
is the grad-div stabilization from \cite{FH88}. Originally, this stabilization was proposed to improve 
the local mass conservation of finite element solutions. But it has been found that it also stabilizes 
dominant convection, e.g., see \cite{DGJN18,GJN21}.

This paper considers SUPG/PSPG/grad-div stabilized finite element methods for the Oseen problem. With 
this method, both inf-sup stable pairs of velocity-pressure finite element spaces and equal order 
pairs can be used. The a priori error analysis of this method, see \cite{TV96}, aims to prove a 
so-called convection-robust (or semi-robust) estimate. That is, the constant in the estimate does not 
blow up if $\nu\to 0$. To achieve such an estimate, one has to use an appropriate norm, which contains 
contributions from the stabilization, compare \eqref{eq:norm_spg} and \eqref{eq:norm_spg_0} below.

The goals of an a posteriori error analysis are twofold. First, a derived error estimator should 
provide information about the global error in some norm. Second, the local contributions of the estimator 
should control the construction of adaptively refined grids. There are several types of a posteriori
error estimators, e.g., residual-based estimators, see \cite{Ver13}, or estimators that are based on the 
solution of local problems, see \cite{AO00}. In particular, estimators for the Oseen or the steady-state 
Navier--Stokes equations can be found, e.g., in \cite{AO97,Ber02,Bur15,ADR16,BCG17}. As discussed in \cite{JKN18},
none of these estimators leads to convection-robust estimates. More recently, in \cite{KK20}, a robust 
residual-based estimator for a $\bH(\mathrm{div})$-conforming discontinuous Galerkin (DG) finite element method 
for the Oseen equations is presented and analyzed. This method contains only the 
standard stabilization term of DG methods.
The robust error bound is derived with respect to a norm that has as one part 
a dual seminorm, which can be only approximated, but not computed. 

This paper can be considered as an extension of \cite{JN13}, where a robust residual-based error estimator 
for convection-dominated scalar equations was proposed. Like in \cite{JN13}, one of the goals is that the estimator 
should be for the same norm for which the a priori finite element error analysis of the method is performed. 
And again like in \cite{JN13}, the derivation of the estimator is based on some hypotheses concerning the relation 
of interpolation errors and the error of the finite element solution in the stabilized norm. Some of the hypotheses
are analogs to those from \cite{JN13}, but some new hypotheses are needed, in particular coming from the fact 
that the second component of the solution is the pressure field. With this approach, a residual-based
a posteriori error estimator is derived and an estimate for the global error, in a norm that 
is used in the a priori error analysis, is proved. Numerical studies  support the robustness of the proposed 
estimator. In addition, the satisfaction of the hypotheses is studied. Finally, an extension of the 
proposed estimator to the steady-state incompressible Navier--Stokes problem is presented and illustrated with a 
brief numerical study. 

The paper is organized as follows. Section~\ref{sec:weak} introduces the SUPG/PSPG/grad-div stabilized
finite element method for the Oseen problem, together with all necessary notations. The residual-based 
a posteriori error estimator is derived and discussed in Section~\ref{sec:est}. Section~\ref{sec:numres}
contains the numerical studies for the Oseen problem. The discussion of an extension of the error 
estimator to the steady-state Navier--Stokes problem is presented in Section~\ref{sec:nse}.

\section{Notations, Preliminaries, Weak Formulation, the SUPG/PSPG/grad-div Stabilized Finite Element Method}\label{sec:weak}

This paper uses standard notations for Lebesgue and Sobolev spaces and their norms. Vector-valued 
functions and spaces are denoted with bold-faced symbols. Let $\omega$ be a domain or part of 
a manifold, then the inner product in $L^2(\omega)$ or $\bL^2(\omega)$ is denoted by $(\cdot,\cdot)_\omega$. 
If $\omega=\Omega$, then the subscript is neglected. 

Let 
\[
\bV = \left\{\bv\in \bH^1(\Omega)\ : \ \bv|_{\gD} = \bzero \right\}
\]
and 
\[
Q = \begin{cases} 
L^2(\Omega) & \mbox{if } \text{meas}_{d-1} \gN > 0, \\
L_0^2(\Omega) & \mbox{else},
\end{cases}
\]
where the functions from $L_0^2(\Omega)$ have integral mean value zero.  Then, a weak formulation 
of the Oseen problem \eqref{eq:oseen-strong} is obtained in the usual way and it reads as follows:
Find $(\bu,p) \in \bV \times Q$ such that 
\begin{eqnarray}
\nu\(\nabla\bu,\nabla\bv\) + \(\(\bb\cdot \nabla\)\bu,\bv\) + \(\sigma\bu,\bv\) - \(p,\nabla \cdot \bv\)
 &=& (\bff,\bv) + (\bg,\bv)_{\gN}
 \quad \foralls \bv\in \bV, \nonumber\\
 & &  \label{eq:weak-1} \\
(\nabla\cdot\bu,q) & = & 0 \quad \foralls q\in Q. \label{eq:weak-2}
\end{eqnarray}
Using that $\bb$ is divergence-free and \eqref{eq:inflow_bdry} gives with integration by parts
\begin{equation}\label{eq:skew_symm}
\(\(\bb\cdot\nabla\)\bv,\bv\) = \frac12 \int_{\gN} \(\bb \cdot\bn\)\(\bv\cdot\bv\)\ \di\bs \ge 0. 
\end{equation}
Since in addition the reaction field $\sigma$ is nonnegative, the operator on the left-hand side of \eqref{eq:weak-1} that describes the velocity-velocity coupling is 
coercive. Together with the inf-sup stability of the spaces $\bV$ and $Q$ and the assumed regularity of the 
data, one can conclude, using the standard theory of linear saddle point problems, that problem
\eqref{eq:weak-1}--\eqref{eq:weak-2} possesses a unique solution. 

Let $\{\mathcal{T}_h\}$, $h>0$, be a family of partitions of $\Omega$ into
mesh cells $K\in\mathcal{T}_h$, either simplices or parallelepipeds. 
It will be assumed that each mesh is admissible in the usual sense, e.g., see \cite[pp.~38 and~51]{Cia78}, and that 
the family is shape-regular. That is, there exists
a constant $\vartheta>0$ such that for each $K\in\mathcal{T}_h$ and for all $\mathcal{T}_h$ it holds
$h_K/\rho_K\le \vartheta$, where $h_K$ denotes the diameter of $K$ and $\rho_K$
the diameter of the largest ball inscribed in $K$. The mesh parameter of a triangulation is defined by 
$h:=\max_{K\in\mathcal{T}_h} h_K$. Denote by $\mathcal F_h$ the set of facets $\{F\}$ ($(d-1)$-dimensional 
faces) and by $\mathcal F_{h,\mathrm{in}}$ the set of facets that 
are not contained in $\partial\Omega$. The interior of each boundary facet should belong either to 
$\gD$ or $\gN$.

As already mentioned in the introduction, the numerical solution of the Oseen problem in the 
convection-dominated regime requires a stabilized discretization. The most popular residual-based 
stabilized finite element method is the SUPG/PSPG/grad-div method. To introduce this method, let 
$\bV_h \times Q_h$ be a pair of conforming finite element spaces, i.e., $\bV_h \subset \bV$ and 
$Q_h\subset Q$. The velocity finite element space is given by 
\[
\bV_h = \left\{\bv_h \in \bR_k \cap \bV, k\ge 1\right\},
\]
where $\bR_k$ is the space $\bP_k$ of continuous piecewise polynomials of at most degree $k$ on simplicial grids 
and the space $\bQ_k$ of continuous piecewise polynomials of degree at most $k$ for each variable on 
meshes consisting of quadrilaterals/hexahedra. Likewise, the space $Q_h$ consists of piecewise polynomials of degree $l\ge0$, 
continuous or discontinuous, so that $Q_h \in \{P_l, P_l^{\mathrm{disc}}\}$ on simplicial grids and 
$Q_h \in \{Q_l, Q_l^{\mathrm{disc}}\}$ on grids consisting of quadrilaterals/hexahedra. Let 
\[
\tilde Q =  \left\{q \in Q \ : \ \left. q\right|_K \in H^1(K) \mbox{ for all } K\in\mathcal T^h \right\}.
\]

Let $K_1, K_2 \in \mathcal T_h$ be two mesh cells 
with a common facet $F = K_1 \cap K_2$.  The unit normal $\bn_F$ at $F$
is chosen to be, without loss of generality, the outward normal with respect to $K_1$.
Then, the jump of a function $v$ across the facet $F$ in the point $\bx\in F$ is
defined by 
\[
\jp{v} = \lim_{\by\to\bx, \by\in K_1} v(\by) - \lim_{\by\to\bx, \by\in K_2} v(\by),
\quad \bx \in F,
\]
if both limits are well defined.

Consider a triangulation  $\mathcal T_h$ of $\Omega$ such that there is a set of facets with 
$\overline{\gN} = \bigcup F$.  Given $\bff \in \bL^2(\Omega)$ and $\bg\in L^2(\gN)^d$, the SUPG/PSPG/grad-div stabilized finite element method seeks $\(\bu_h,p_h\)\in
\bV_h\times Q_h$ such that 
\begin{equation}\label{eq:oseen_SUPG_PSPG_graddiv}
A_{\mathrm{spg}}\(\(\bu_h,p_h\),\(\bv_h,q_h\)\)=L_{\mathrm{spg}}\(\(\bv_h,q_h\)\) \quad 
\foralls \(\bv_h,q_h\)\in \bV_h\times Q_h,
\end{equation}
where the bilinear form $A_{\mathrm{spg}}\ : \ (\bV \times \tilde Q) \times (\bV \times \tilde Q) \to \mathbb R$
is defined by 
\begin{eqnarray}\label{eq:a_spg}
\lefteqn{A_{\mathrm{spg}}\(\(\bu,p\),\(\bv,q\)\)}\nonumber \\
&=&\nu\(\nabla \bu,\nabla \bv\)+\(\(\bb\cdot
\nabla\)\bu+\sigma\bu,\bv\)-\(\dive \bv,p\)+ \(\dive \bu,q\) \nonumber\\
&& +\sum_{K\in\mathcal
T_h} \mu_K \(\nabla \cdot \bu,\nabla \cdot \bv\)_K + \sum_{F\in\mathcal
F_h} \gamma_F \(\jp{p},\jp{q}\)_F \\
&& +\sum_{K\in\mathcal
T_h} \(-\nu \Delta \bu+\(\bb\cdot \nabla \)\bu+\sigma \bu+\nabla
p,\delta_K \( \(\bb\cdot \nabla \)\bv + \nabla q\)\)_K \nonumber
\end{eqnarray}
and the linear form $L_{\mathrm{spg}}\ : \ (\bV \times \tilde Q) \to \mathbb  R$ by
\begin{equation}\label{eq:l_spg}
L_{\mathrm{spg}}\(\(\bv,q\)\)=\(\bff,\bv\)+ \sum_{F\in \overline{\gN}} \(\bg,\bv \)_F +
\sum_{K\in\mathcal T_h} \(\bff,\delta_K \(\(\bb\cdot
\nabla \)\bv + \nabla q\)\)_K.
\end{equation}

The parameters of method \eqref{eq:oseen_SUPG_PSPG_graddiv}--\eqref{eq:l_spg} are the grad-div parameter
$\mu_K \ge 0$, the jump parameter $\gamma_F >0$, and the SUPG parameter $\delta_K >0$. Notice that the jump 
term disappears if $Q_h$ consists of continuous functions. It can be also observed that the SUPG term leads 
to a coupling of pressure ansatz and test functions. This kind of coupling is called PSPG method. Consequently, 
the SUPG/PSPG/grad-div method does not lead to a problem of saddle point type. 

The goal of using stabilized methods consists of computing numerical solutions that show an improved physical 
consistency. For the SUPG/PSPG/grad-div stabilized finite element method of the Oseen equations, spurious
oscillations should be reduced by the SUPG stabilization and the conservation of mass should be improved
by the grad-div stabilization. The reduction of spurious oscillations is usually supported in the numerical
analysis with so-called robust error estimates
(sometimes called semi-robust). These are estimates where the constant does not blow up for $\nu\to0$ and 
all other data fixed. That is, if the dependency on $\nu$ is tracked in the analysis, robust error bounds do not contain 
inverse powers of $\nu$. It is well known that robust estimates for convection-dominated problems can be 
achieved only for special norms. Such norms contain usually terms of the actually used stabilization. 
For the analysis of the  SUPG/PSPG/grad-div method, appropriate (mesh-dependent) norms are 
\begin{eqnarray}\label{eq:norm_spg}
\no{\(\bv,q\)}{\mathrm{spg}}&=&\Bigg(\nu \no{\nabla
\bv}{\bL^2(\Omega)}^2+ \no{\sigma^{1/2}\bv}{\bL^2(\Omega)}^2+ \sum_{K\in\mathcal T_h} \mu_K \no{\dive
\bv}{L^2(K)}^2 \\
&& +\sum_{F\in\mathcal F_h}\gamma_F \no{\jp{q}}{L^2(F)}^2 
+\sum_{K\in\mathcal T_h}\delta_K \no{\(\bb\cdot \nabla\)\bv + \nabla q}{\bL^2(K)}^2\Bigg)^{1/2}
\nonumber 
\end{eqnarray}
and
\begin{equation}\label{eq:norm_spg_0}
\no{\(\bv,q\)}{\mathrm{spg},p} = \( \no{\(\bv,q\)}{\mathrm{spg}}^2 + \omega_{\mathrm{pres}}^{-2}\no{q}{L^2(\Omega)}^2\)^{1/2}
\end{equation}
with
\begin{equation*}
\omega_{\mathrm{pres}} = \max\left\{1,\nu^{-1/2},\no{\sigma}{L^\infty(\Omega)}^{1/2}\right\}.
\end{equation*}

The numerical analysis of the SUPG/PSPG/grad-div method, usually presented for the situation $\gD=\partial\Omega$, 
is well established, showing, under appropriate assumptions on some stabilization parameters, existence and 
uniqueness of a solution and a robust estimate for the errors  $\|(\bu-\bu^h,p-p^h)\|_{\mathrm{spg}}$ 
and $\|(\bu-\bu^h,p-p^h)\|_{\mathrm{spg},p}$, e.g., see \cite{TV96},  \cite[Section~IV.3.1]{RST08}, or 
\cite[Section~5.3.2]{Joh16}. 

\begin{remark}[Optimal asymptotic choice of the stabilization parameters] \label{rem:para_choice}
To achieve optimal order of convergence in these norms, one has to choose 
the stabilization parameters appropriately, e.g., see the discussion in \cite[Remark~5.42]{Joh16}. Here, one 
has to distinguish two situations that are of main interest. 

If one uses a pair of finite element spaces
with $k=l+1$ that 
satisfies a discrete inf-sup condition, like Taylor--Hood spaces $\bP_k/P_{k-1}$, $k\ge2$, then the optimal 
asymptotic choices are 
\begin{equation}\label{eq:para_inf_sup}
\delta_K \sim h_K^2, \quad \mu_K \sim 1, \quad \gamma_F \sim h_K.
\end{equation}
These choices lead to the convergence of order $k$. If in the following 'inf-sup stable pairs of spaces' are 
mentioned, then always pairs with polynomials of order $k$ for the velocity and $k-1$ for the pressure are meant.

In the case of using equal order spaces, i.e., $k=l \ge 1$, 
then the recommended asymptotic choices are 
\begin{equation}\label{eq:para_equal}
\delta_K \sim \begin{cases}
h_K & \mbox{if } \nu < h_K,\\
h_K^2 & \mbox{if } \nu \ge h_K,
\end{cases} \quad
\mu_K \sim \begin{cases} h_K  & \mbox{if } \nu < h_K,\\
h_K \mbox{ or } 1 &  \mbox{if } \nu \ge h_K,
\end{cases}
\quad
\gamma_F \sim \begin{cases}
1 & \mbox{if } \nu < h_K,\\
h_K & \mbox{if } \nu \ge h_K.
\end{cases}
\end{equation}
The order of error reduction in the convection-dominated regime is $k+\nicefrac12$ and in the diffusion-dominated
case $k$. In this paper, the notion `pair of equal order spaces' will be used exclusively for pairs like $\bP_k/P_k$ or 
$\bQ_k/Q_k$.

Comprehensive numerical studies on the choice of the grad-div stabilization parameter for finite element
methods for the Oseen problem can be found in \cite{Ahm17}.
\end{remark}

Let, on simplicial meshes $\bP_k \subset \bV_h$ or on quadrilateral/hexahedral meshes $\bQ_k \subset \bV_h$, $k\ge 1$, 
let $s\in\{0,1\}$ with $s\le t\le k+1$, then denote by $\Ihv\ : \ \bH^t(\Omega) \to \bV_h$ a bounded linear 
interpolation operator. It is assumed that this operator satisfies 
for all $\bv\in \bH^t(\Omega)$ and for all $K\in\mathcal T_h$ the interpolation estimate
\begin{equation}\label{eq:local_int}
\left|\bv-\Ihv\bv\right|_{\bH^s(K)} \le C h_K^{t-s} |\bv|_{\bH^t(K)},
\end{equation}
where $\bH^0(K) = \bL^2(K)$. Using the technique of \cite[Lemma~2.1]{JN11}, one can derive 
for all $\bv \in \bV \cap \bH^{k+1}(\Omega)$ the estimate
\[
\left\|\Delta \(\bv - \Ihv\bv \) \right\|_{\bL^2(K)}^2 \le C h_K^{2k-2} \|\bv \|_{\bH^{r+1}(K)}^2.
\]
Likewise, let all polynomials of degree $l\ge0$, continuous or discontinuous, belong to $Q_h$, 
a bounded linear interpolation operator $\Ihq\ : \ H^t(\Omega) \to Q_h$ is defined with the same interpolation
property as in \eqref{eq:local_int}, with $s\le t\le l+1$.

In addition, a local trace inequality is needed. Let $K\in\mathcal T_h$ be a mesh cell and 
$F\in\partial K$ one of its facets, then the scalar-valued form of this  inequality is
\begin{equation}\label{eq:trace}
\no{q}{L^2(F)} \le C \(h_F^{-1/2} \no{q}{L^2(K)} + \no{q}{L^2(K)}^{1/2} \no{\nabla q}{L^2(K)}^{1/2}\)\quad 
\foralls q\in H^1(K),  
\end{equation}
see \cite[Lemma~3.1]{Ver98} for the proof. In the following, the vector-valued analog is used. 

\section{The a Posteriori Estimator}\label{sec:est}

The derivation of the a posteriori error indicator follows \cite{JN13}. It is based on some 
hypotheses concerning the relation of some norms of interpolation errors and the error in the 
norm $\no{(\cdot,\cdot)}{\mathrm{spg}}$. In particular, hypotheses on pressure interpolation
errors are needed. The difficulty is that the norm $\no{(\cdot,\cdot)}{\mathrm{spg}}$ does not 
contain a separate pressure term. Hence it is hard to find good reasonings for the stated assumptions.
In fact, we cannot provide such a good reasoning for the term that contains the jump of the 
finite element pressure. For this reason, only pairs of finite element spaces will be considered 
with continuous finite element pressure, so that the jump term does not appear. In practice, the 
use of continuous finite element pressures seems to be much more popular when applying the SUPG/PSPG/grad-div method, 
since it simplifies the implementation of this method considerably.

First, some properties of the discrete problem are shown. It can be checked with a straightforward
calculation that the SUPG/PSPG/grad-div method is consistent, i.e., 
\[
A_{\mathrm{spg}}\(\(\bu,p\),\(\bv_h,q_h\)\)=L_{\mathrm{spg}}\(\(\bv_h,q_h\)\) \quad 
\foralls \(\bv_h,q_h\)\in \bV_h\times Q_h,
\]
for a sufficiently smooth solution $(\bu,p) \in \bV\times Q$ of \eqref{eq:weak-1}--\eqref{eq:weak-2},
e.g., compare \cite[Lemma~5.31]{Joh16}.
From consistency the so-called Galerkin orthogonality follows
\begin{equation}\label{eq:gal_orth}
A_{\mathrm{spg}}\(\(\bu-\bu_h,p-p_h\),\(\bv_h,q_h\)\)=0 \quad \foralls \(\bv_h,q_h\)\in \bV_h\times Q_h.
\end{equation}

\begin{lemma}[Preliminary estimate of the error in the norm $\no{\cdot,\cdot)}{\mathrm{spg}}$] 
Let the solution $(\bu,p)$ of \eqref{eq:weak-1}--\eqref{eq:weak-2} be contained in $\bH^2(\Omega)\times H^1(\Omega)$, 
and let $(\bu_h,p_h)$ be the solution of the  SUPG/PSPG/grad-div problem \eqref{eq:oseen_SUPG_PSPG_graddiv}--\eqref{eq:l_spg}.
Assume that 
\begin{equation}\label{eq:ass_lem}
\delta_K \le \min \left\{\frac{h_K^2}{8 \cinv \nu}, \frac1{2\no{\sigma}{L^\infty(\Omega)}} \right\}
\quad \foralls K \in \mathcal T_h,
\end{equation}
then it holds
\begin{eqnarray}\label{eq:prelim_ee}
\no{\(\bu-\bu_h,p-p_h\)}{\mathrm{spg}}^2&\le& 2 A_{\mathrm{spg}}\((\bu-\bu_h,p-p_h),(\bu-\bu_h,p-p_h)\)\nonumber\\
&& + 8 \sum_{K\in \mathcal T_h} \delta_K \cinv^2 \nu^2h_K^{-2} \no{\nabla\(\bu-\Ihv\bu\)}{L^2(K)}^2\nonumber\\
&& + 4  \sum_{K\in \mathcal T_h} \delta_K \nu^2  \no{\Delta\(\bu-\Ihv\bu\)}{L^2(K)}^2.
\end{eqnarray}
\end{lemma}

\begin{proof} The proof follows the proof of \cite[Lemma~1]{JN13}, with only slight technical additions like using 
\eqref{eq:skew_symm}. For this reason, a detailed presentation will be omitted here. 
\end{proof}

As next step, local residuals are defined in the usual way, namely the mesh cell residual 
\[
\br_K\(\bu_h,p_h\) = \left.\(\bff + \nu \Delta \bu_h - \(\bb \cdot \nabla\)\bu_h - \sigma \bu_h - \nabla p_h \)\right|_K
\quad \foralls K \in \mathcal T_h
\]
and the residual at the facets 
\begin{equation}\label{eq:rF}
\br_F\(\bu_h,p_h\) =  \begin{cases}
\jp{\(-\nu \nabla \bu_h + p_h \mathbb I\)}\bn_F, & \mbox{if } F \in \mathcal F_{h,\mathrm{in}},\\
\bg  - \(\nu \nabla \bu_h - p_h \mathbb I\) \bn_F, & \mbox{if } F  \in \mathcal F_h, \ F \subset \overline{\gN},\\
\bzero, &  \mbox{if } F  \in \mathcal F_h, \ F \subset \overline{\gD}.
\end{cases}
\end{equation}

\begin{remark}[Polynomial approximations of coefficients] Strictly speaking, instead of the coefficients $\bff, \bb, \sigma$ one 
would need to use polynomial approximations of these coefficients, so that the arising integrals and norms can be 
computed exactly. Then, the errors committed by these approximations enter the error bound. We decided to neglect 
this technical issue to keep the notations simple and to concentrate in the error bounds below on the 
essential terms. 
\end{remark}

With the notation of the residuals, one obtains, by utilizing the Galerkin orthogonality \eqref{eq:gal_orth}, the weak form 
\eqref{eq:weak-1}--\eqref{eq:weak-2} of the Oseen problem, and integration by parts
\begin{eqnarray}\label{eq:bilin_repr}
\lefteqn{
A_{\mathrm{spg}}\((\bu-\bu_h,p-p_h),(\bv,q)\)}\nonumber\\
& = & A_{\mathrm{spg}}\((\bu-\bu_h,p-p_h),\(\bv-\Ihv\bv,q-\Ihq q\)\) \nonumber\\
& = & \sum_{K\in\mathcal T_h} \(\br_K\(\bu_h,p_h\),\bv-\Ihv\bv\)_K- \(\nabla\cdot\bu_h, q-\Ihq q \)\nonumber\\
&& + \sum_{F\in \mathcal F_h} \(\br_F\(\bu_h,p_h\), \bv-\Ihv\bv\)_F \nonumber \\
&& + \sum_{K\in\mathcal T_h} \(\br_K\(\bu_h,p_h\), \delta_K \((\bb\cdot\nabla)\(\bv-\Ihv\bv\) + \nabla \(q-\Ihq q\)\)\)_K \\
&& - \sum_{K\in\mathcal T_h} \mu_K \(\nabla \cdot \bu_h,\nabla \cdot \(\bv-\Ihv\bv\) \)_K, \nonumber
\end{eqnarray}
for all $(\bv,q) \in \bV\times Q$. 

In the same spirit as in \cite{JN13}, some hypotheses on the relation of certain norms of the interpolation
errors and the error in the norm $\no{(\cdot,\cdot)}{\mathrm{spg}}$ are formulated now. The rationale behind these
hypotheses is that the (unknown) interpolation $(\Ihv\bu, \Ihq p)$ is usually considered to be a very good approximation
of the solution $(\bu,p)$ of \eqref{eq:weak-1}--\eqref{eq:weak-2} and consequently that interpolation errors are usually (much) smaller than 
errors committed by the solution $(\bu_h,p_h)$ of the discrete problem \eqref{eq:oseen_SUPG_PSPG_graddiv}--\eqref{eq:l_spg}.

\begin{hypo}[Relations of the interpolation error and the error in the norm $\no{(\cdot,\cdot)}{\mathrm{spg},p}$]
\label{hypo:errors}
Let $(\bu,p)$ the solution of \eqref{eq:weak-1}--\eqref{eq:weak-2} and  $(\bu_h,p_h)$ be the solution of the  SUPG/PSPG/grad-div problem \eqref{eq:oseen_SUPG_PSPG_graddiv}--\eqref{eq:l_spg}.
Assume that $(\bu,p)$ is sufficiently smooth, at least $\bu \in \bH^{k+1}(\Omega)$ and $p\in H^k(\Omega)$. 

First, it will be assumed that 
\begin{equation}
\sum_{K\in\mathcal T_h} \delta_K^{-1} \no{\bu-\Ihv\bu}{L^2(K)}^2 \le  2  \no{\(\bu-\bu_h,p-p_h\)}{\mathrm{spg}}^2.
\label{eq:L2_delta}
\end{equation}
The discussion of this hypothesis and all the followings is based on the special case of a  family of uniform triangulations. 
Consider on the one hand the optimal choices of $\delta_K$ and the corresponding orders of convergence 
discussed in Remark~\ref{rem:para_choice} and, on the other hand, the interpolation estimate \eqref{eq:local_int}
for $s=0$ and smooth functions. Then one finds that both sides of \eqref{eq:L2_delta} are of the same order. 
Thus, the actual hypothesis is the factor $2$ on the right-hand side. 

An assumption of a similar spirit as \eqref{eq:L2_delta} is 
\begin{equation}
\sum_{K\in\mathcal T_h} \delta_K^\alpha \no{\nabla\(\bu-\Ihv\bu\)}{L^2(K)}^2 \le  2  \no{\(\bu-\bu_h,p-p_h\)}{\mathrm{spg}}^2
\label{eq:H1_delta}
\end{equation}
with 
\begin{equation}\label{eq:alpha}
\alpha = \begin{cases} 1 & \mbox{if } \nu <  h \mbox{ and equal order spaces are used,}\\
0 & \mbox{else.}
\end{cases}
\end{equation}
Comparing here the asymptotics of both sides, then one finds again that the orders are the same.

The third assumption reads as follows
\begin{equation}
\sum_{F\in\mathcal F_h} \no{\bb}{L^\infty(F)}^2 \no{\bu-\Ihv\bu}{L^2(F)}^2 \le  2  \no{\(\bu-\bu_h,p-p_h\)}{\mathrm{spg}}^2.
\label{eq:L2_F_b}
\end{equation} 
With the trace inequality \eqref{eq:trace} and the convergence orders discussed in Remark~\ref{rem:para_choice}, it 
can be shown that term on the left-hand side of \eqref{eq:L2_F_b} is of order $2k+1$, whereas the 
term on the right-hand side is of order $2k$ or $2k+1$, depending on the concrete situation, compare Remark~\ref{rem:para_choice}.
Concerning this assumption, notice that the corresponding hypothesis in \cite{JN13} contains the factor 
$\no{\bb}{L^\infty(F)}$ on the left-hand side. However, since the convection field appears quadratically on 
the right-hand side, we think it more appropriate from the scaling point of view that it appears quadratically
on the left-hand side, too. 

Next, it will be assumed that 
\begin{equation}
\sum_{K\in\mathcal T_h} \delta_K  \no{\(\bb\cdot \nabla\) \(\bu-\Ihv\bu\) + \nabla \(p-\Ihq p\)}{L^2(K)}^2 \le  2 \no{\(\bu-\bu_h,p-p_h\)}{\mathrm{spg}}^2.
\label{eq:res_delta}
\end{equation}
In the case of inf-sup stable pairs of finite element spaces with order $k$ for the velocity and $k-1$ for the pressure, 
both sides of \eqref{eq:res_delta} are of the same order $2k$ if the optimal asymptotic choice \eqref{eq:para_inf_sup} of $\delta_K$ 
is used. Notice that both interpolation operators can be chosen independently, so that one can consider the order 
of the interpolation errors for both terms on the left-hand side separately. Performing the same considerations for 
equal order pairs of spaces and the corresponding optimal asymptotic choice \eqref{eq:para_equal} of $\delta_K$, 
one can see that both sides of \eqref{eq:res_delta} are also of the same order, provided that $p \in H^{k+1}(\Omega)$.

Additionally, an assumption for the interpolation error of the pressure is needed. Since the pressure error does not 
appear as a separate term in the error measured in the norm $\no{(\cdot,\cdot)}{\mathrm{spg}}$, the pressure interpolation 
cannot be related directly to the error in this norm. Instead, the pressure interpolation will be related to the velocity
interpolation. Consider first the case of equal order finite element spaces and assume that $p \in H^{k+1}(\Omega)$.
In this case, both interpolation operators map to a finite element space with the same polynomials. On the one hand, 
the interpolation operators are not identical, since $\Ihv$ has to obey Dirichlet boundary conditions and 
if $\partial\Omega =\gD$ then $\Ihq$ has to map to a function with vanishing integral mean value. But on the other hand,
it can be expected that generally the velocity interpolation error is not much worse than the pressure interpolation
error. Thus, it will be assumed that 
\begin{equation}\label{eq:pressure_hypo}
\no{p-\Ihq p}{L^2(\Omega)} \le 2 \no{\bu-\Ihv\bu}{L^2(\Omega)}.
\end{equation}
Of course, there are special situations, e.g., if $\bu \in\bV_h, p\not\in Q_h$, where this assumption does not hold.
Inserting \eqref{eq:pressure_hypo} in \eqref{eq:L2_delta} leads to
\begin{equation}\label{eq:pressure_hypo_equal}
\no{p-\Ihq p}{L^2(\Omega)}^2 \le 4 \max_{K \in \mathcal T_h} \{\delta_K\} \sum_{K\in\mathcal T_h} \delta_K^{-1} \no{\bu-\Ihv\bu}{L^2(K)}^2 \le  8  \max_{K \in \mathcal T_h} \{\delta_K\}  \no{\(\bu-\bu_h,p-p_h\)}{\mathrm{spg}}^2.
\end{equation}
In the case of inf-sup stable pairs of spaces, the pressure interpolator maps to a finite element space with polynomials 
of one degree less than contained in the velocity finite element space. Taking this issue into account and assuming 
again that the velocity interpolation error is not much worse than the pressure interpolation error leads to the 
hypotheses
\begin{equation}\label{eq:pressure_hypo_00}
\no{p-\Ihq p}{L^2(\Omega)} \le 2 h^{-1} \no{\bu-\Ihv\bu}{L^2(\Omega)},
\end{equation}
which gives with  \eqref{eq:L2_delta}
\begin{equation}\label{eq:pressure_hypo_infsup}
\no{p-\Ihq p}{L^2(\Omega)}^2 \le   \frac{8}{h^2}  \max_{K \in \mathcal T_h} \{\delta_K\}  \no{\(\bu-\bu_h,p-p_h\)}{\mathrm{spg}}^2.
\end{equation}
The hypotheses \eqref{eq:pressure_hypo} and \eqref{eq:pressure_hypo_00} are based on the assumption that the, 
usually unknown, velocity and pressure components of the solution are of a similar complexity. If it is known
a priori that this is not the case, then the constants on the right-hand sides of  \eqref{eq:pressure_hypo} and \eqref{eq:pressure_hypo_00} 
could be adapted, one direction or the other. 
\end{hypo}

\begin{theorem}[A posteriori error estimate, upper bound]\label{thm:upper_bound}
Let $(\bu,p)$ be a sufficiently smooth solution of  \eqref{eq:weak-1}--\eqref{eq:weak-2}
and let $(\bu_h,p_h)$ be the solution of the  SUPG/PSPG/grad-div problem \eqref{eq:oseen_SUPG_PSPG_graddiv}--\eqref{eq:l_spg}.
Assume \eqref{eq:ass_lem}. Under the hypotheses stated in Remark~\ref{hypo:errors}, the following estimate holds 
\begin{eqnarray}\label{eq:est}
\lefteqn{\no{\(\bu-\bu_h,p-p_h\)}{\mathrm{spg}}^2} \nonumber \\
&\le& \eta_{K,\mathrm{res}}^2 + \eta_{K,\mathrm{div}}^2 + \eta_F^2 + \eta_{K,\delta}^2 + \eta_{K,\mu}^2 
+ 16 \sum_{K\in \mathcal T_h} \delta_K \cinv^2 \nu^2h_K^{-2} \no{\nabla\(\bu-\Ihv\bu\)}{L^2(K)}^2\nonumber\\
&& + 8  \sum_{K\in \mathcal T_h} \delta_K \nu^2  \no{\Delta\(\bu-\Ihv\bu\)}{L^2(K)}^2,
\end{eqnarray}
where
\begin{eqnarray*}
\eta_{K,\mathrm{res}}^2 & = &  \min\left\{ \frac{C}{\sigma_0},\frac{Ch_K^2}{\nu} ,  40 \delta_K   \right\}
\no{\br_K(\bu_h,p_h)}{L^2(K)}^2, \\
\eta_{K,\mathrm{div}}^2 & = &
\begin{cases} 160 \max_{K \in \mathcal T_h} \{\delta_K\} \no{\nabla\cdot\bu_h}{L^2(\Omega)}^2, & \mbox{for equal order pairs of spaces,}\\
\frac{160}{h^2} \max_{K \in \mathcal T_h} \{\delta_K\} \no{\nabla\cdot\bu_h}{L^2(\Omega)}^2, & \mbox{for inf-sup stable pairs of spaces,}\\
\end{cases} \\
\eta_F^2 & = & \min\left\{\frac{Ch_K}{\nu},\frac{C}{h_K\sigma_0},  \frac{C}{\nu^{1/2}\sigma_0^{1/2}}, \frac{40}{\no{\bb}{L^\infty(F)}^2}
\right\} \no{\br_F\(\bu_h,p_h\)}{L^2(F)}^2, \\
\eta_{K,\delta}^2 & = & 40 \sum_{K\in\mathcal T_h} \delta_K \no{\br_K\(\bu_h,p_h\)}{L^2(K)}^2,  \\ 
\eta_{K,\mu}^2 & = & \min\left\{\frac{40 d \mu_K^2}{\delta_K^\alpha},  \frac{C d \mu_K^2}{\nu}\right\} \no{\nabla \cdot \bu_h}{L^2(K)}^2 
\end{eqnarray*}
with $\alpha$ defined in \eqref{eq:alpha} and positive constants $C$, which might be different in the different expressions, 
but which are independent of $h$, the coefficients of the problem, and the stabilization parameters. 
\end{theorem}

\begin{proof} Setting $(\bv,q) = (\bu-\bu_h,p-p_h)$ in \eqref{eq:bilin_repr}, observing that 
\begin{equation}\label{eq:err_inter}
\(\bu-\bu_h\) - \Ihv\(\bu-\bu_h\) = \bu-\Ihv\bu, \quad \(p-p_h\) - \Ihq  \(p-p_h\) = p - \Ihq p,
\end{equation}
and inserting the result in \eqref{eq:prelim_ee}  yields 
\begin{eqnarray}\label{eq:start_est}
\lefteqn{\no{\(\bu-\bu_h,p-p_h\)}{\mathrm{spg}}^2} \nonumber \\
&\le&
2 \sum_{K\in\mathcal T_h} \(\br_K\(\bu_h,p_h\),\bu-\Ihv\bu\)_K- 2\(\nabla\cdot\bu_h, p-\Ihq p \)\nonumber\\
&& + 2\sum_{F\in \mathcal F_h} \(\br_F\(\bu_h,p_h\), \bu-\Ihv\bu\)_F \nonumber \\
&& + 2\sum_{K\in\mathcal T_h} \(\br_K\(\bu_h,p_h\), \delta_K \((\bb\cdot\nabla)\(\bu-\Ihv\bu\) + \nabla \(p-\Ihq p\)\)\)_K \nonumber \\
&& - 2 \sum_{K\in\mathcal T_h} \mu_K \(\nabla \cdot \bu_h,\nabla \cdot \(\bu-\Ihv\bu\) \)_K \nonumber\\
&& + 8 \sum_{K\in \mathcal T_h} \delta_K \cinv^2 \nu^2h_K^{-2} \no{\nabla\(\bu-\Ihv\bu\)}{L^2(K)}^2\nonumber\\
&& + 4  \sum_{K\in \mathcal T_h} \delta_K \nu^2  \no{\Delta\(\bu-\Ihv\bu\)}{L^2(K)}^2.
\end{eqnarray}
The terms on the right-hand side will be estimated separately, some of them in various ways. 

For the first term on the right-hand side of \eqref{eq:start_est}, one obtains by using the 
Cauchy--Schwarz inequality, \eqref{eq:err_inter}, \eqref{eq:local_int} with $s=t=0$, and Young's inequality 
\[
\sum_{K\in\mathcal T_h} \(\br_K\(\bu_h,p_h\),\bu-\Ihv\bu\)_K \le  \sum_{K\in\mathcal T_h} \frac{C}{\sigma_0}
\no{\br_K}{L^2(K)}^2 + \frac1{20} \no{\sigma^{1/2}\(\bu-\bu_h\)}{L^2(\Omega)}^2.
\]
Using a similar way, with $s=0$, $t=1$ in \eqref{eq:local_int}, leads to 
\[
\sum_{K\in\mathcal T_h} \(\br_K\(\bu_h,p_h\),\bu-\Ihv\bu\)_K \le \sum_{K\in\mathcal T_h} \frac{Ch_K^2}{\nu} 
\no{\br_K}{L^2(K)}^2 + \frac{\nu}{20} \no{\nabla\(\bu-\bu_h\)}{L^2(\Omega)}^2.
\]
A third estimate can be performed by utilizing hypothesis \eqref{eq:L2_delta}, which gives 
\begin{eqnarray*}
\sum_{K\in\mathcal T_h} \(\br_K\(\bu_h,p_h\),\bu-\Ihv\bu\)_K &\le& \sum_{K\in\mathcal T_h} 10 \delta_K \no{\br_K}{L^2(K)}^2 
+ \frac1{40} \sum_{K\in\mathcal T_h} \delta_K^{-1} \no{\bu-\Ihv\bu}{L^2(K)}^2 \\
& \le & \sum_{K\in\mathcal T_h} 10 \delta_K \no{\br_K}{L^2(K)}^2 + \frac1{20} \no{\(\bu-\bu_h,p-p_h\)}{\mathrm{spg}}^2.
\end{eqnarray*}
Collecting these estimates yields
\begin{eqnarray*}
\lefteqn{
\sum_{K\in\mathcal T_h} \(\br_K\(\bu_h,p_h\),\bu-\Ihv\bu\)_K}\\
&\le&\sum_{K\in\mathcal T_h}   \min\left\{ \frac{C}{\sigma_0},\frac{Ch_K^2}{\nu} ,  10 \delta_K   \right\}
\no{\br_K}{L^2(K)}^2  + \frac1{20} \no{\(\bu-\bu_h,p-p_h\)}{\mathrm{spg}}^2.
\end{eqnarray*}

For the second term on the right-hand side of \eqref{eq:start_est} consider first equal order pairs of finite element 
spaces. Then, this term is estimated by the Cauchy--Schwarz inequality, \eqref{eq:pressure_hypo_equal}, and Korn's inequality
\begin{eqnarray*}
- \(\nabla\cdot\bu_h, p-\Ihq p \)
&\le& \no{\nabla\cdot\bu_h}{L^2(\Omega)} \no{p-\Ihq p }{L^2(\Omega)} \\
& \le & 40 \max_{K \in \mathcal T_h} \{\delta_K\} \no{\nabla\cdot\bu_h}{L^2(\Omega)}^2 + 
\frac1{20}  \no{\(\bu-\bu_h,p-p_h\)}{\mathrm{spg}}^2.
\end{eqnarray*}
To derive the estimate for inf-sup stable pairs of spaces, one uses hypothesis \eqref{eq:pressure_hypo_infsup}
to obtain 
\[
- \(\nabla\cdot\bu_h, p-\Ihq p \) \le \frac{40}{h^2} \max_{K \in \mathcal T_h} \{\delta_K\} \no{\nabla\cdot\bu_h}{L^2(\Omega)}^2 + 
\frac1{20}  \no{\(\bu-\bu_h,p-p_h\)}{\mathrm{spg}}^2.
\]

Coming to the third term on the right-hand side of \eqref{eq:start_est}, its estimate starts with the 
Cauchy--Schwarz inequality and \eqref{eq:err_inter}, leading to 
\begin{equation}\label{eq:res_facets}
\sum_{F\in \mathcal F_h} \(\br_F\(\bu_h,p_h\), \bu-\Ihv\bu\)_F
\le  \sum_{F\in \mathcal F_h}  \no{\br_F\(\bu_h,p_h\)}{L^2(F)} \no{\(\bu-\bu_h\) - \Ihv\(\bu-\bu_h\) }{L^2(F)}.
\end{equation}
Denote $\bee_{\bu} = \bu-\bu_h$, then the second term can be estimated by the local trace inequality \eqref{eq:trace},
the interpolation estimate \eqref{eq:local_int}, the shape-regularity of the family of triangulations, 
Young's inequality, and the definition of the norm $\no{\cdot}{\mathrm{spg}}$ 
\begin{eqnarray*}
\lefteqn{ \no{\bee_{\bu} - \Ihv\bee_{\bu}}{L^2(F)}}\\
& \le&
 C \(h_F^{-1/2} \no{\bee_{\bu}- \Ihv\bee_{\bu}}{L^2(K)} + \no{\bee_{\bu}- \Ihv\bee_{\bu}}{L^2(K)}^{1/2} \no{\nabla\( \bee_{\bu}- \Ihv\bee_{\bu}\)}{L^2(K)}^{1/2}\)\\
 & \le &  C  h_K^{-1/2}  \min\left\{\frac{h_K}{\nu^{1/2}},\frac1{\sigma_0^{1/2}} \right\} \(\nu^{1/2}\no{\nabla \bee_{\bu}}{L^2(K)}
 + \no{\sigma^{1/2} \bee_{\bu}} {L^2(K)}\)\\
&& + C \nu^{-1/4}  \min\left\{\frac{h_K}{\nu^{1/2}},\frac1{\sigma_0^{1/2}} \right\}^{1/2} \(\nu^{1/2} \no{\nabla \bee_{\bu}}{L^2(K)}
 + \nu^{1/4}\no{\nabla \bee_{\bu}}{L^2(K)}^{1/2}\no{\sigma^{1/2} \bee_{\bu}}{L^2(K)}^{1/2}\)\\
& \le &   C \( h_K^{-1/2}  \min\left\{\frac{h_K}{\nu^{1/2}},\frac1{\sigma_0^{1/2}} \right\}
+ \nu^{-1/4}  \min\left\{\frac{h_K}{\nu^{1/2}},\frac1{\sigma_0^{1/2}} \right\}^{1/2}\)\\
&& \times \(\nu^{1/2}\no{\nabla \bee_{\bu}}{L^2(K)}
 + \no{\sigma^{1/2} \bee_{\bu}} {L^2(K)}\)\\
&\le &  C   \min\left\{\frac{h_K^{1/2}}{\nu^{1/2}},\frac1{h_K^{1/2}\sigma_0^{1/2}},  \frac1{\nu^{1/4}\sigma_0^{1/4}}
\right\}\no{\(\bu-\bu_h,p-p_h\)}{\mathrm{spg}(K)}.
\end{eqnarray*}
This estimate is inserted in \eqref{eq:res_facets} and Young's inequality is applied to give 
\begin{eqnarray*}
\lefteqn{\sum_{F\in \mathcal F_h} \(\br_F\(\bu_h,p_h\), \bu-\Ihv\bu\)_F}\\
&\le& \sum_{F\in \mathcal F_h} C \min\left\{\frac{h_K}{\nu},\frac1{h_K\sigma_0},  \frac1{\nu^{1/2}\sigma_0^{1/2}}
\right\}
\no{\br_F\(\bu_h,p_h\)}{L^2(F)}^2 + \frac1{20} \no{\(\bu-\bu_h,p-p_h\)}{\mathrm{spg}}^2.
\end{eqnarray*}
Alternatively, applying Young's inequality in \eqref{eq:res_facets} and then utilizing hypothesis \eqref{eq:L2_F_b} leads to 
\begin{eqnarray*}
\lefteqn{\sum_{F\in \mathcal F_h} \(\br_F\(\bu_h,p_h\), \bu-\Ihv\bu\)_F}\\
&\le& \sum_{F\in \mathcal F_h} \frac{10}{\no{\bb}{L^\infty(F)}^2}
\no{\br_F\(\bu_h,p_h\)}{L^2(F)}^2 + \frac1{20} \no{\(\bu-\bu_h,p-p_h\)}{\mathrm{spg}}^2.
\end{eqnarray*}

The SUPG term is estimated simply by using the Cauchy--Schwarz inequality, Young's inequality, and hypothesis 
\eqref{eq:res_delta}
\begin{eqnarray*}
\lefteqn{\sum_{K\in\mathcal T_h} \(\br_K\(\bu_h,p_h\), \delta_K \((\bb\cdot\nabla)\(\bu-\Ihv\bu\) + \nabla \(p-\Ihq p\)\)\)_K}\\
&\le & \sum_{K\in\mathcal T_h} 10 \delta_K \no{\br_K\(\bu_h,p_h\)}{L^2(K)}^2 + \frac1{40}
\sum_{K\in\mathcal T_h} \delta_K \no{(\bb\cdot\nabla)\(\bu-\Ihv\bu\) + \nabla \(p-\Ihq p\)}{L^2(K)}^2 \\
&\le& \sum_{K\in\mathcal T_h} 10 \delta_K \no{\br_K\(\bu_h,p_h\)}{L^2(K)}^2 + \frac1{20}\no{\(\bu-\bu_h,p-p_h\)}{\mathrm{spg}}^2.
\end{eqnarray*}

The term in \eqref{eq:start_est} coming from the grad-div stabilization 
can be estimated by using the Cauchy--Schwarz inequality, 
a standard estimate of the divergence by the gradient, Young's inequality, and
hypothesis \eqref{eq:H1_delta}, so that 
\begin{eqnarray*}
\lefteqn{\sum_{K\in\mathcal T_h} \mu_K \(\nabla \cdot \bu_h,\nabla \cdot \(\bu-\Ihv\bu\) \)_K}\\
& \le & \sum_{K\in\mathcal T_h} \mu_K \no{\nabla \cdot \bu_h}{L^2(K)} \no{\nabla \cdot \(\bu-\Ihv\bu\)}{L^2(K)}\\
& \le & \sum_{K\in\mathcal T_h} \frac{10 d \mu_K^2}{\delta_K^\alpha} \no{\nabla \cdot \bu_h}{L^2(K)}^2 + \frac{1}{40}\sum_{K\in\mathcal T_h} 
\delta_K^\alpha \no{\nabla \(\bu-\Ihv\bu\)}{L^2(K)}^2\\
& \le & \sum_{K\in\mathcal T_h} \frac{10 d \mu_K^2}{\delta_K^\alpha} \no{\nabla \cdot \bu_h}{L^2(K)}^2 + \frac1{20}\no{\(\bu-\bu_h,p-p_h\)}{\mathrm{spg}}^2.
\end{eqnarray*}
Alternatively, one can use \eqref{eq:err_inter}, the estimate of the divergence by the gradient and 
the interpolation property \eqref{eq:local_int} to obtain
\begin{eqnarray*}
\lefteqn{\sum_{K\in\mathcal T_h} \mu_K \(\nabla \cdot \bu_h,\nabla \cdot \(\bu-\Ihv\bu\) \)_K}\\
& \le &\sum_{K\in\mathcal T_h} \mu_K \no{\nabla \cdot \bu_h}{L^2(K)} \no{\nabla \cdot \((\bu-\bu_h) - 
\Ihv\(\bu-\bu_h\)\)}{L^2(K)}\\
& \le &\sum_{K\in\mathcal T_h} C d^{1/2} \mu_K \no{\nabla \cdot \bu_h}{L^2(K)} \no{\nabla (\bu-\bu_h)}{L^2(K)}\\
& \le & \sum_{K\in\mathcal T_h} \frac{C d \mu_K^2}{\nu} \no{\nabla \cdot \bu_h}{L^2(K)}^2 +
\frac1{20}\no{\(\bu-\bu_h,p-p_h\)}{\mathrm{spg}}^2.
\end{eqnarray*}

Inserting the derived bounds for the terms on the right-hand side of \eqref{eq:start_est} in this estimate gives
the statement of the theorem. 
\end{proof}

\begin{remark}[On the last two terms on the right-hand side of \eqref{eq:est}]\label{rem:est_last_two_terms}
The last two terms on the right-hand side of \eqref{eq:est} are of the same form as in the corresponding error bound for the 
scalar convection-diffusion-reaction problem presented in \cite{JN13}. These terms are discussed in 
\cite[Remark~3]{JN13}. Here, this discussion shall not be repeated and it is just notices that these terms can be expected 
to be small in the convection-dominated regime, which comes essentially from the factor $\nu^2$, 
compare also Figure~\ref{fig:smooth-hypos} below. For the 
diffusion-dominated regime, one can apply other techniques than in the present paper to derive global a posteriori
error bounds that do not lead to terms of this form. 
\end{remark}

\begin{remark}[On adaptive mesh refinement] As already mentioned in the introduction, a second goal of utilizing a posteriori error estimators, 
besides providing information about global errors, consists of controlling an adaptive mesh refinement. 
To support the use of an a posteriori error estimator from the point of view of numerical analysis, local
estimates are derived so that the local error is bounded by a local contribution of the error estimator. 
In the case of residual-based estimators, one applies cut-off functions, so-called bubble functions, 
to derive the local estimate, see \cite{Ver94}. The derivation is usually quite technical and very 
lengthy, compare \cite[Section~4]{JN13} for the case of convection-dominated convection-diffusion-reaction 
problems. Since the main goal of the current paper consists of proposing a robust estimator for the 
error $\no{\(\bu-\bu_h,p-p_h\)}{\mathrm{spg}}$, as well as for the sake of brevity, we decided not to 
present a local error analysis. The numerical studies in Section~\ref{sec:layer} will show that 
with the proposed estimator properly refined grids are obtained. 
\end{remark}

\section{Numerical Studies}\label{sec:numres}

This section presents numerical studies that investigate the robustness in the convection-dominated regime, with respect to the $\no{(\cdot,\cdot)}{\mathrm{spg}}$
norm of the error, of the a posteriori error estimator
\begin{equation}\label{eq:eta}
\eta = \left( \sum_{K\in\mathcal T_h} \eta_{K,\mathrm{res}}^2 + \eta_{K,\mathrm{div}}^2 + \eta_F^2 + \eta_{K,\delta}^2 + \eta_{K,\mu}^2 \right)^{1/2},
\end{equation}
where the individual contributions are defined in Theorem~\ref{thm:upper_bound} and always the 
constant $C=1$ was used. Two examples in 
two dimensions are considered. The first one has a smooth solution without layers and the solution
of the second example exhibits boundary layers. Both examples are defined in the unit square. 

Simulations were performed on triangular grids, where on the coarsest grid, 
the unit square was divided with a diagonal from 
bottom left to top right, and then refinement was applied. The equal order pairs $P_1/P_1$, 
$P_2,P_2$, $P_3/P_3$ and the inf-sup stable Taylor--Hood pairs $P_2/P_1$, $P_3/P_2$ were 
included in our numerical studies. The stabilization parameters \eqref{eq:para_inf_sup} for the 
inf-sup stable pairs were chosen to be 
\[
\delta_K = 0.5 h_K^2, \quad \mu_K = 0.5,
\]
and the parameters \eqref{eq:para_equal} for the equal order pairs 
\[
\delta_K = \begin{cases}
0.5  h_K & \mbox{if } \nu < h_K,\\
0.5  h_K^2 & \mbox{if } \nu \ge h_K,
\end{cases} \quad
\mu_K = 0.5  h_K.
\]
The total number of degrees of freedom for velocity (including Dirichlet nodes) and pressure is abbreviated with `DoFs'. All simulations were performed with the code {\em MooNMD}, \cite{JM04}.

\subsection{Smooth solution without layers} \label{sec:smooth}

The right-hand side and the Dirichlet boundary conditions are chosen such that
\begin{eqnarray}\label{eq:smooth_u}
\bu(x,y) & = & \begin{pmatrix}  1000 x^2 (1-x)^4 y^2 (1-y)(3-5y) \\
-2000 x (1-x)^3 (1-3x) y^3 (1-y)^2 \end{pmatrix}, \\
p(x,y) &=& \pi^2 ( x y^3 \cos (2\pi x^2 y) - x^2 y \sin (2\pi x y) )
          + \frac{1}{8} \label{eq:smooth_p}
\end{eqnarray}
is the solution of the Oseen problem \eqref{eq:oseen-strong}, where the convection
field is chosen to be $\bb = \bu$ and the reaction parameter
set to be  $\sigma = 1$. For this example, uniformly refined grids were used. 

Since the solution does not possess layers, one expects from the a priori error 
analysis that the errors $\no{\(\bu-\bu_h,p-p_h\)}{\mathrm{spg}}$ behave robustly, 
i.e., they do not blow up for small values of $\nu$. Exactly this behavior can be 
observed in Figure~\ref{fig:smooth-all-error}. In addition, one can see the optimal 
orders of error reduction for the convection-dominated regime, compare Remark~\ref{rem:para_choice}.

\begin{figure}[t!]
\centerline{\includegraphics[width=0.32\textwidth]{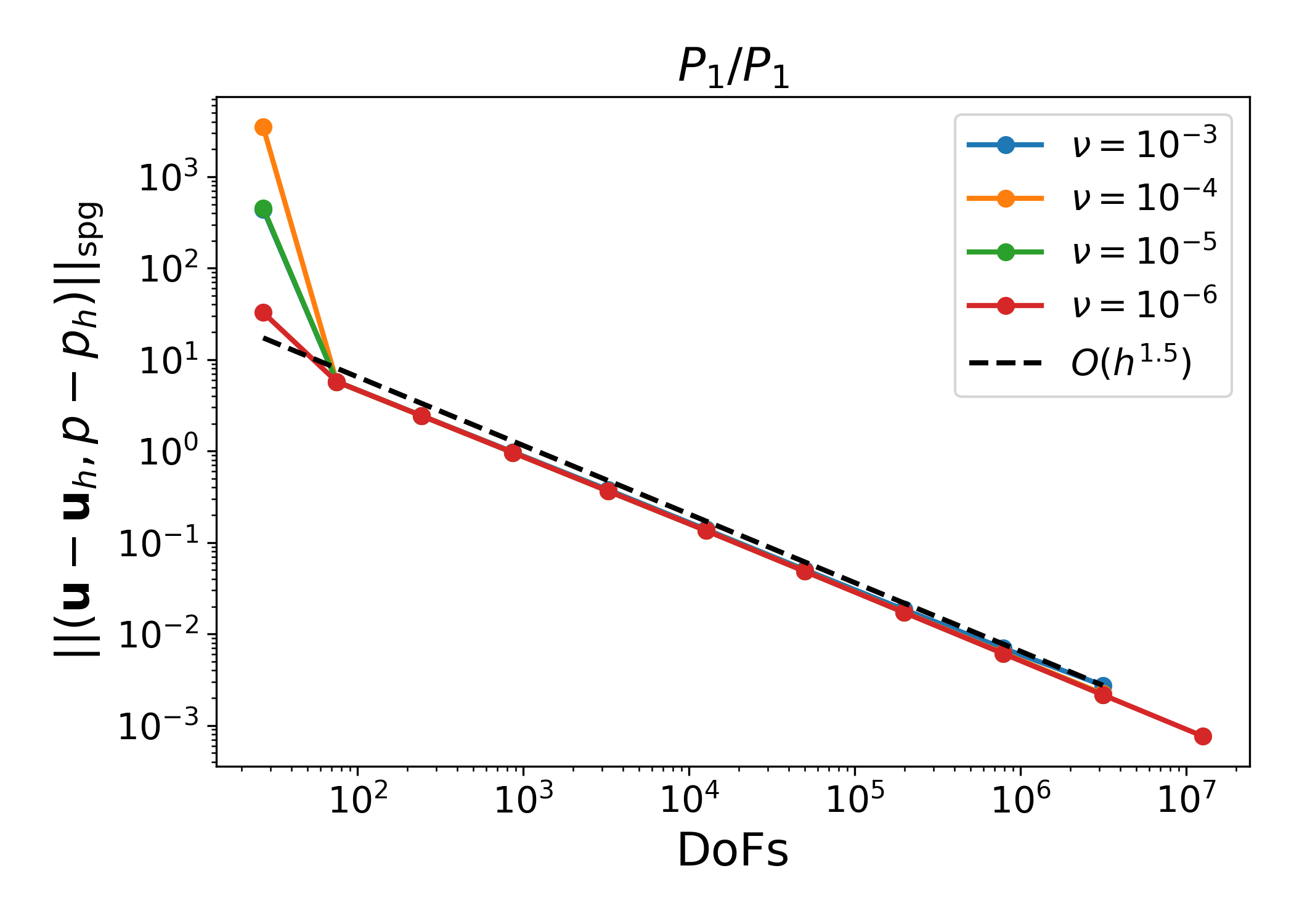}
  \hspace*{0.5em}
  \includegraphics[width=0.32\textwidth]{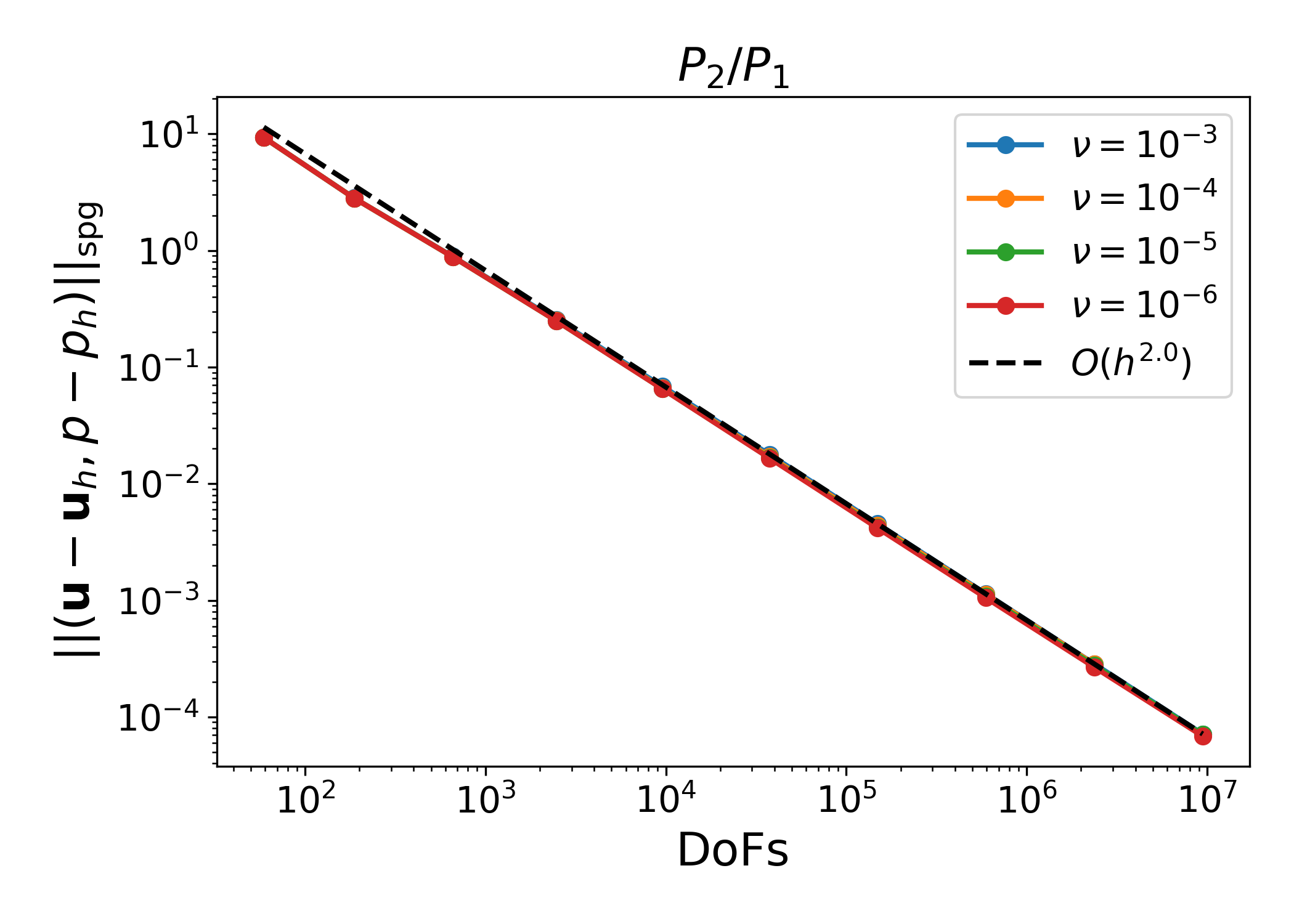}
  \hspace*{0.5em}
  \includegraphics[width=0.32\textwidth]{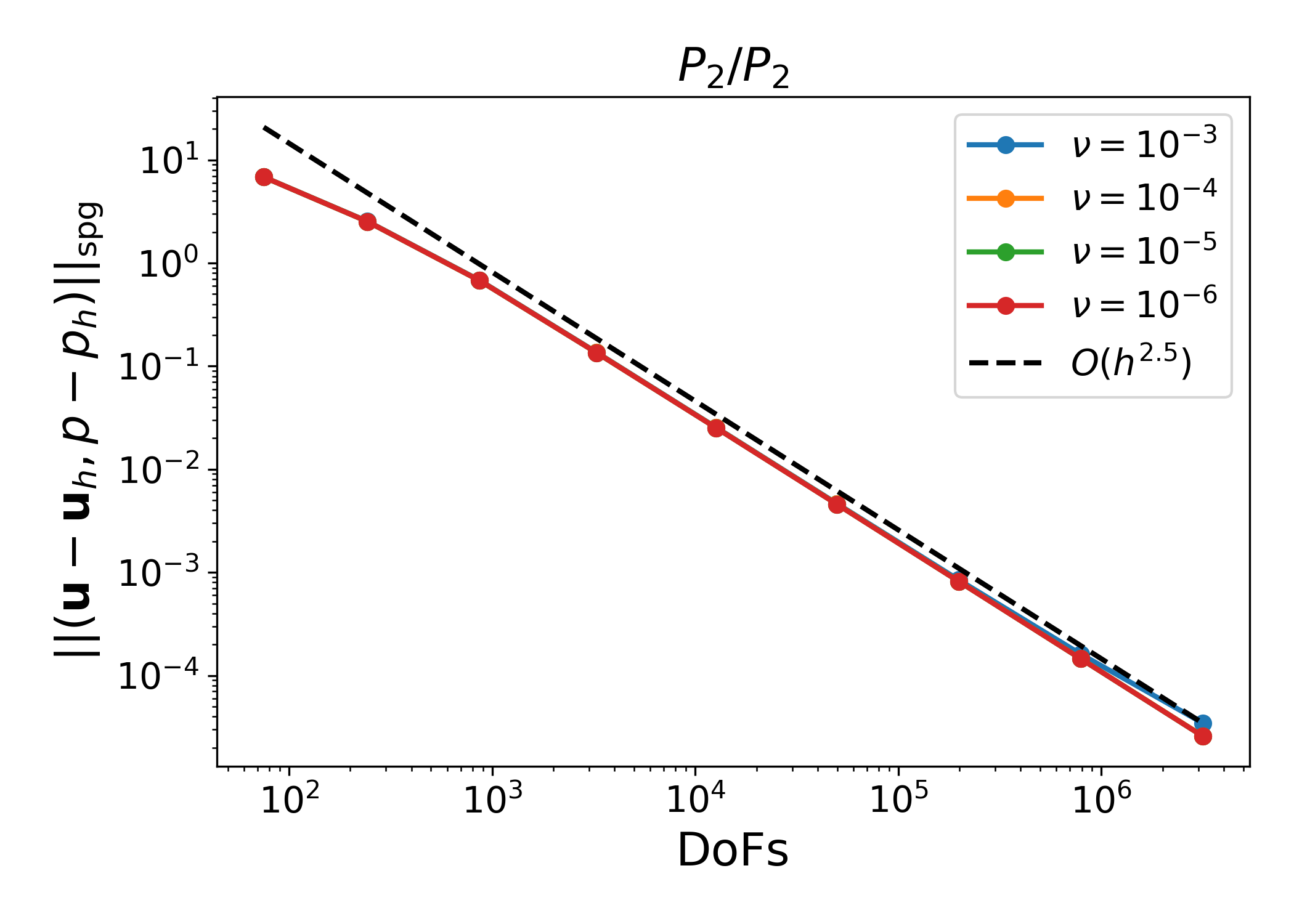}}
\centerline{\includegraphics[width=0.32\textwidth]{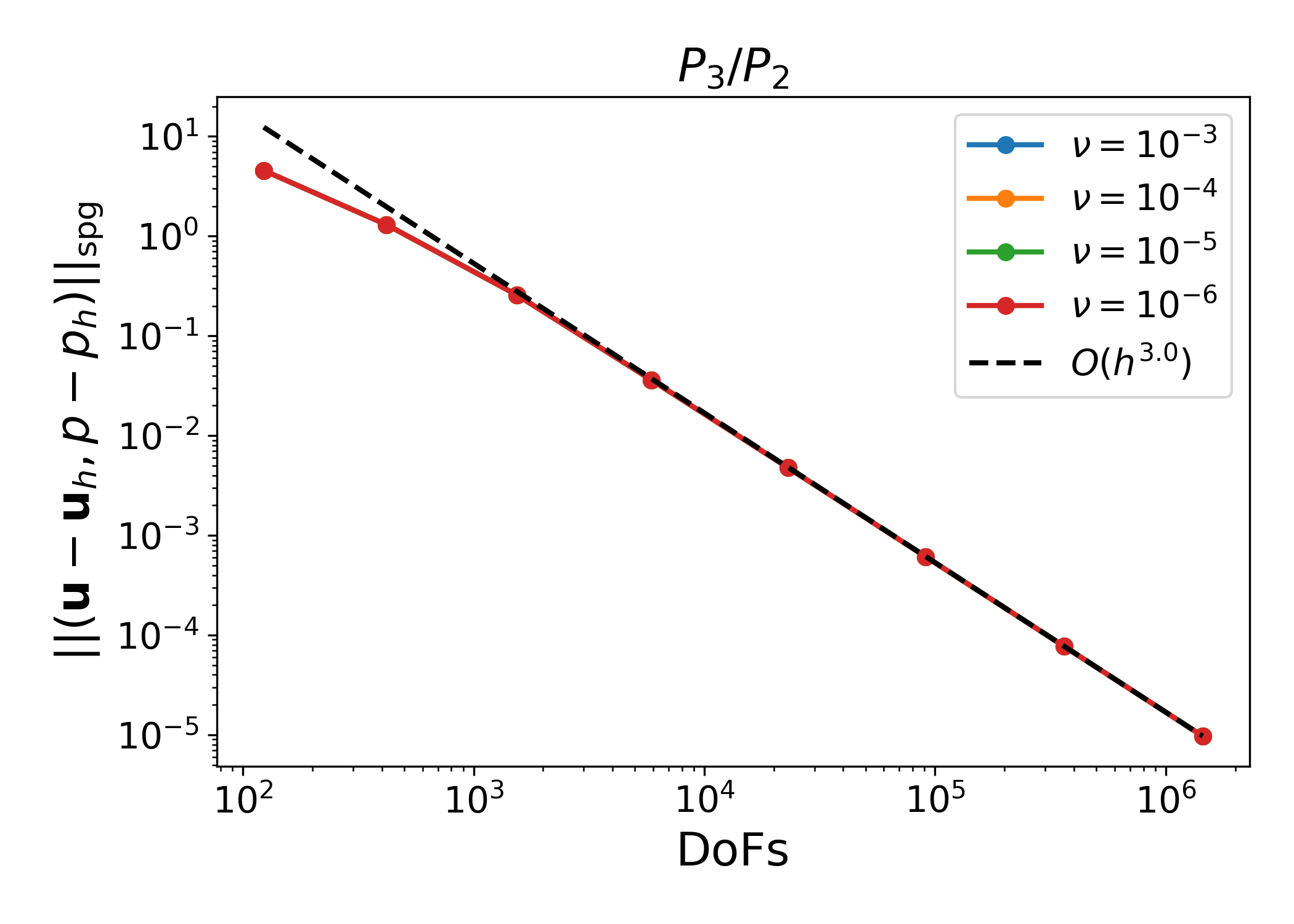}
\hspace*{0.5em}
\includegraphics[width=0.32\textwidth]{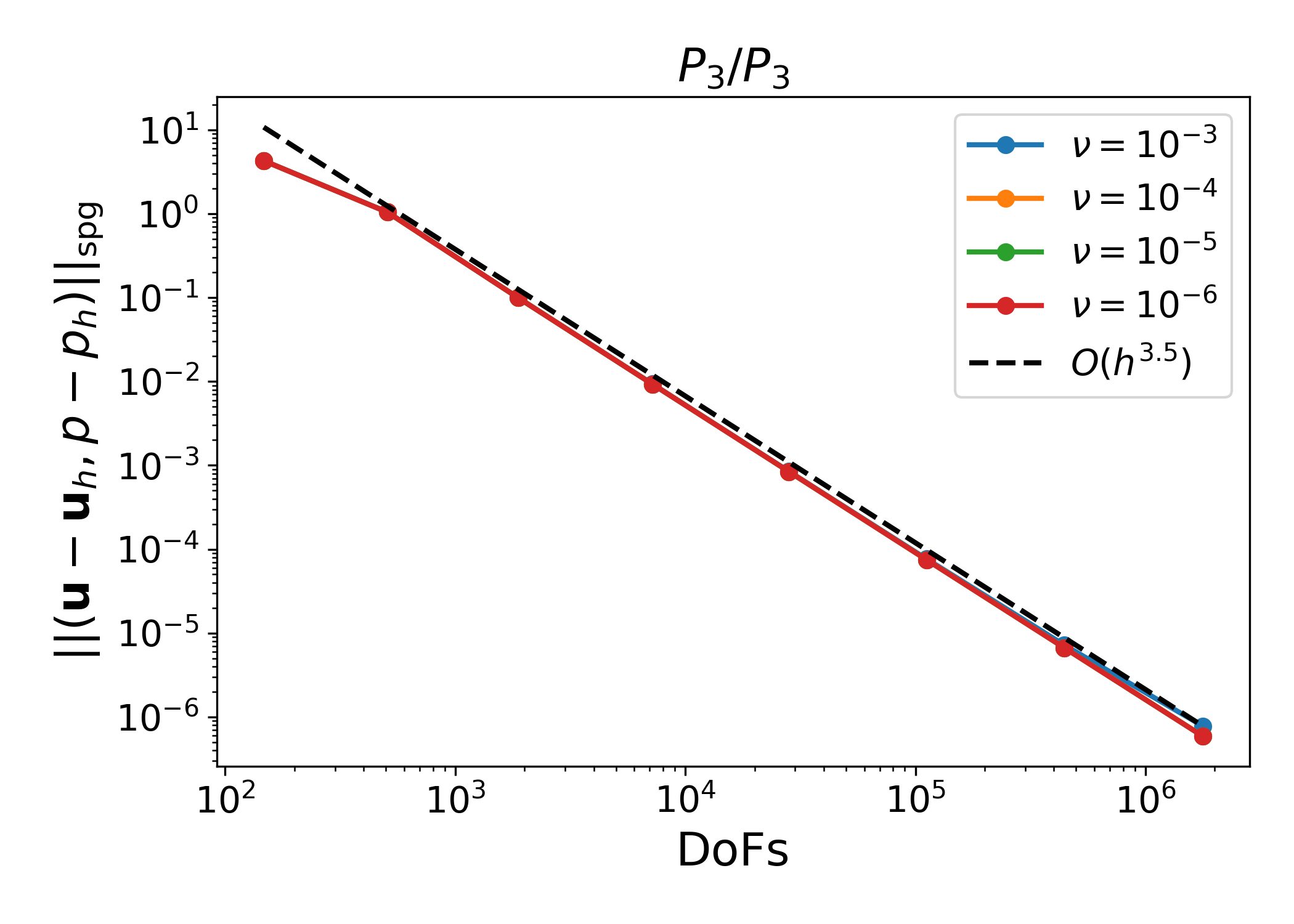}}
  \caption{Example~\ref{sec:smooth}. Errors $\no{\(\bu-\bu_h,p-p_h\)}{\mathrm{spg}}$ 
  for different pairs of finite element spaces and different values of the viscosity coefficient.}
  \label{fig:smooth-all-error}
\end{figure}


\begin{figure}[t!]
\centerline{\includegraphics[width=0.32\textwidth]{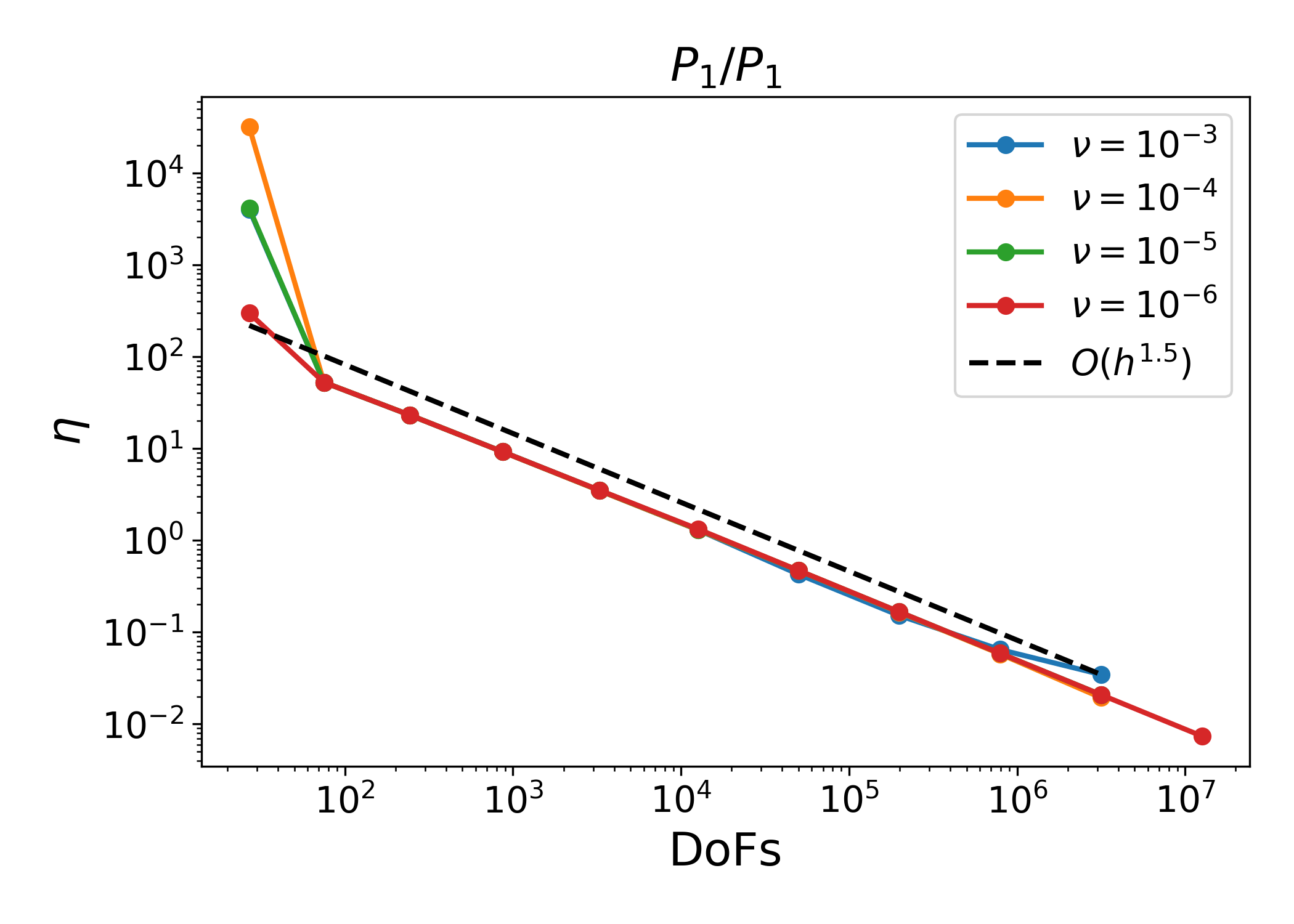}
\hspace*{0.5em}
\includegraphics[width=0.32\textwidth]{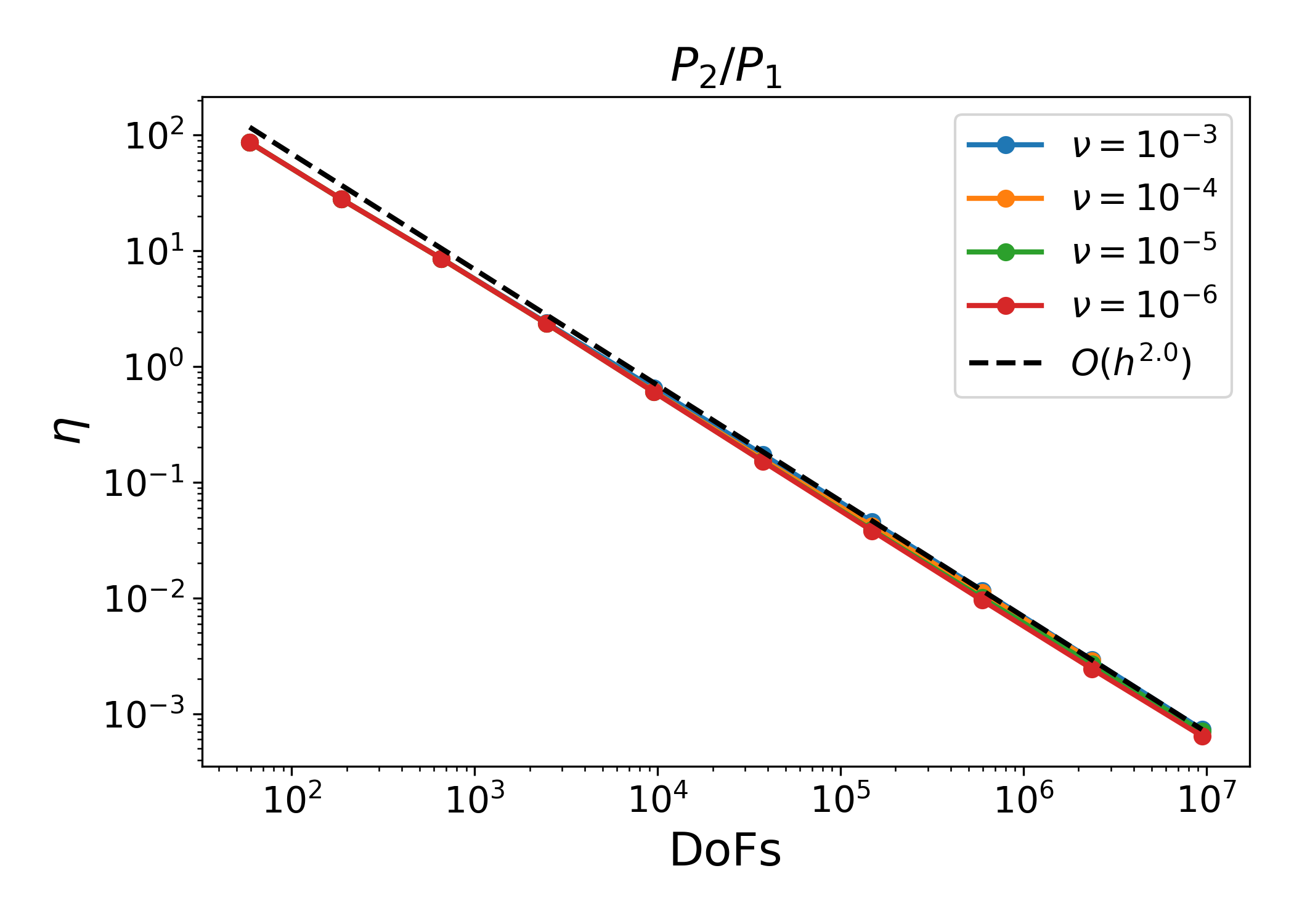}
\hspace*{0.5em}\includegraphics[width=0.32\textwidth]{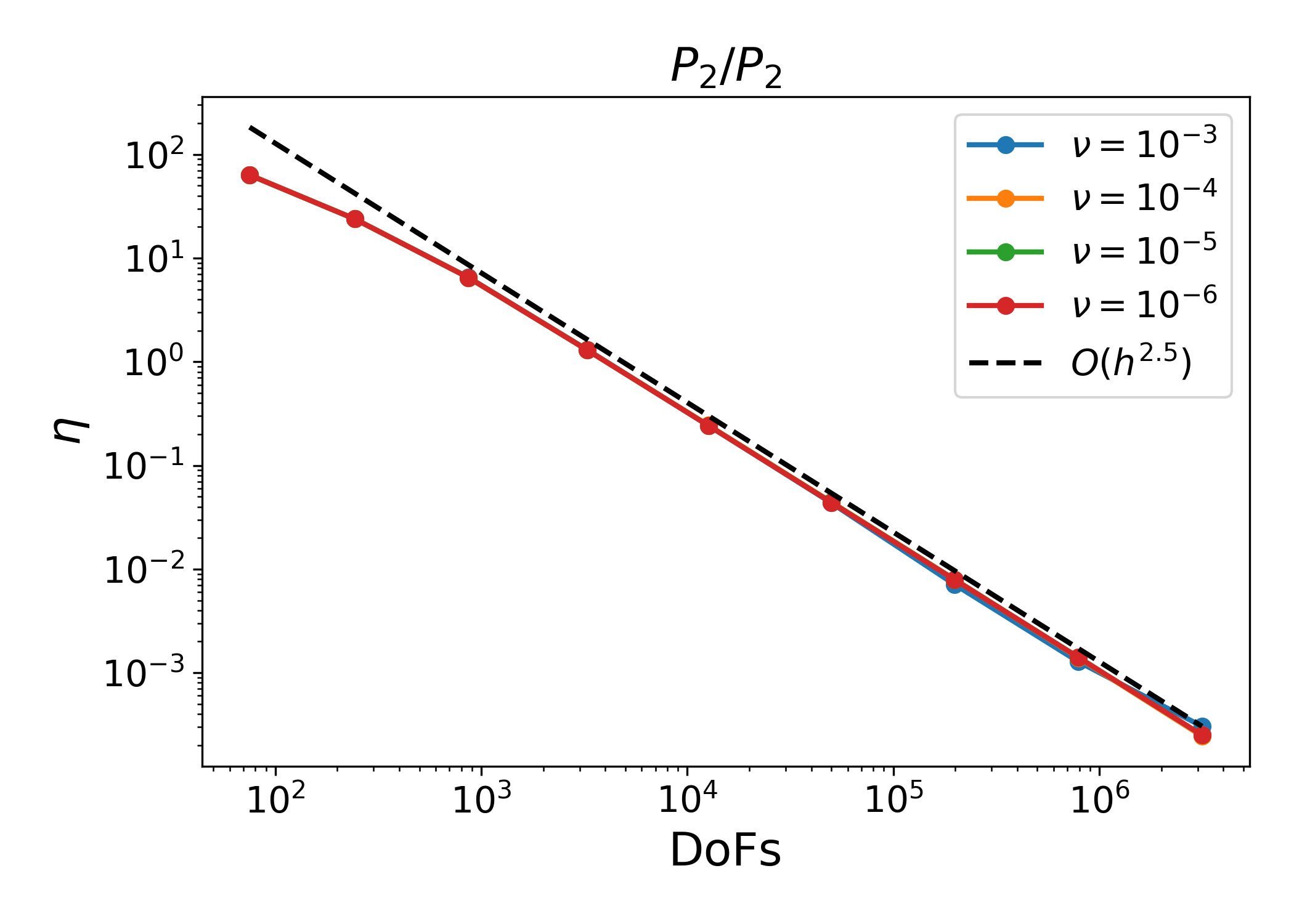}}
\centerline{\includegraphics[width=0.32\textwidth]{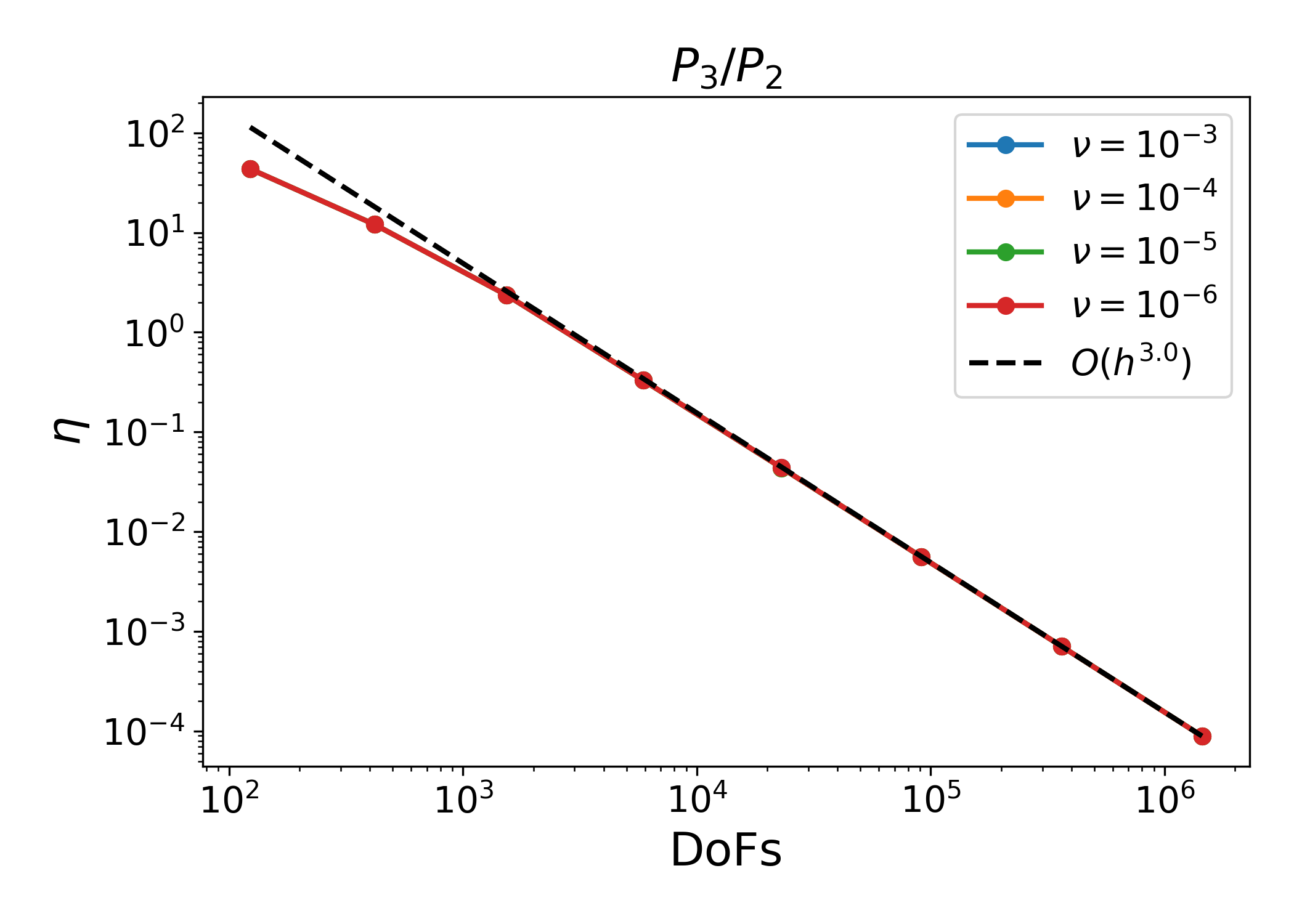}
\hspace*{0.5em}
\includegraphics[width=0.32\textwidth]{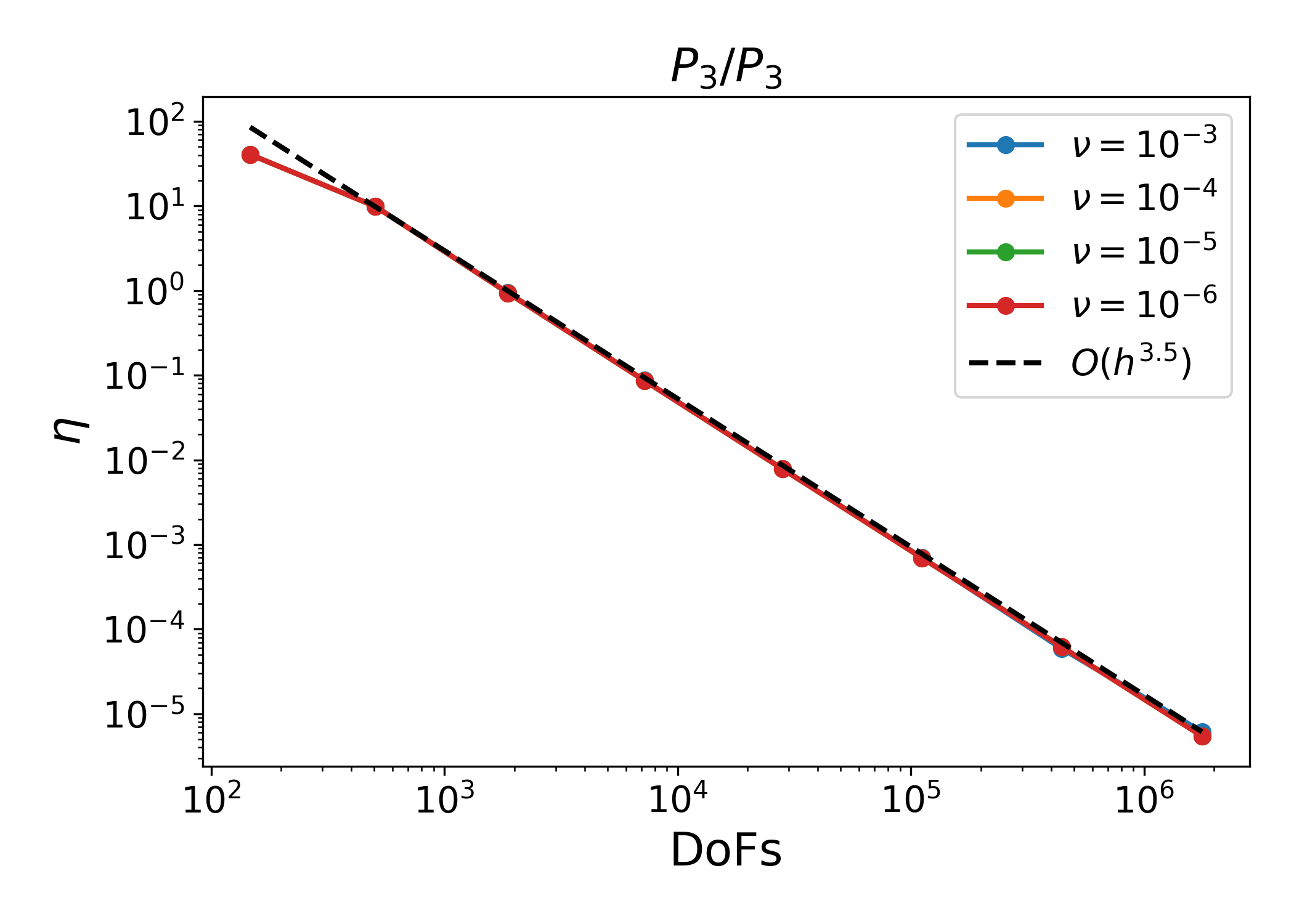}}
  \caption{Example~\ref{sec:smooth}. The a posteriori error estimator $\eta$ 
  for different pairs of finite element spaces and different values of the viscosity coefficient.}
  \label{fig:smooth-all-estimate}
\end{figure}

\begin{figure}[t!]
\centerline{
  \includegraphics[width=0.32\textwidth]{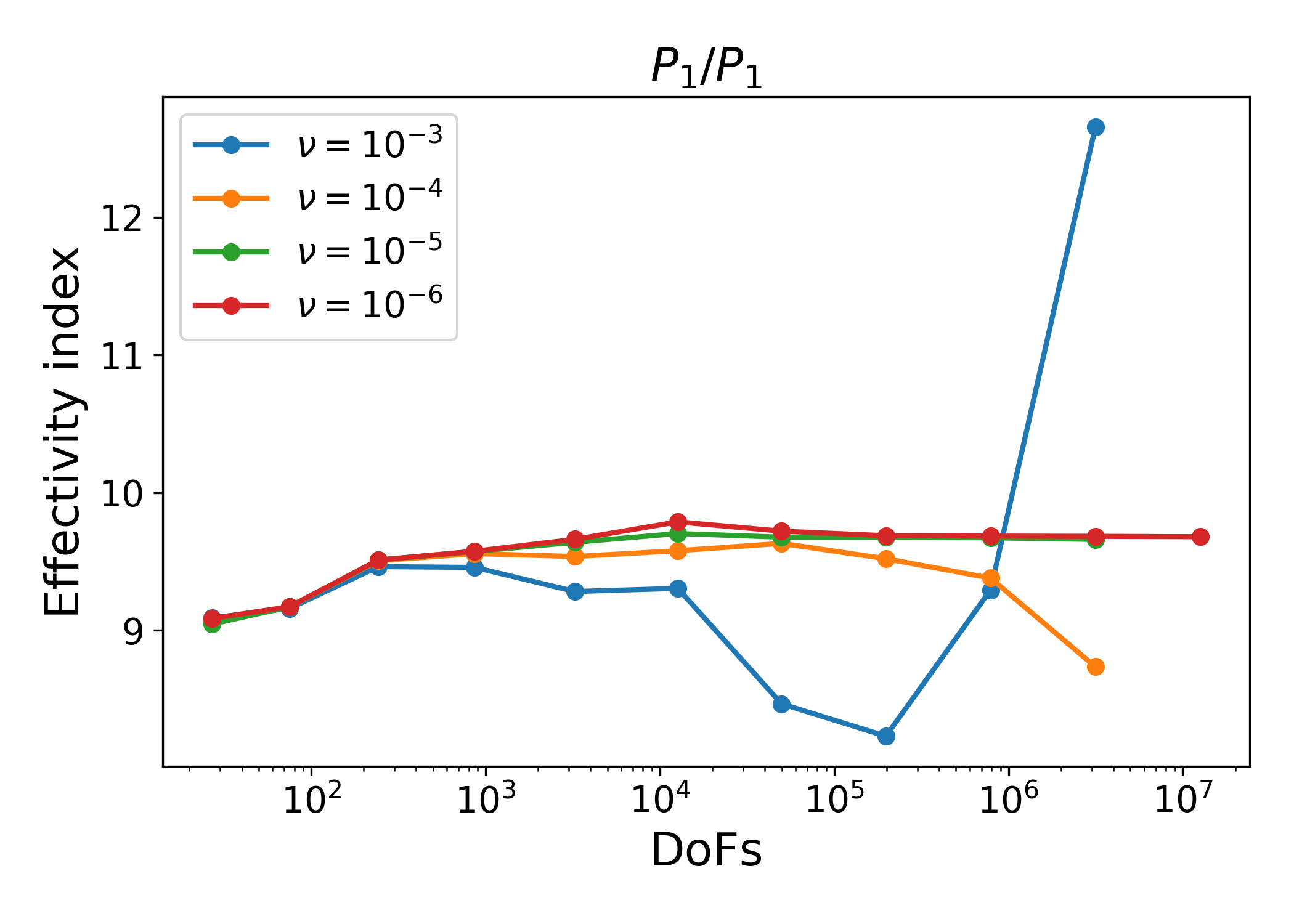}
  \hspace*{0.5em}
  \includegraphics[width=0.32\textwidth]{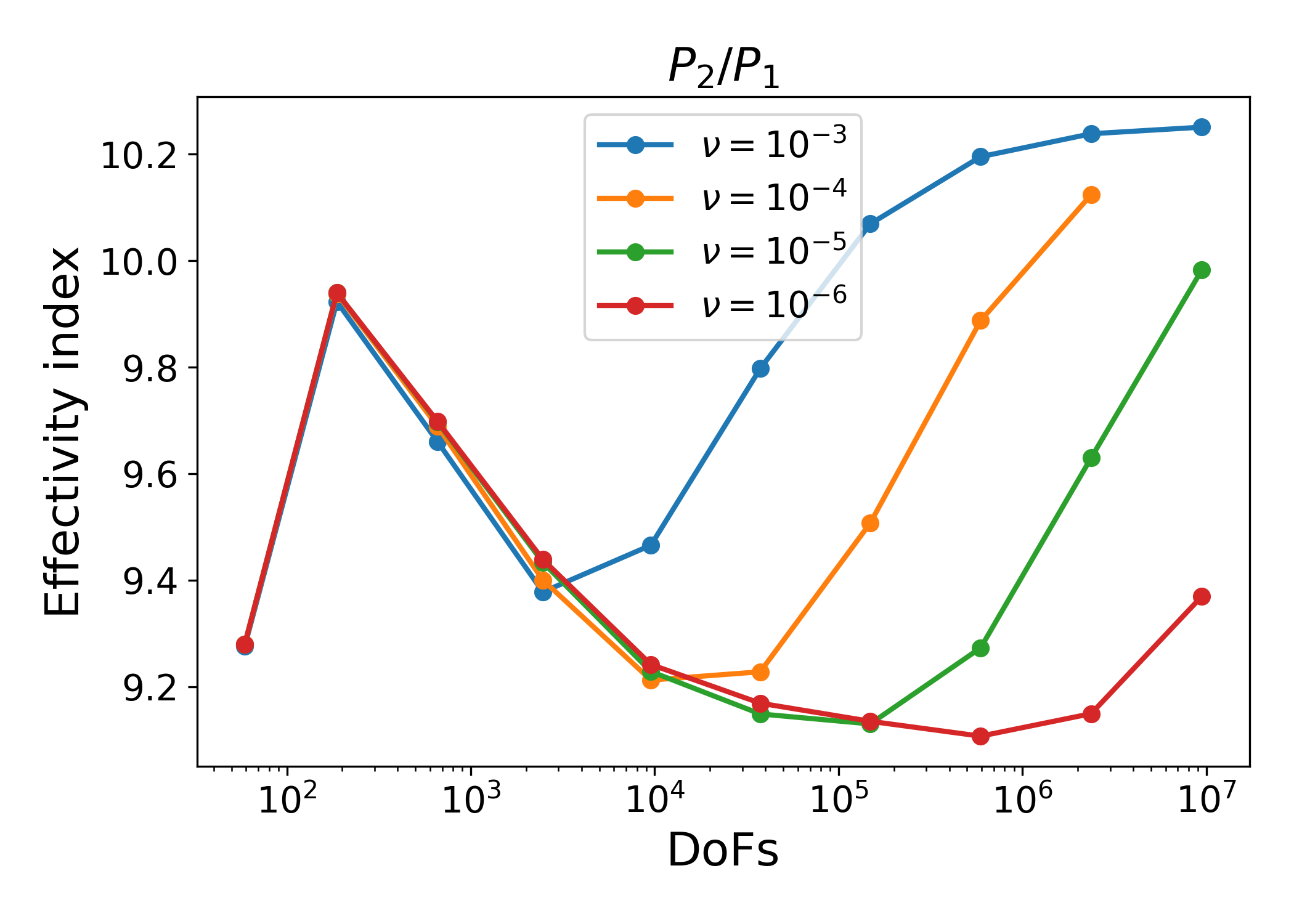}
  \hspace*{0.5em}
  \includegraphics[width=0.32\textwidth]{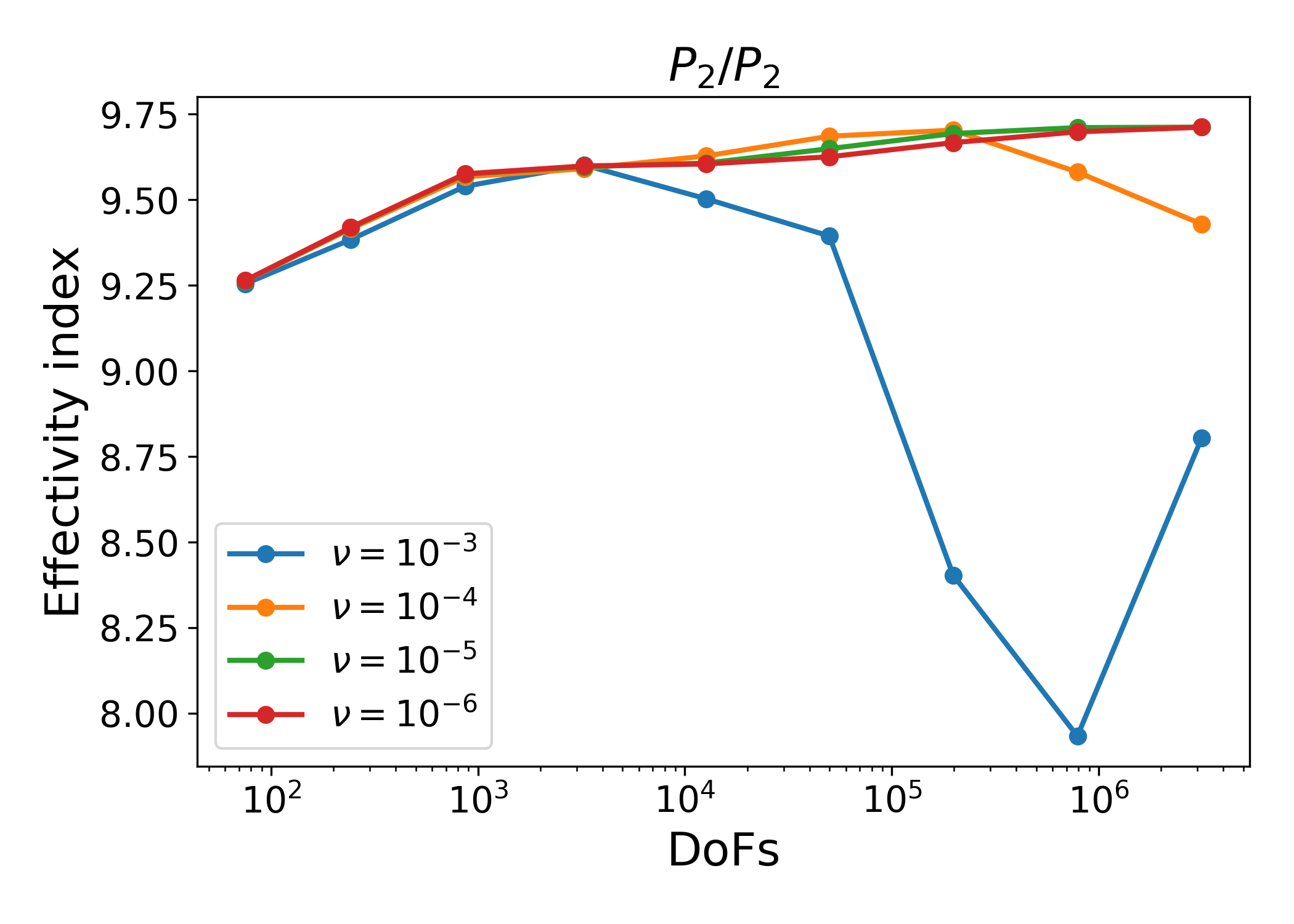}}
\centerline{\includegraphics[width=0.32\textwidth]{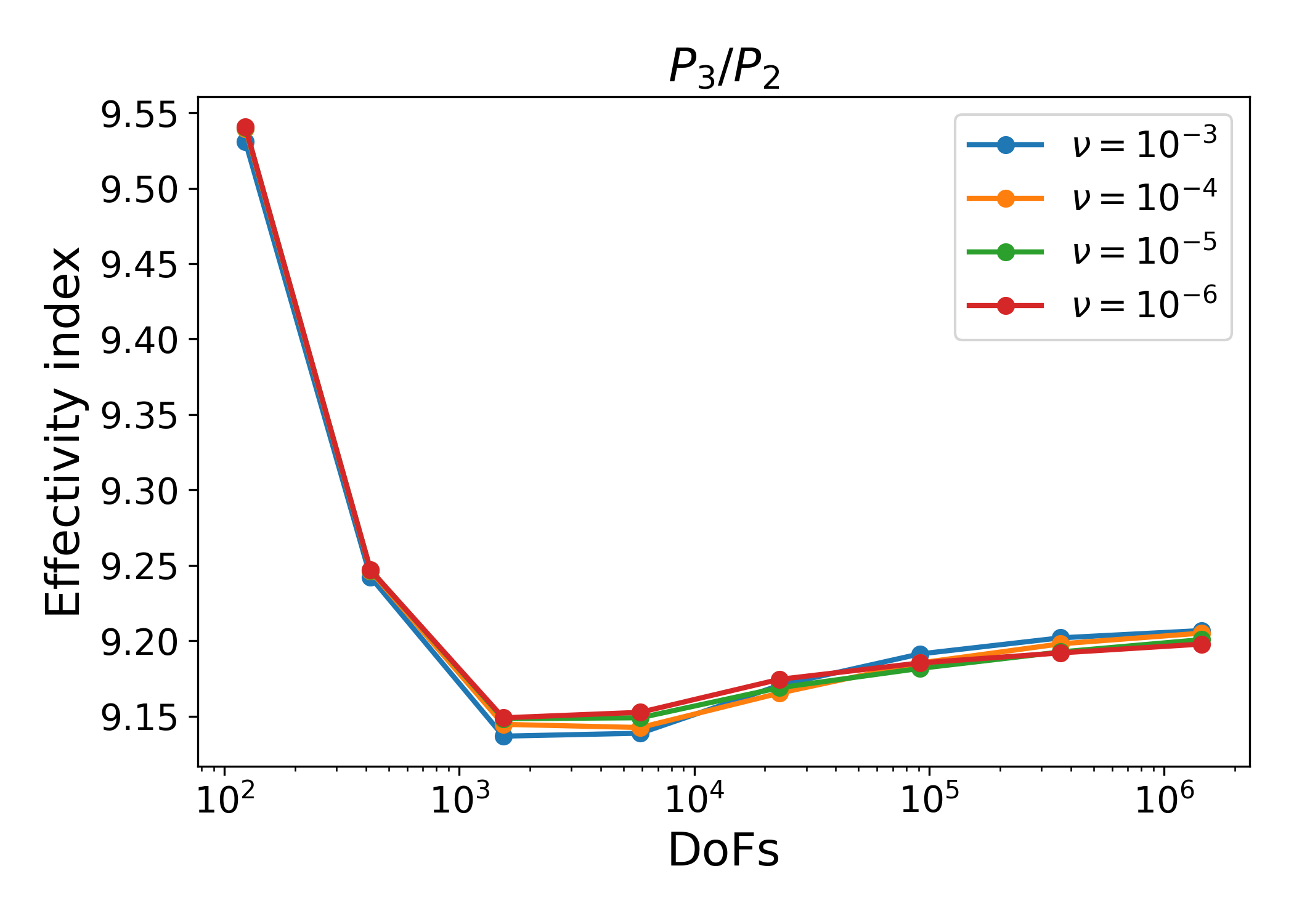}
\hspace*{0.5em}
  \includegraphics[width=0.32\textwidth]{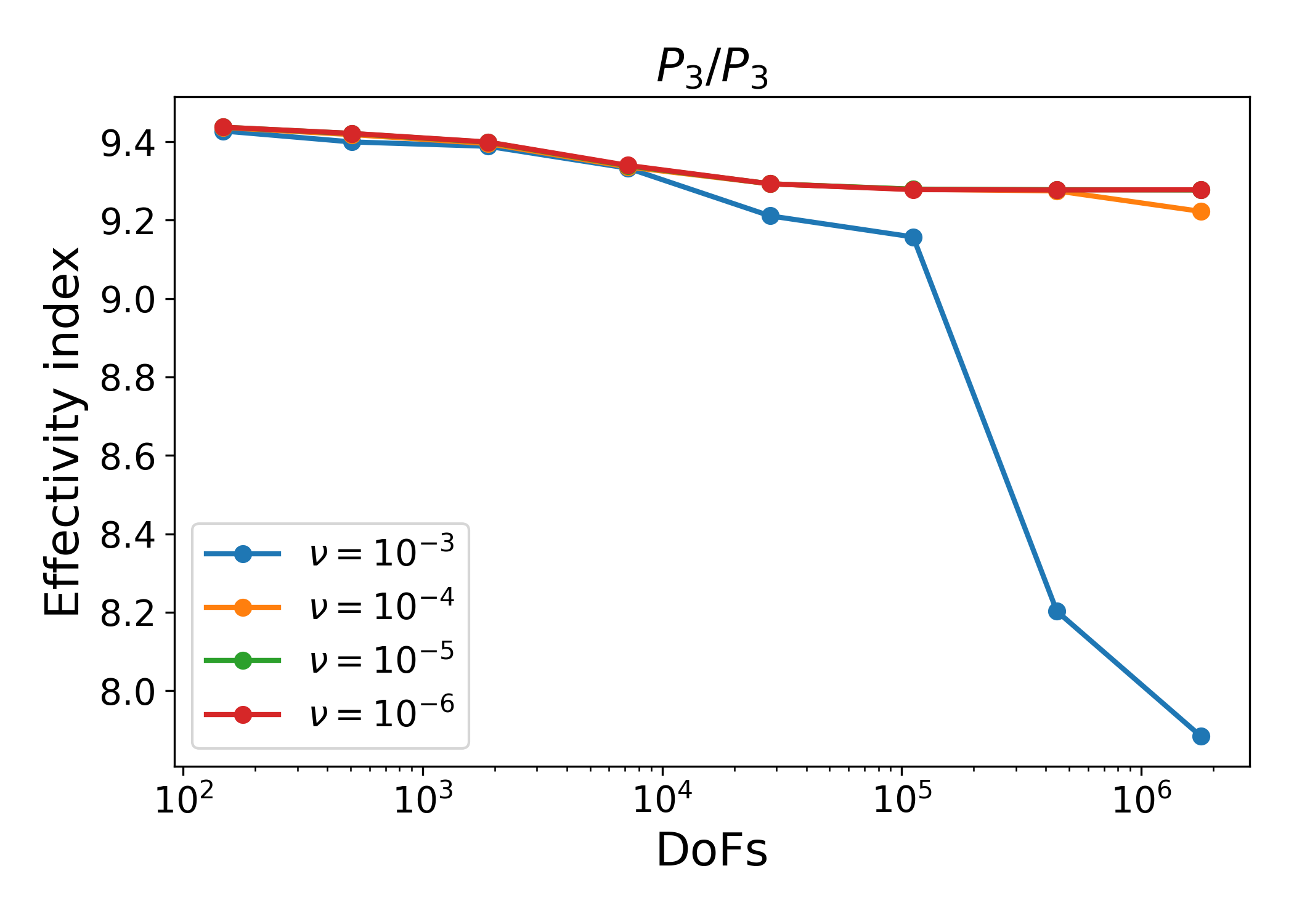}}
  \caption{Example~\ref{sec:smooth}. Effectivity indices for different pairs of finite element spaces and different values of the viscosity coefficient.}
  \label{fig:smooth-all-effectivity}
\end{figure}

Figure~\ref{fig:smooth-all-estimate} shows that the error estimator behaves in all cases
qualitatively in the same way as the error. The quantitative behavior is given by the 
effectivity index, which is the ratio of the error estimate and the error. Figure~\ref{fig:smooth-all-effectivity}
reveals that this index takes in most situations values of around 9 and in the other situations values in $[4,11]$.
That means, the estimator always
overestimates the error, often by a factor of around 9, independently of the pair of 
finite element spaces and the value of the viscosity coefficient. The latter property 
supports the analysis on the robustness of the proposed estimator. Notice that the 
numerical studies for scalar convection-diffusion problems in \cite{JN13} show a similar 
overestimation (factor 5-10). 

\begin{figure}[t!]
\centerline{\includegraphics[width=0.32\textwidth]{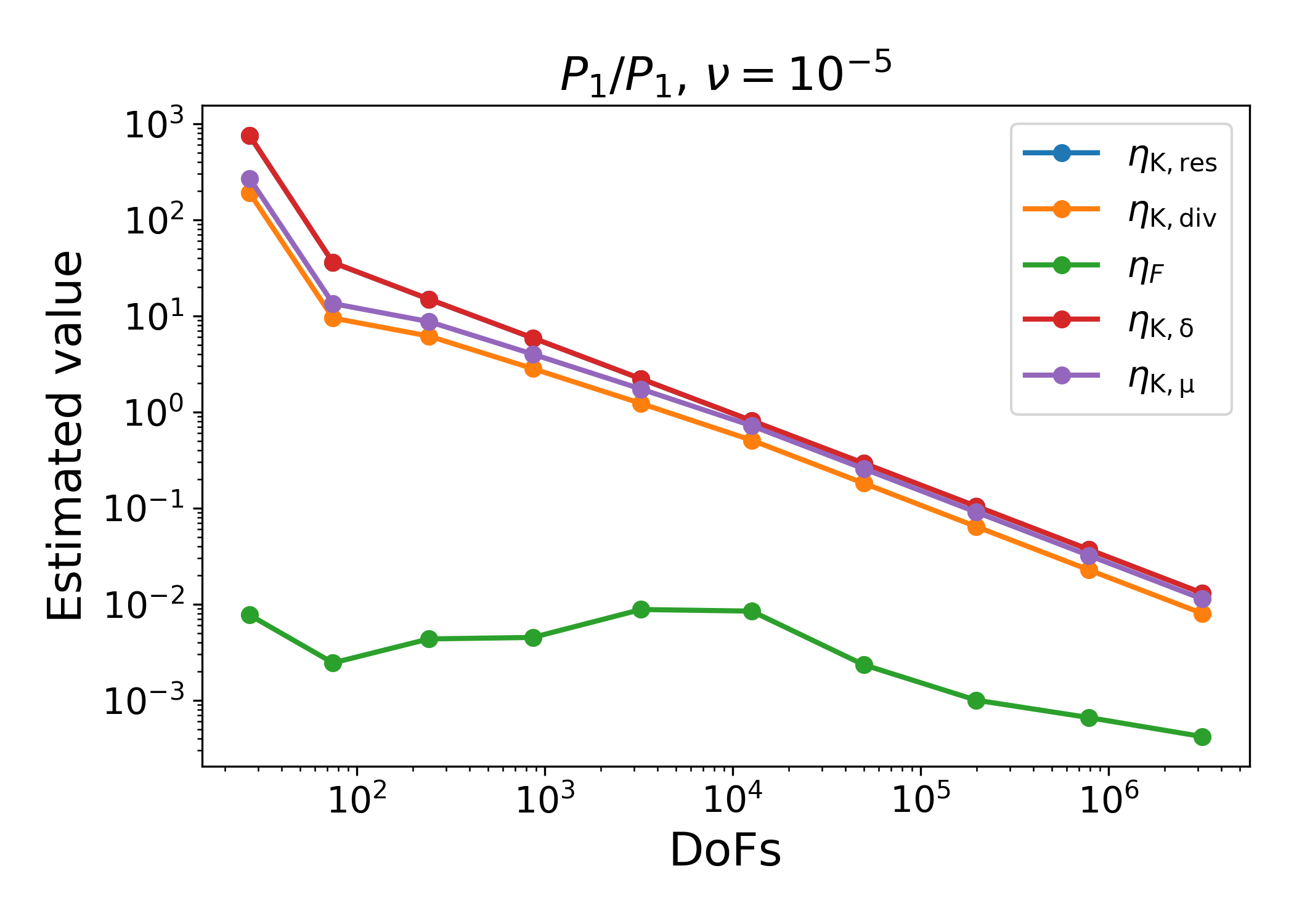}
\hspace*{0.5em}
\includegraphics[width=0.32\textwidth]{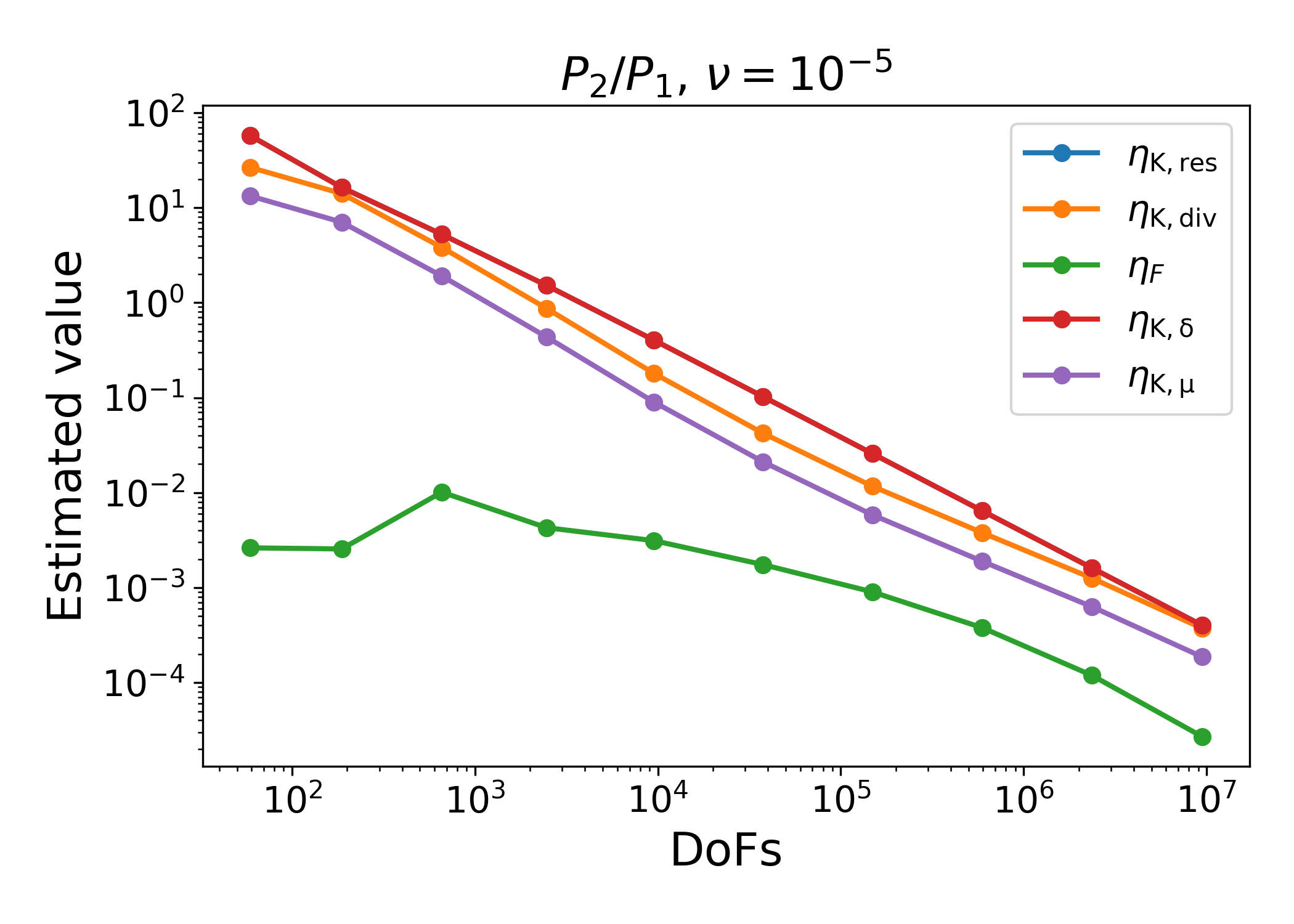}
\hspace*{0.5em}
\includegraphics[width=0.32\textwidth]{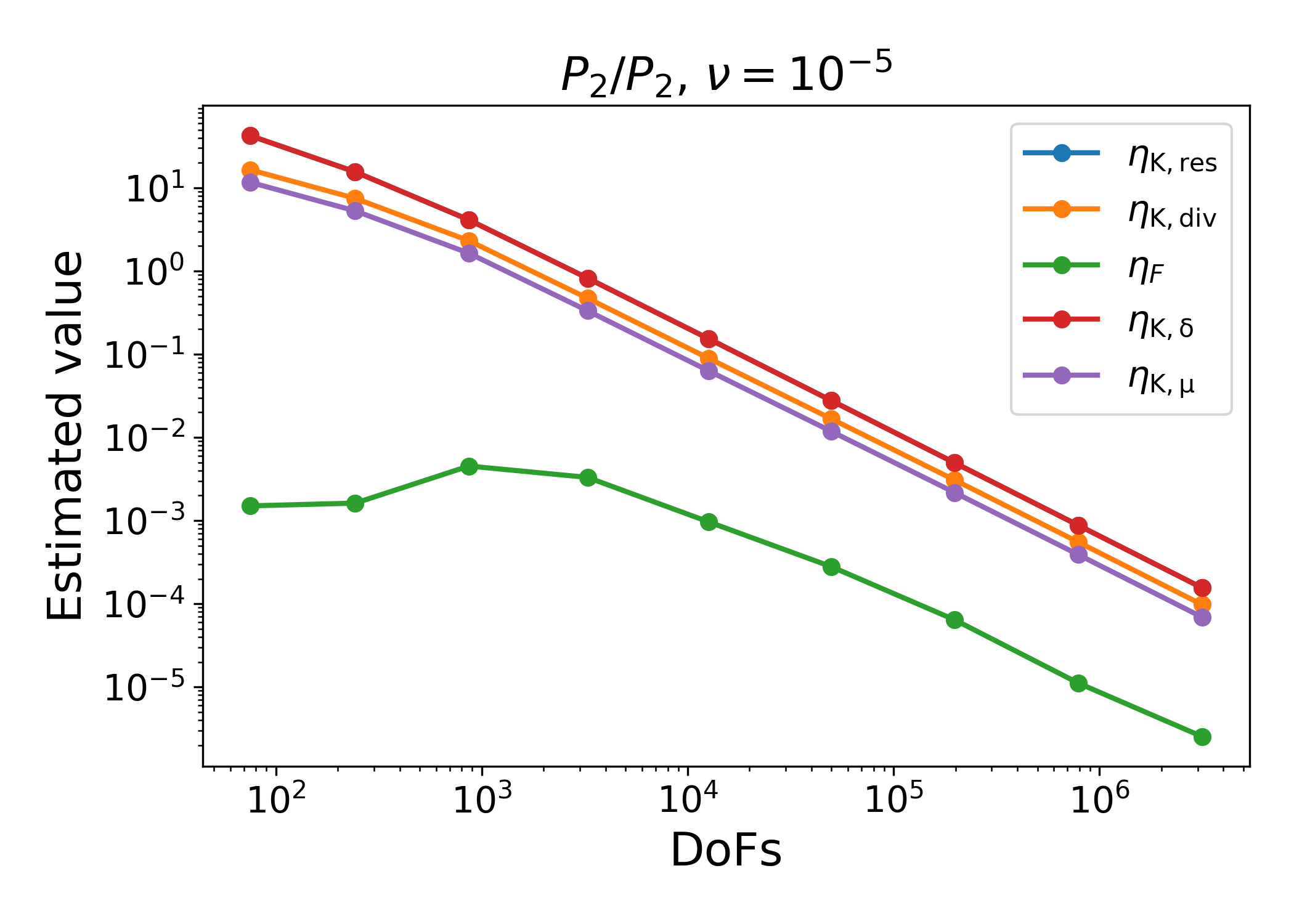}}
\centerline{\includegraphics[width=0.32\textwidth]{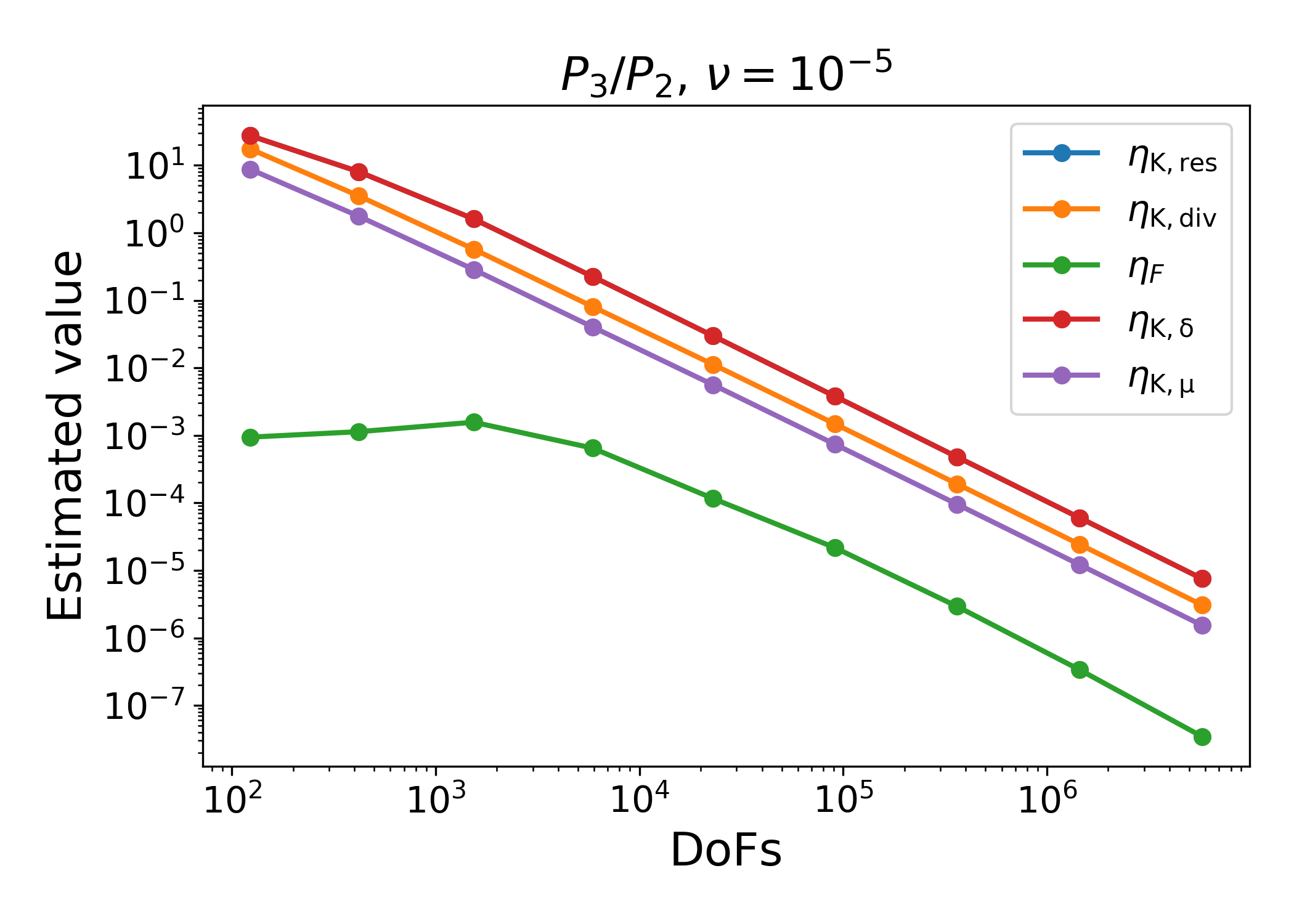}
\hspace*{0.5em}
\includegraphics[width=0.32\textwidth]{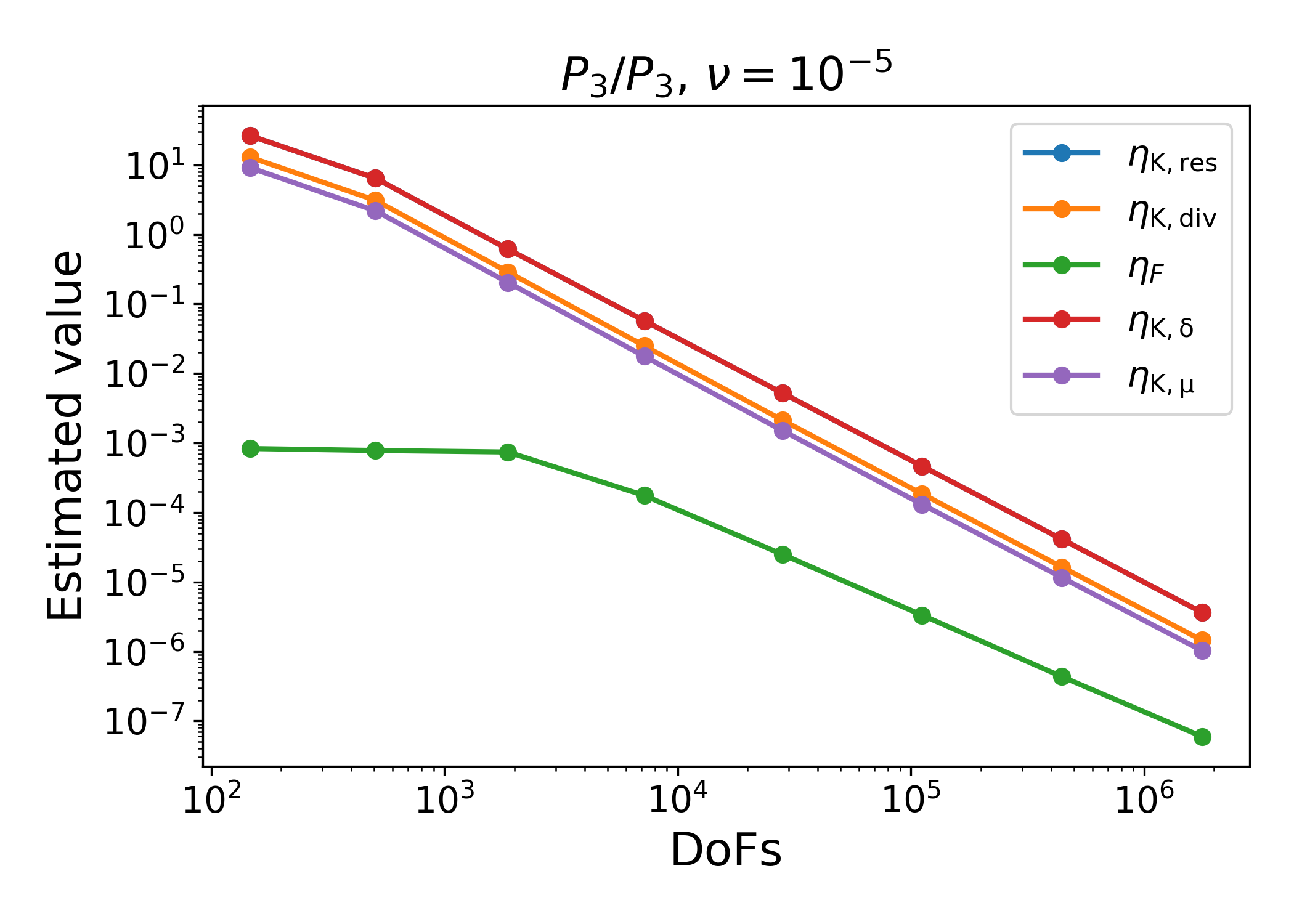}}
  \caption{Example~\ref{sec:smooth}. Contributions of the individual parts of $\eta$, simulations with  $\nu=10^{-5}$.}
  \label{fig:smooth-all-indicators}
\end{figure}

An investigation of the contributions of the individual parts of the error estimator $\eta$ from 
\eqref{eq:eta} is presented in Figure~\ref{fig:smooth-all-indicators}. In this figure, 
$\eta_{K,\ast}$ stands for $(\sum_{K\in\mathcal T_h} \eta_{K,\ast}^2)^{1/2}$. It can be 
seen that $\eta_{K,\mathrm{res}}$, $\eta_{K,\mathrm{div}}$, $\eta_{K,\delta}$, and  $\eta_{K,\mu}$
are often of similar size and that $\eta_F$ is always notably smaller. In particular, the 
contribution  $\eta_{K,\mu}$ from the grad-div stabilization is not negligible. Concerning 
$\eta_F$, the jump of the pressure always vanishes in the residuals at the internal facets defined 
in \eqref{eq:rF}, since in all cases continuous pressure approximations were used, so that 
the term $\br_F (\bu_h,p_h)$, and with that $\eta_F$, contains the small factor $\nu$ coming 
from the velocity contribution in the jumps. 

\begin{figure}[t!]
\centerline{\includegraphics[width=0.32\textwidth]{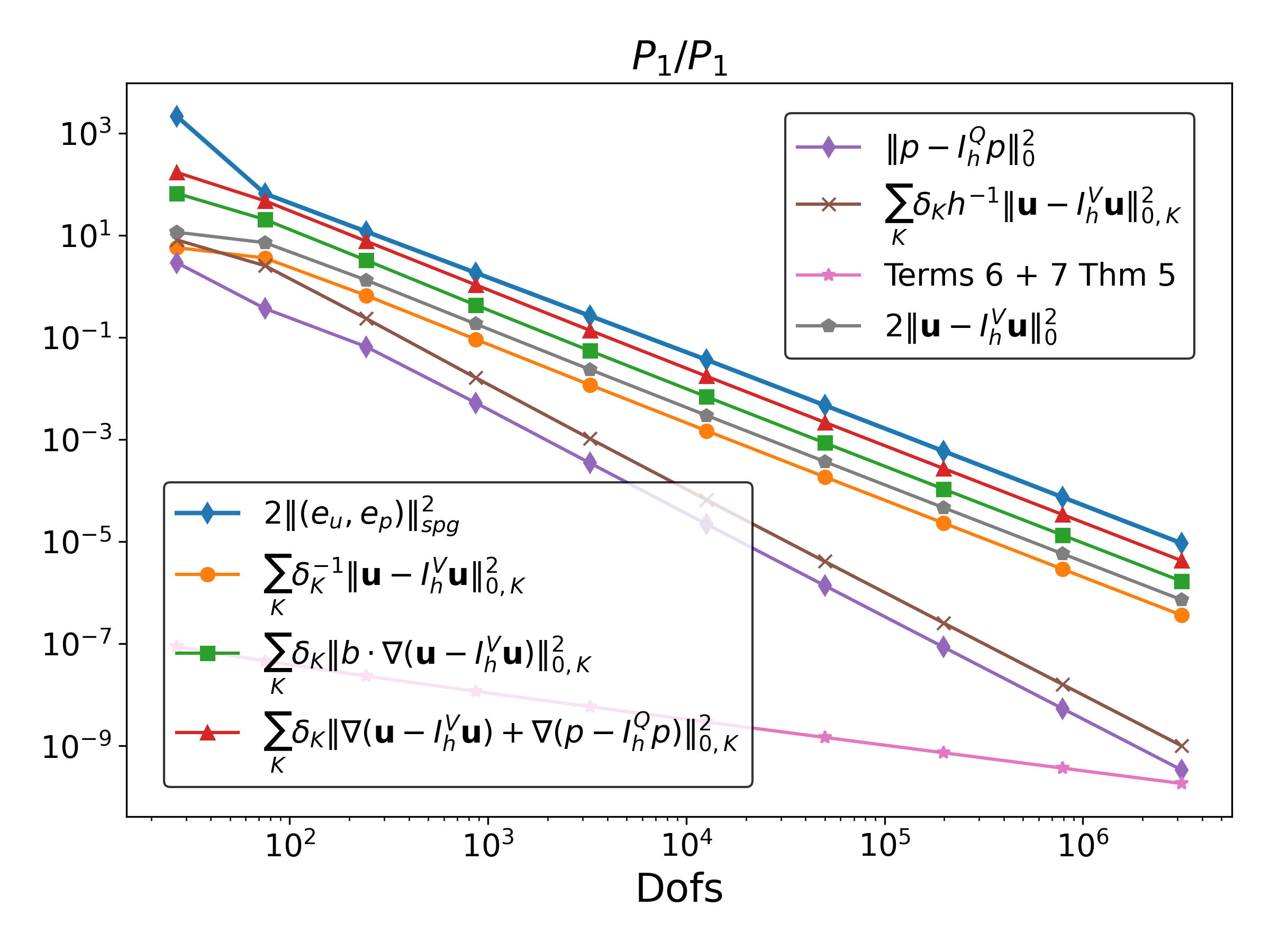}
\hspace*{0.5em}
\includegraphics[width=0.32\textwidth]{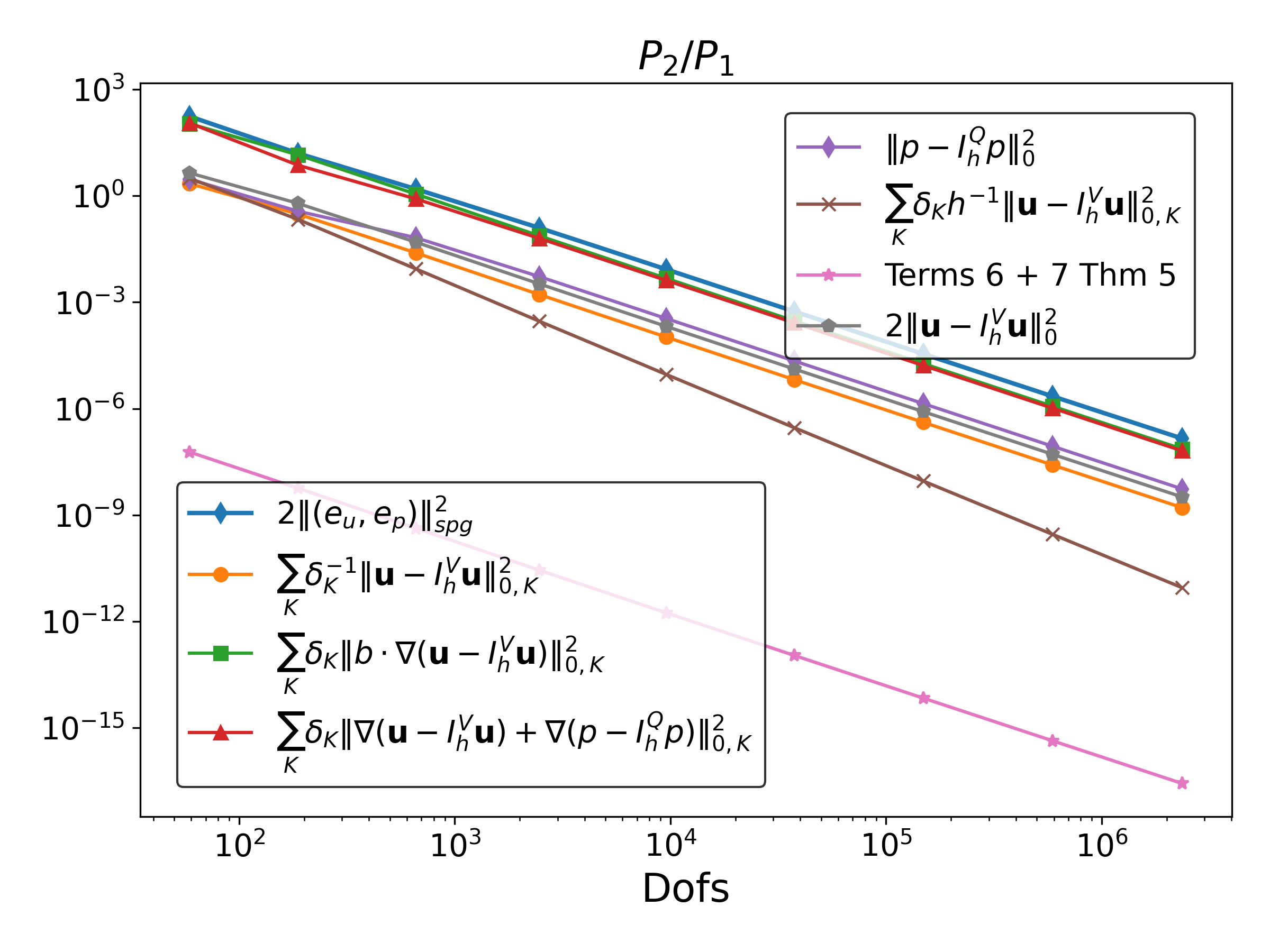}
\hspace*{0.5em}
\includegraphics[width=0.32\textwidth]{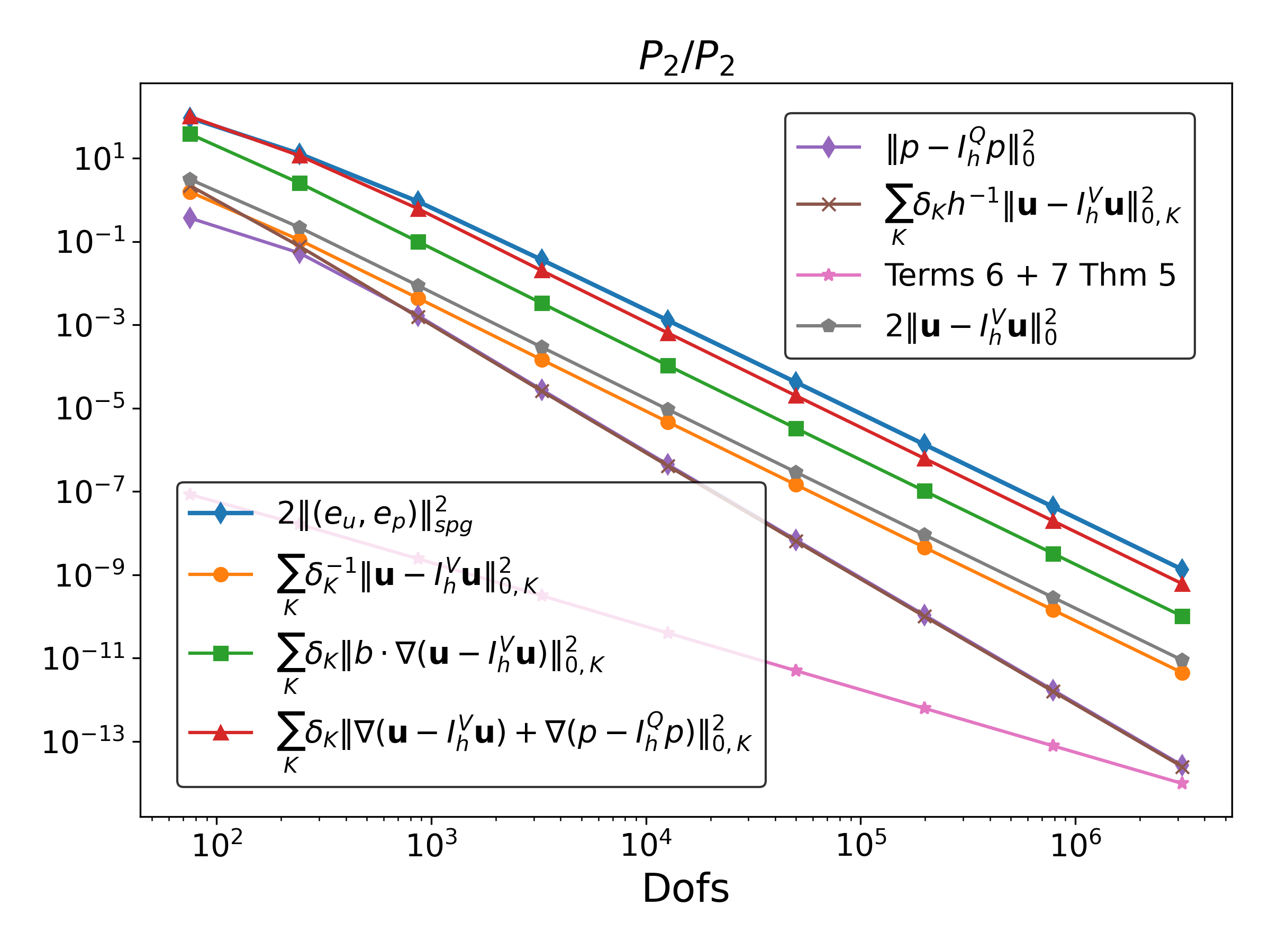}}
\centerline{\includegraphics[width=0.32\textwidth]{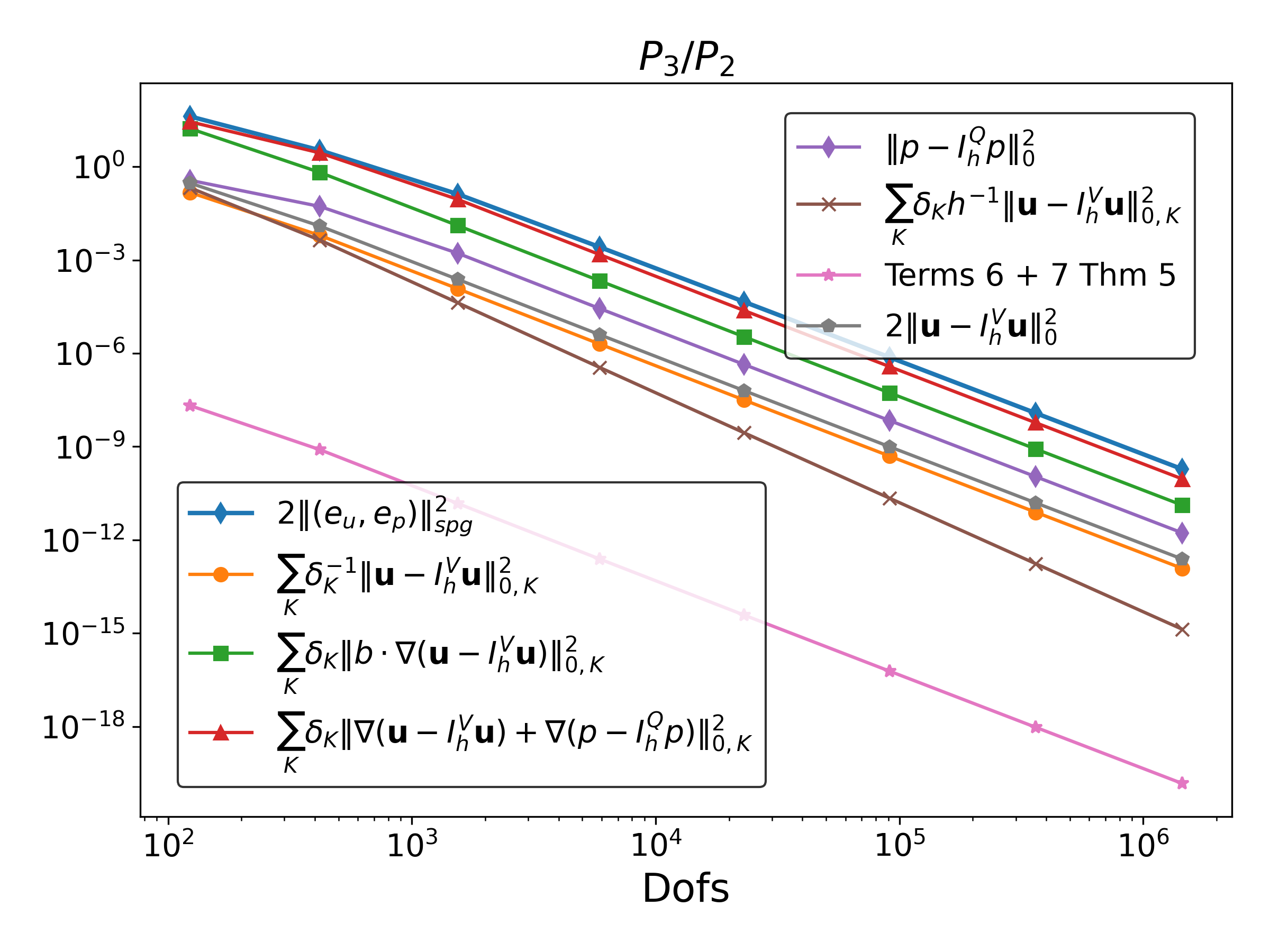}
\hspace*{0.5em}
\includegraphics[width=0.32\textwidth]{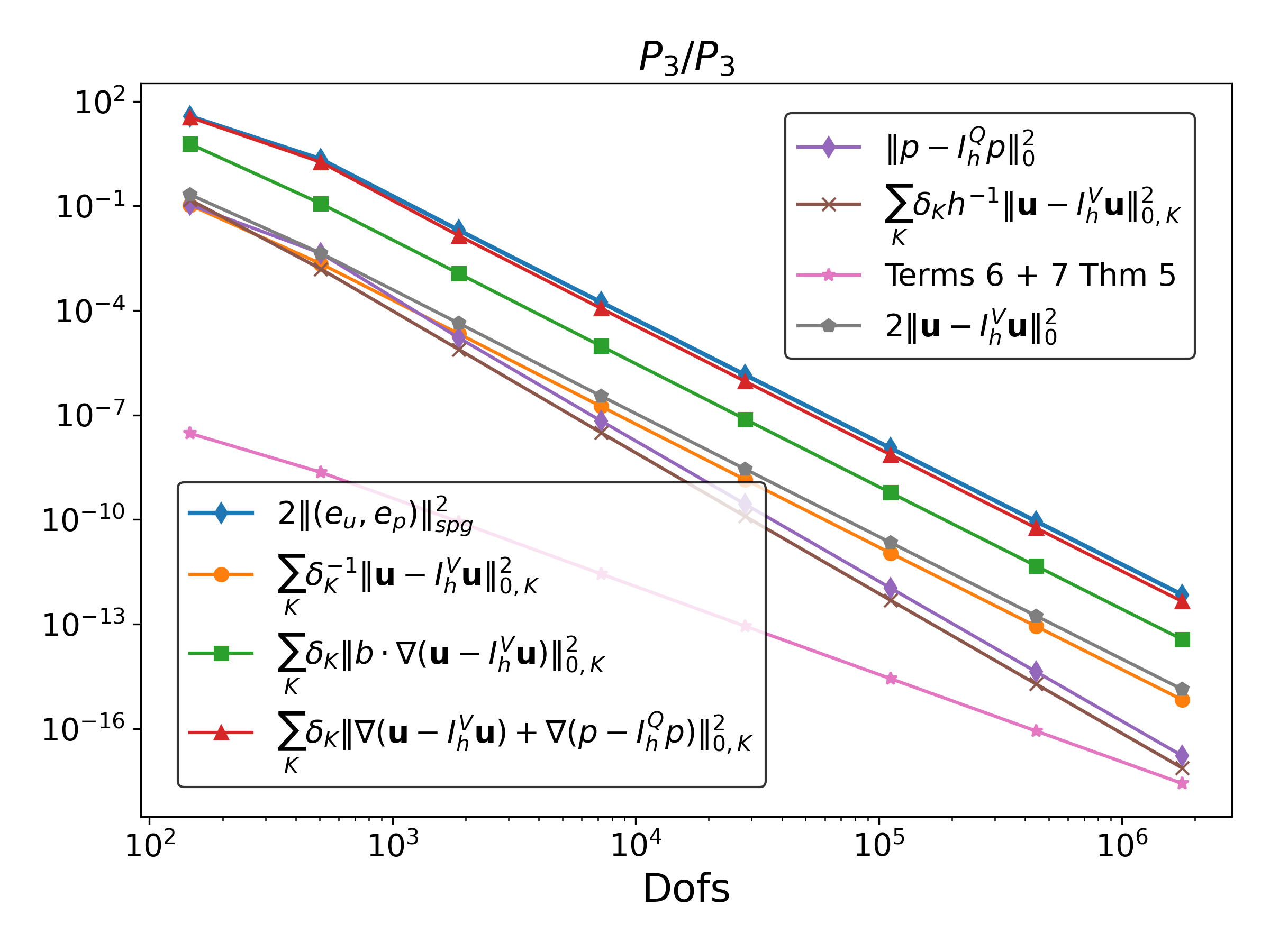}}
  \caption{Example~\ref{sec:smooth}. Comparison of $2 \no{\(\bu-\bu_h,p-p_h\)}{\mathrm{spg}}^2$ (blue line) and of $2 \|\bu-\Ihv\bu\|_{L^2(\Omega)}^2$ (gray line) with  terms that appear in the 
  hypotheses from Hypothesis~\ref{hypo:errors} as well as with the sum of the  two last terms from the error bound \eqref{eq:est},  simulations with  $\nu=10^{-6}$.}
  \label{fig:smooth-hypos}
\end{figure}

The derivation of the estimator $\eta$ from \eqref{eq:eta} is based on a couple of hypotheses 
with respect to interpolation errors. It is argued in Hypothesis~\ref{hypo:errors} that, 
on the one hand, these hypotheses are often reasonable but, on the other hand, there are cases 
where they might not be satisfied. Figure~\ref{fig:smooth-hypos} presents for the considered 
example a comparison between the terms on the left-hand sides of the hypotheses and the 
terms on the right-hand sides, which are $2 \no{\(\bu-\bu_h,p-p_h\)}{\mathrm{spg}}^2$ 
and $2 \|\bu-\Ihv\bu\|_{L^2(\Omega)}^2$. It can be seen that for this example all hypotheses 
are satisfied. In addition, there are two terms in the error bound, the last two terms on the 
right-hand side of \eqref{eq:est}, which are argued to be negligible in practice, 
compare Remark~\ref{rem:est_last_two_terms}. In Figure~\ref{fig:smooth-hypos}, the sum of both 
terms is presented and it can be seen that it is in fact much smaller than all the other terms, 
and with that much smaller than the error estimate. 

\subsection{Solution with Boundary Layers}\label{sec:layer}

This example is taken from \cite{KK20}. The solution is given by 
\[
  \bu = \operatorname{curl}\phi = \left(\frac{\partial \phi}{\partial y},\,-\frac{\partial \phi}{\partial x}\right)^T,\quad
  p = \mathrm{e}^{x+y} - p_m,
\]
with the potential function
\[
  \phi(x,y) = x^2 y^2 \left(1 - \mathrm{e}^{\lambda(x-1)}\right)^2 \left(1 - \mathrm{e}^{\lambda(y-1)}\right)^2,
\]
$\lambda = 0.5/\sqrt{\nu}$, and $p_m \in \mathbb{R}$ is chosen such that
$\int_\Omega p\ \mathrm{d}\bx = 0$. By construction, the velocity field possesses boundary
layers at $x=1$ and $y=1$. The convection field for the Oseen problem was chosen to be 
$\bb = (1,1)^T$ and the reaction field to be $\sigma = 0$. Dirichlet boundary conditions were 
prescribed at $\partial\Omega$.

We present results on adaptively refined grids, where the adaptive refinement was controlled 
by the local error estimator 
\[
 \left(\eta_{K,\mathrm{res}}^2 + \eta_{K,\mathrm{div}}^2 + \eta_F^2 + \eta_{K,\delta}^2 + \eta_{K,\mu}^2 \right)^{1/2},
 \quad K \in \mathcal T_h.
\]
The same refinement strategy as used in  \cite{Joh00} was applied, with the same parameters as 
in this paper.  Exemplarily, Figure~\ref{fig:bl-meshes} presents some 
adaptively refined grids.  It can be observed that the error estimator leads correctly to a refinement 
in the region and the neighborhood 
of the boundary layers. 

\begin{figure}[t!]
  \centering
  \subfigure[Level 5, total DoFs $=5130$]{\includegraphics[width=0.40\textwidth]{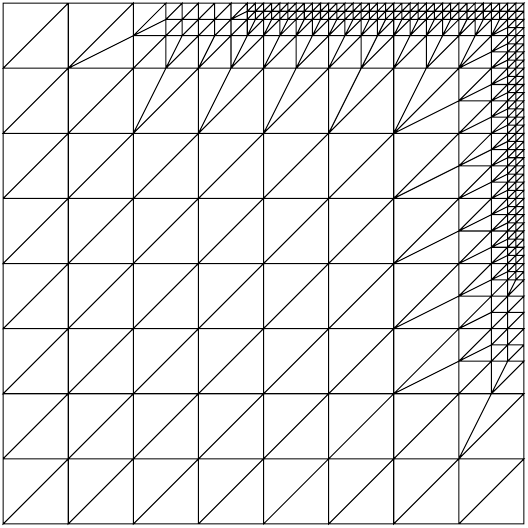}}
  \hspace*{0.5em}
  \subfigure[Level 10, total DoFs $=14217$]{\includegraphics[width=0.40\textwidth]{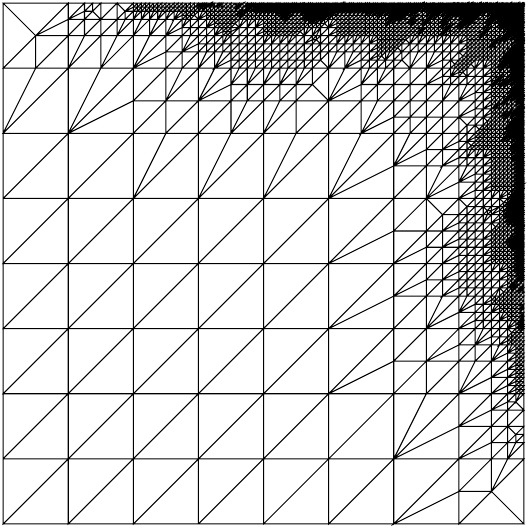}}\\[0.5em]
  \subfigure[Level 15, total DoFs $=114108$]{\includegraphics[width=0.40\textwidth]{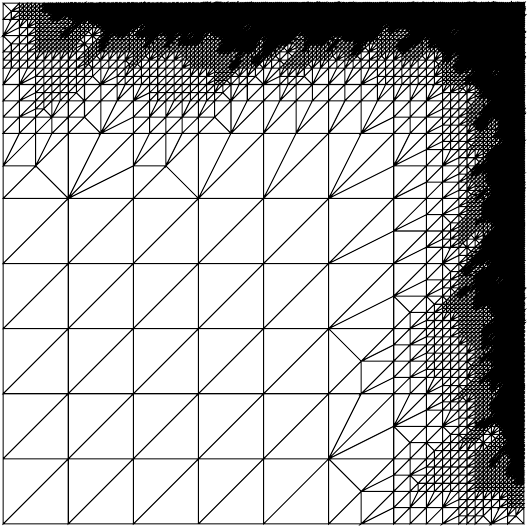}}
  \hspace*{0.5em}
  \subfigure[Level 19, total DoFs $=2748675$]{\includegraphics[width=0.40\textwidth]{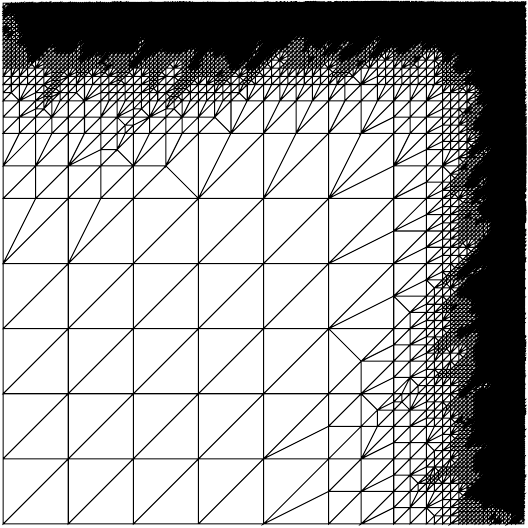}}
  \caption{Example~\ref{sec:layer}. Adaptively refined meshes for the  $P_2/P_1$ simulation with $\nu=10^{-6}$.}
  \label{fig:bl-meshes}
\end{figure}

The a priori error analysis states that $\no{\(\bu-\bu_h,p-p_h\)}{\mathrm{spg}}$ is bounded by 
a constant, which is independent of the viscosity, times a higher order Sobolev norm of the solution. 
Because of the boundary layers, higher order Sobolev norms of the velocity field depend on 
inverse powers of the viscosity coefficient, which is also clear from the analytic representation of $\bu$. 
Consequently, one expects for this example that $\no{\(\bu-\bu_h,p-p_h\)}{\mathrm{spg}}$ increases 
if the viscosity coefficient decreases, which is the reason that the corresponding a priori estimate is sometimes called 
semi-robust. This behavior is clearly visible in Figure~\ref{fig:bl-all-error}.

\begin{figure}[t!]
  \centerline{
  \includegraphics[width=0.32\textwidth]{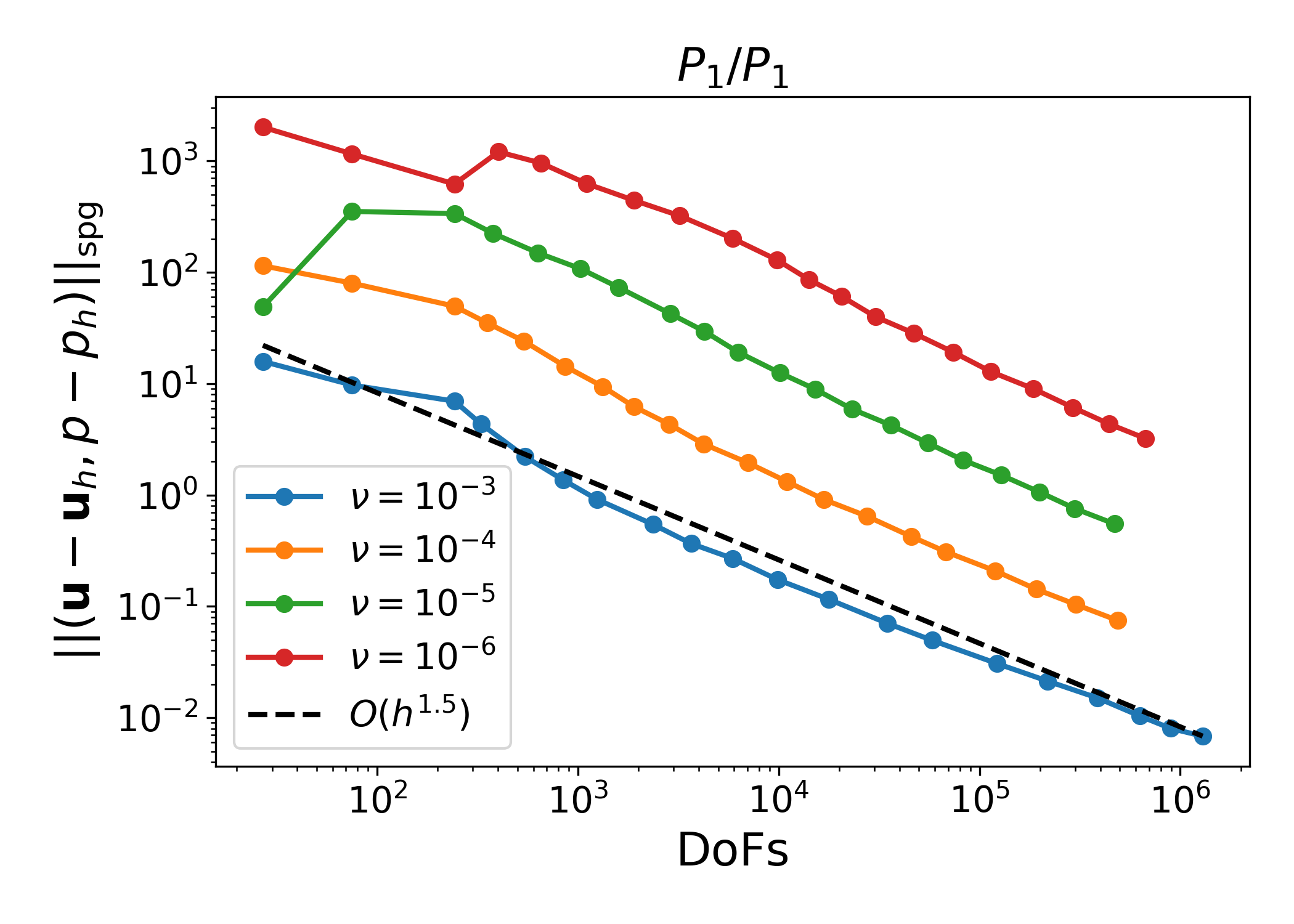}
   \hspace*{0.5em}
  \includegraphics[width=0.32\textwidth]{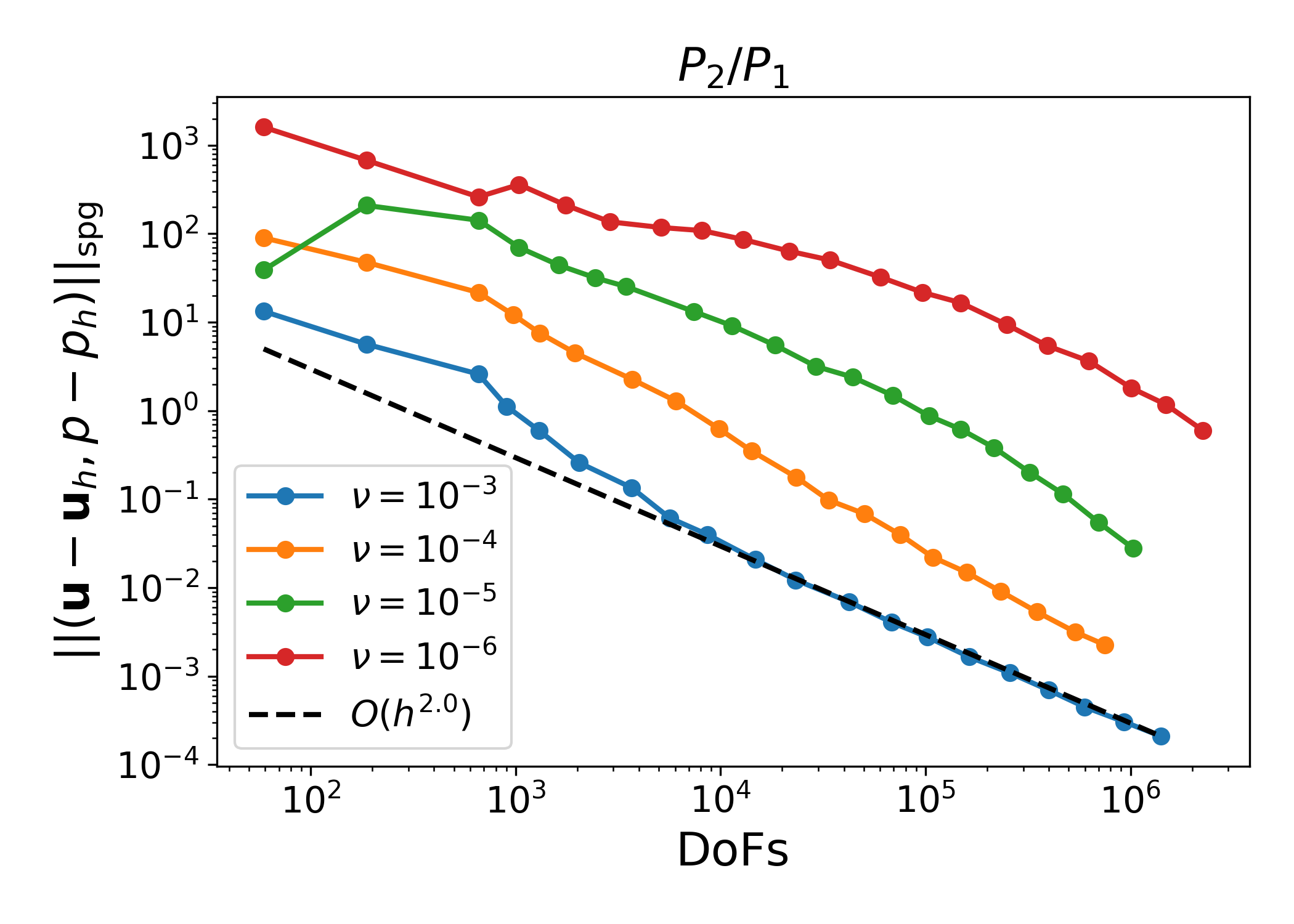}
   \hspace*{0.5em}
  \includegraphics[width=0.32\textwidth]{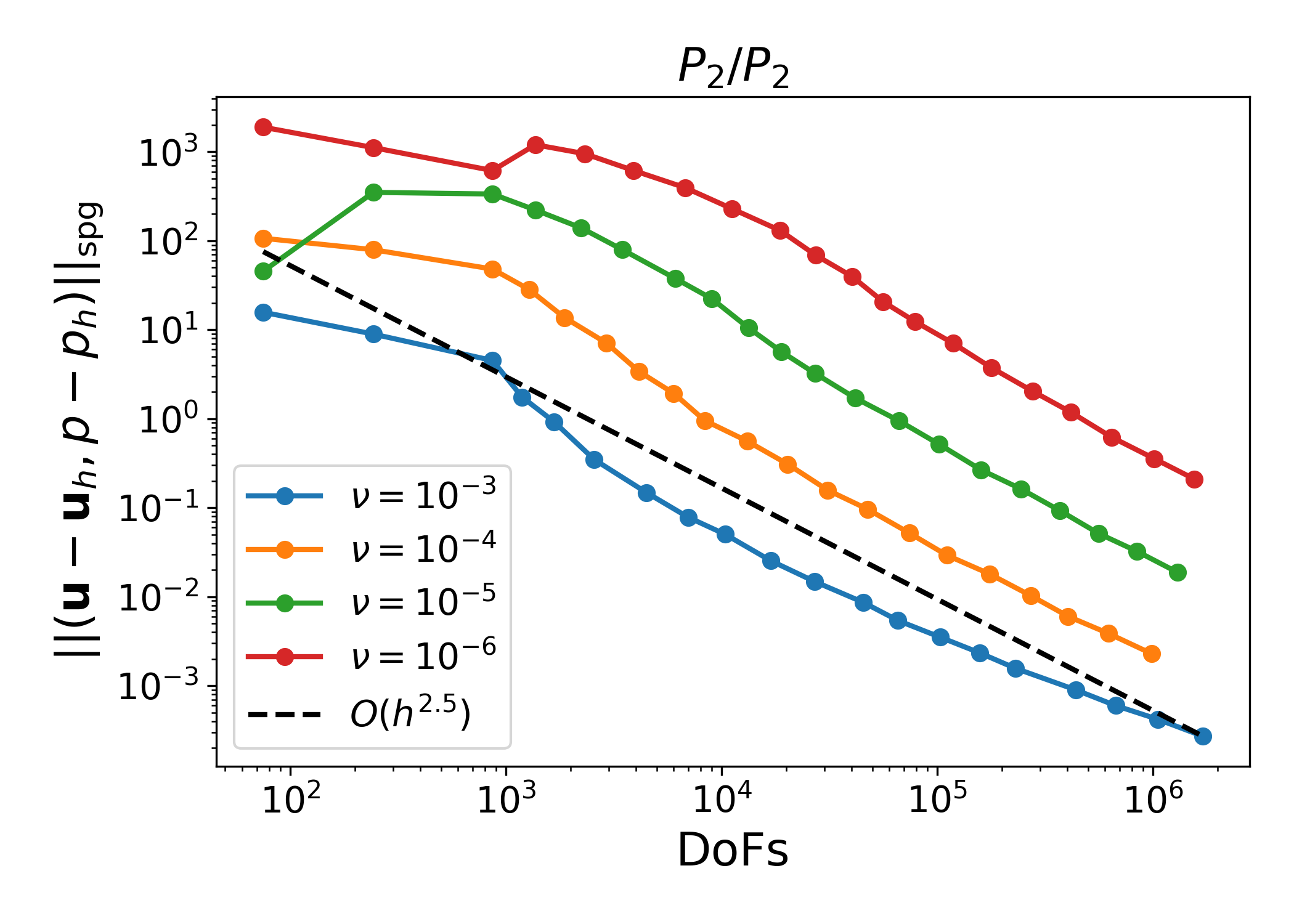}}
  \centerline{\includegraphics[width=0.32\textwidth]{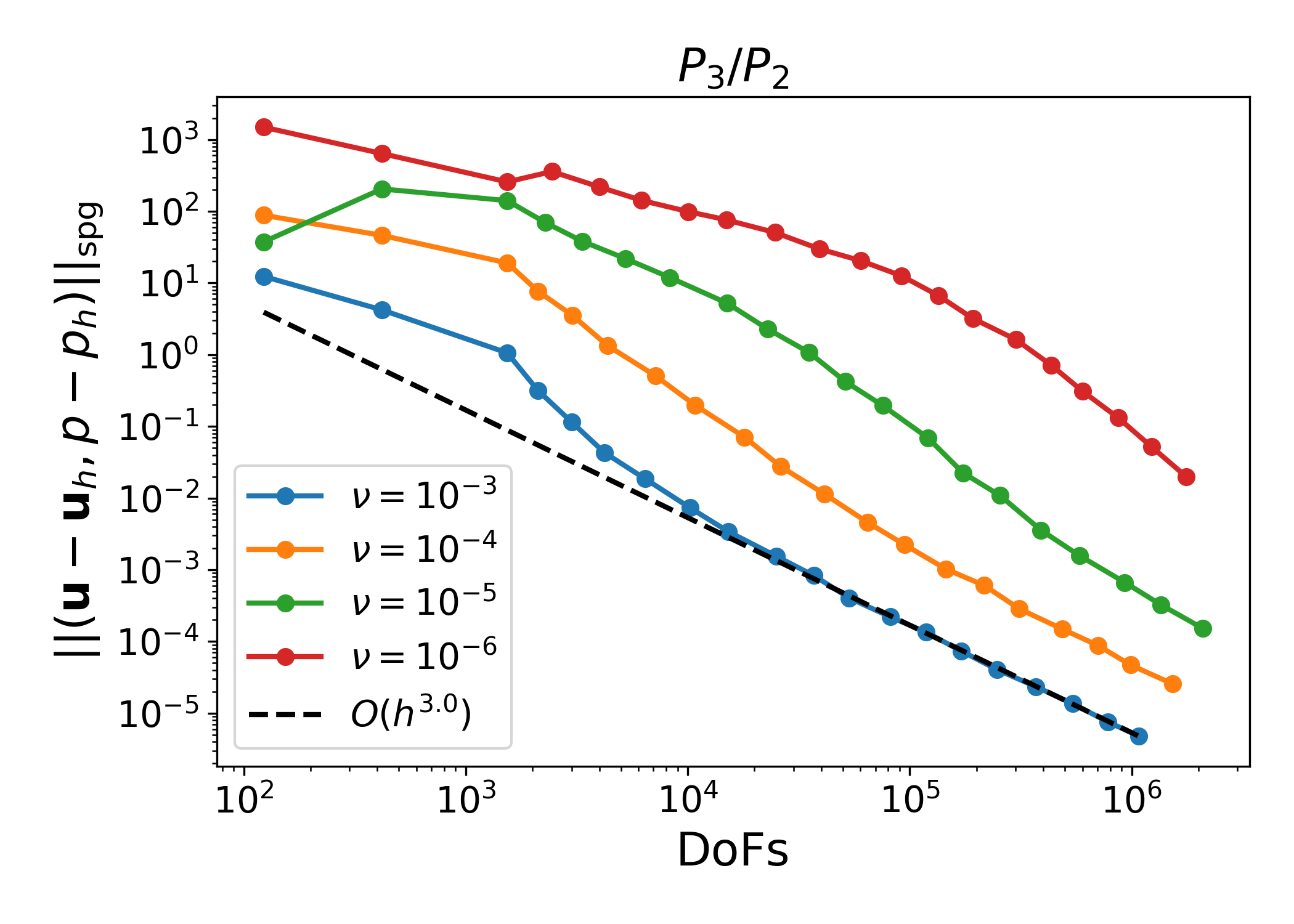}
   \hspace*{0.5em}
  \includegraphics[width=0.32\textwidth]{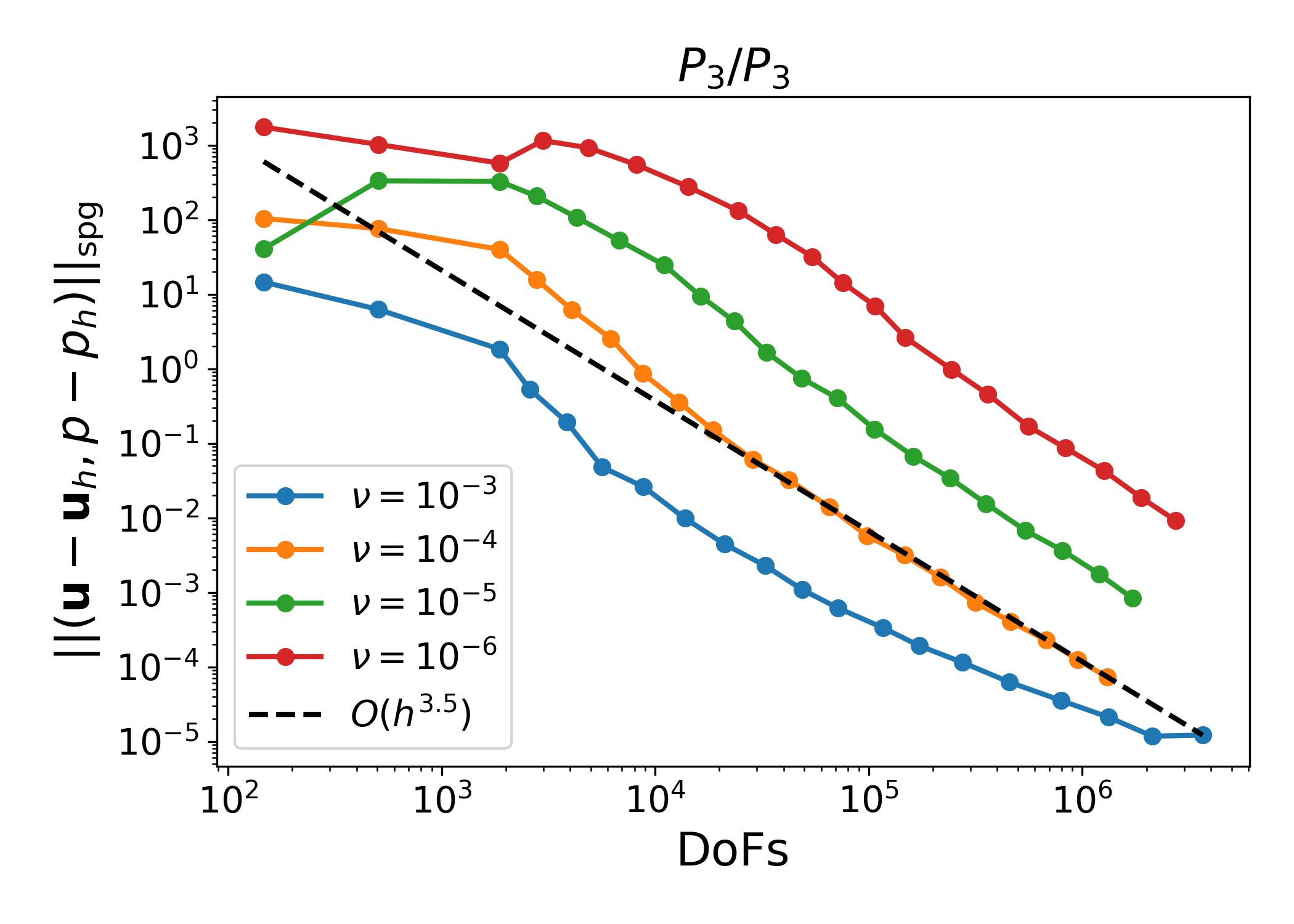}}
  \caption{Example~\ref{sec:layer}. Errors $\no{\(\bu-\bu_h,p-p_h\)}{\mathrm{spg}}$ 
  for different pairs of finite element spaces and different values of the viscosity coefficient.}
  \label{fig:bl-all-error}
\end{figure}

Figure~\ref{fig:bl-all-estimate} reveals that the proposed error estimator shows the 
same behavior: The estimates increase for smaller viscosity coefficients. In Figure~\ref{fig:bl-all-effectivity}, 
it can be seen that the effectivity indices take values of around 9, independently of the viscosity 
coefficient,  for this example, 
similarly as in Example~\ref{sec:smooth}. Hence, also these results support the robustness of the proposed
error estimator.

\begin{figure}[t!]
  \centerline{
  \includegraphics[width=0.32\textwidth]{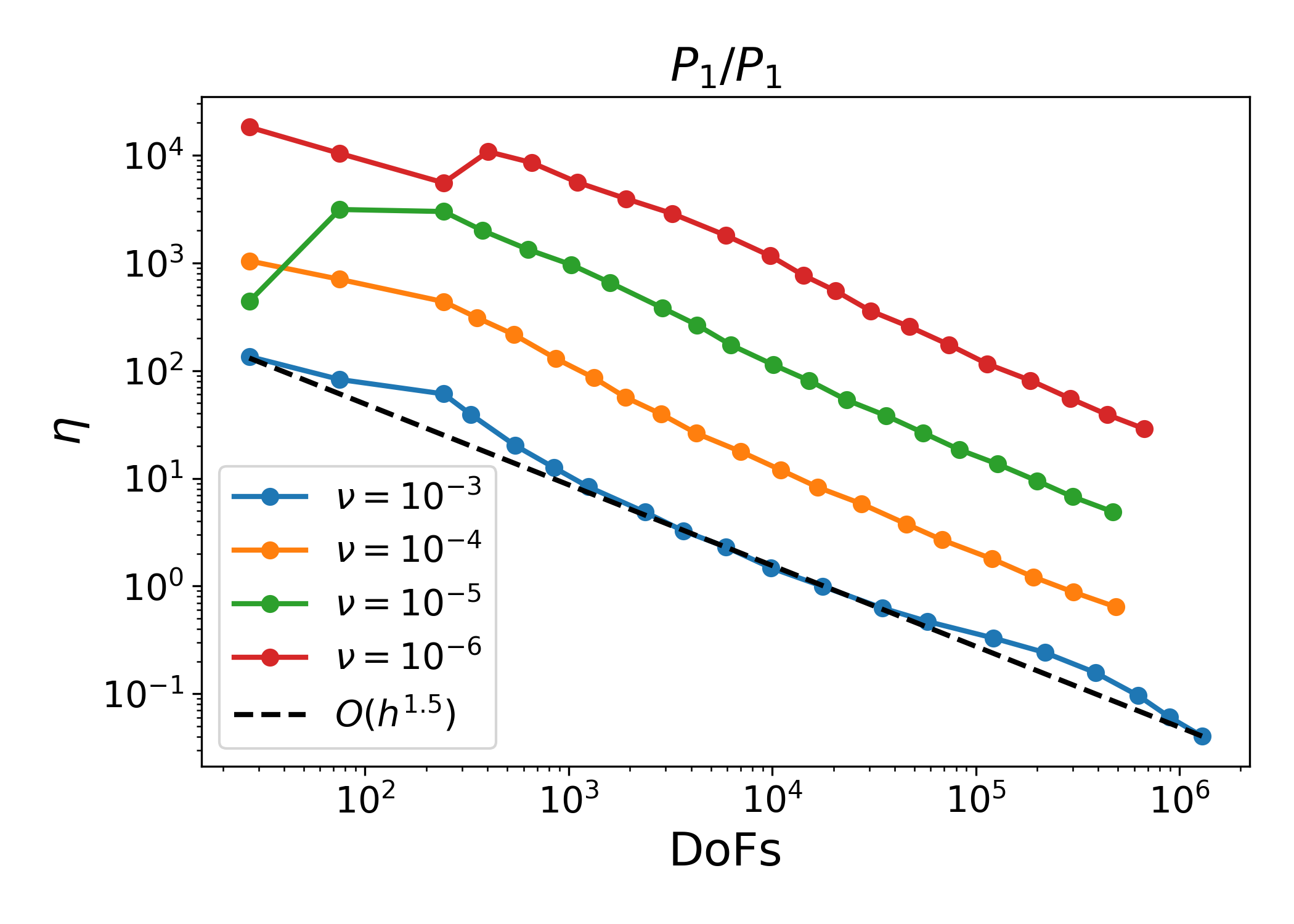}
  \hspace*{0.5em}
  \includegraphics[width=0.32\textwidth]{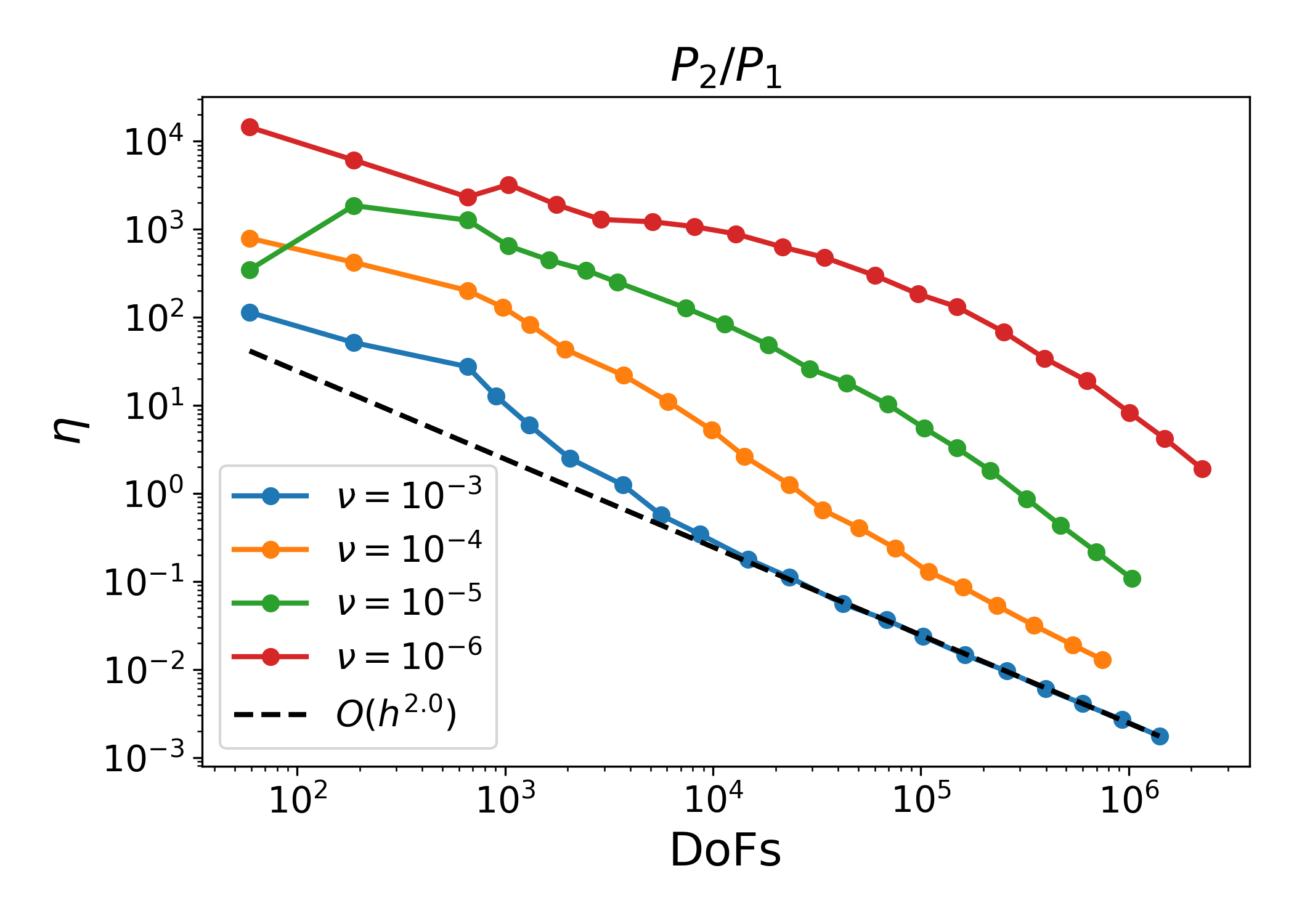}
  \hspace*{0.5em}
  \includegraphics[width=0.32\textwidth]{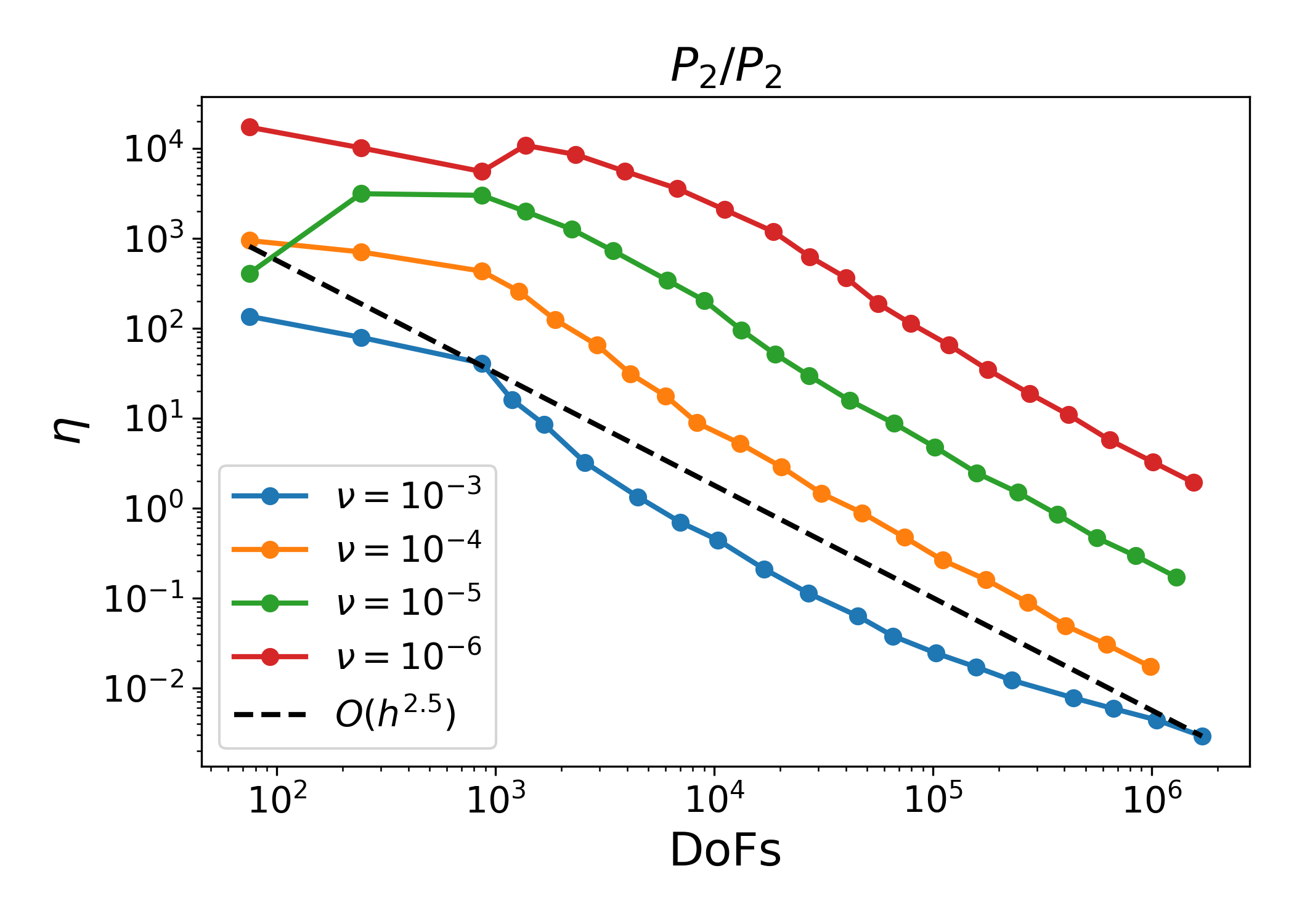}}
 \centerline{\includegraphics[width=0.32\textwidth]{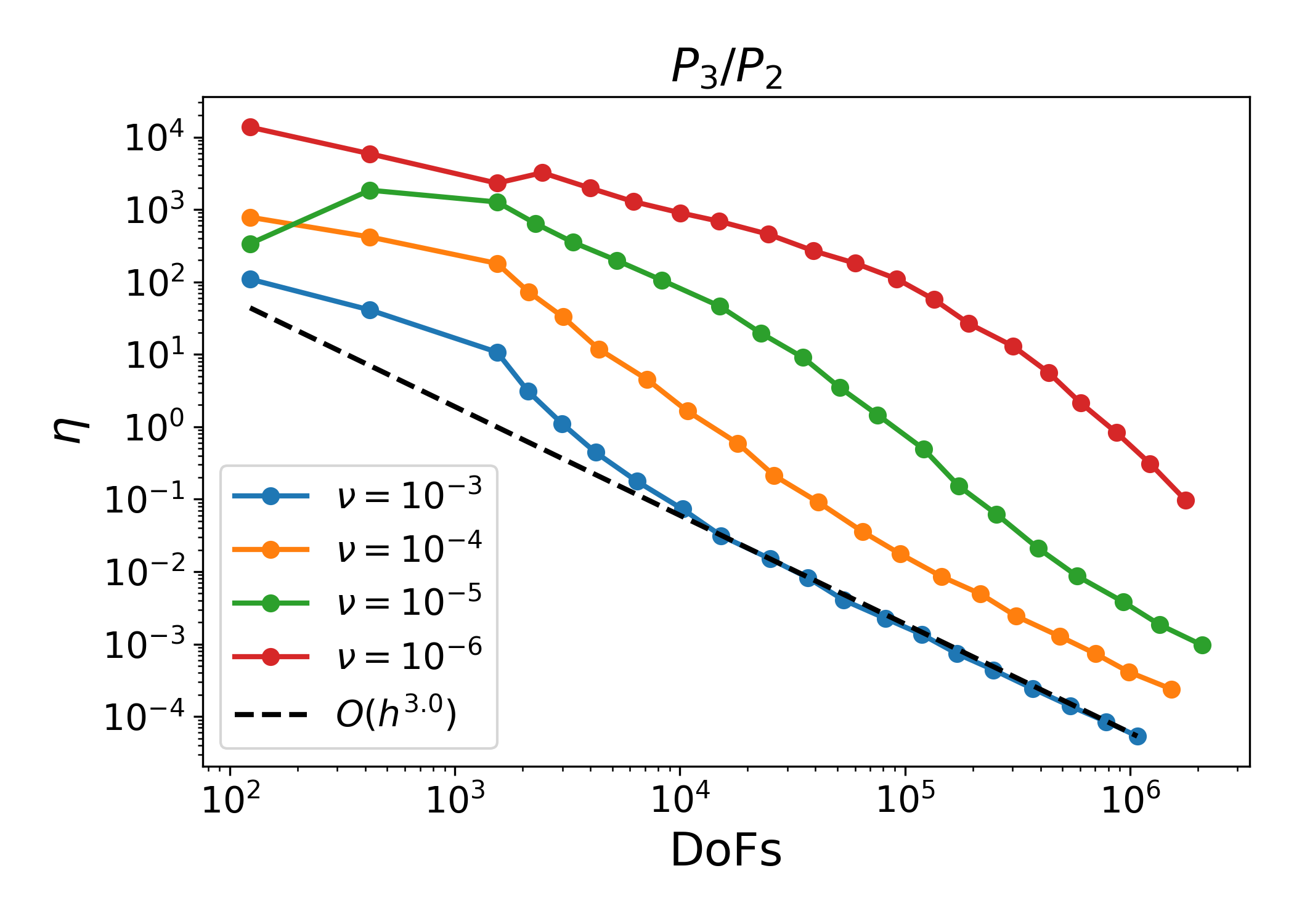}
  \hspace*{0.5em}
  \includegraphics[width=0.32\textwidth]{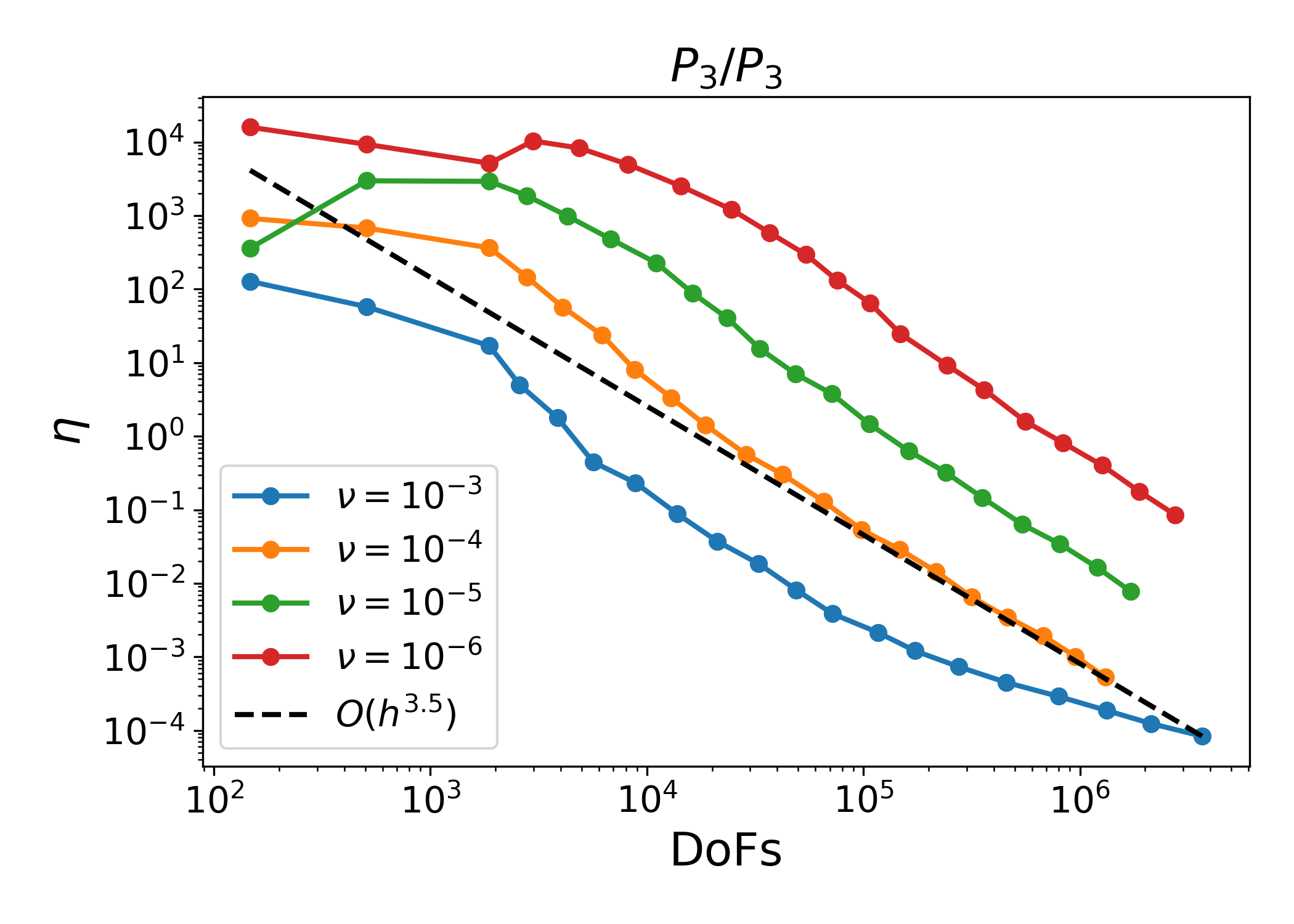}}
  \caption{Example~\ref{sec:layer}. The a posteriori error estimator $\eta$ 
  for different pairs of finite element spaces and different values of the viscosity coefficient.}
  \label{fig:bl-all-estimate}
\end{figure}

\begin{figure}[t!]
  \centerline{
  \includegraphics[width=0.32\textwidth]{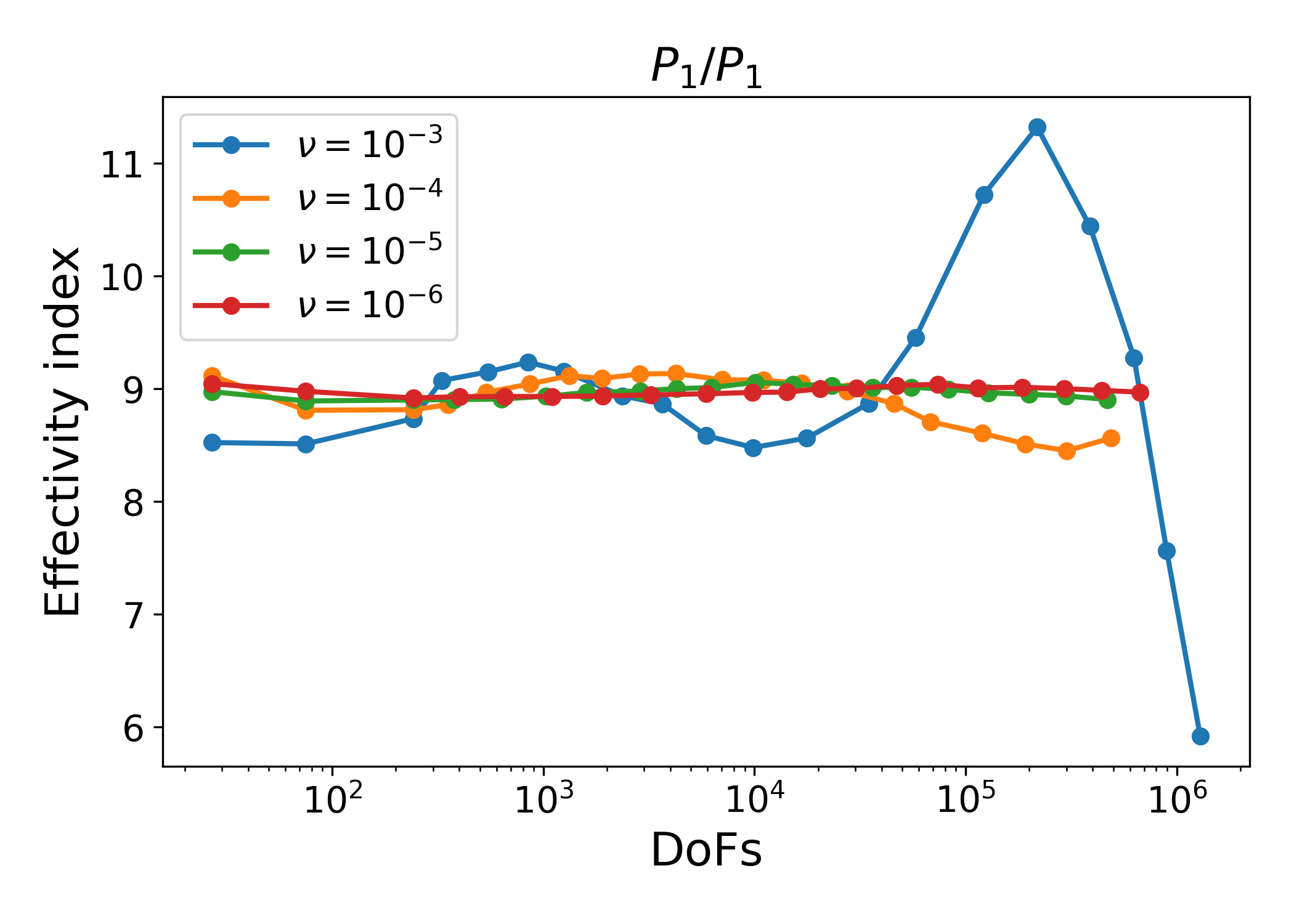}
  \hspace*{0.5em}
  \includegraphics[width=0.32\textwidth]{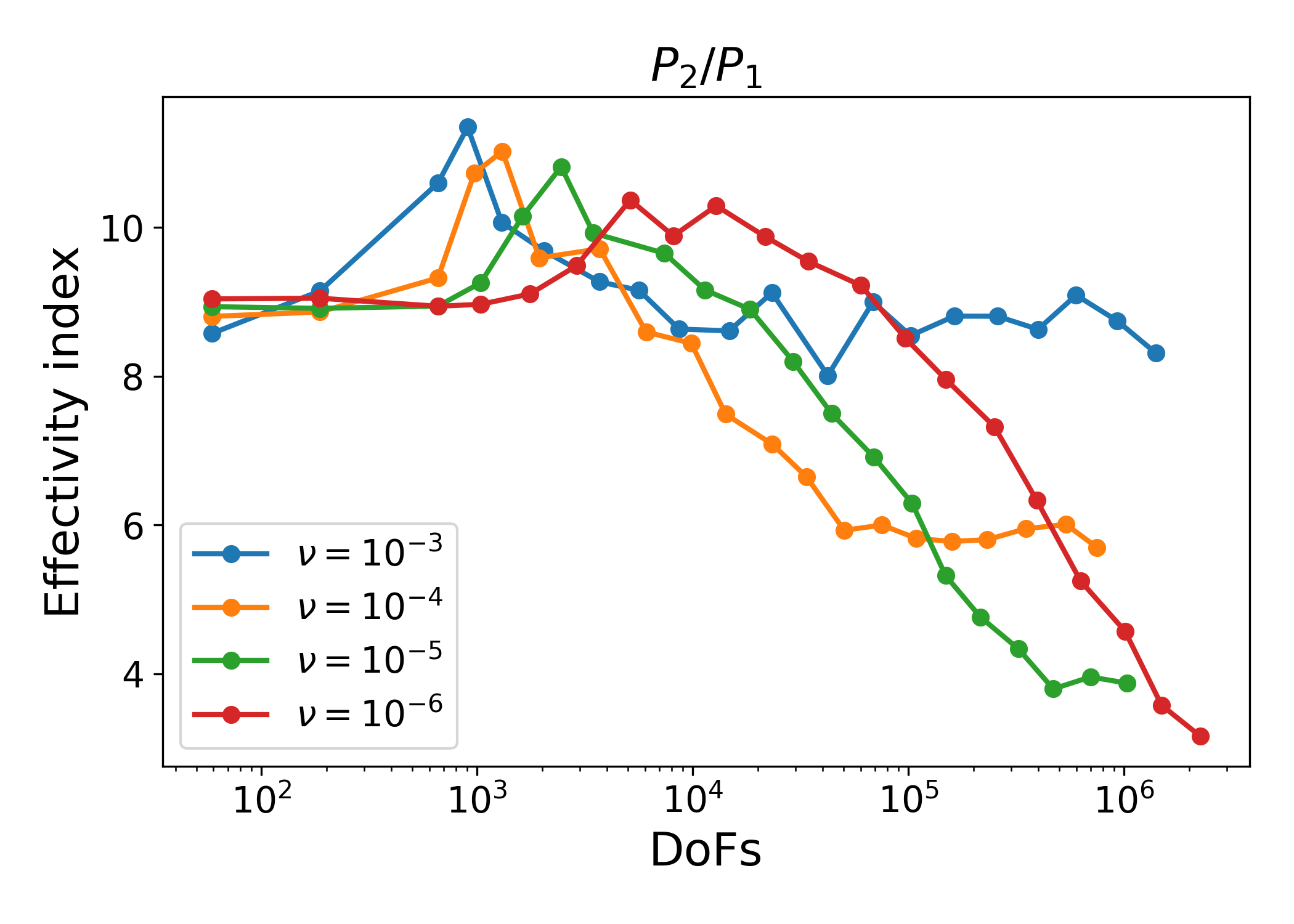}
  \hspace*{0.5em}
  \includegraphics[width=0.32\textwidth]{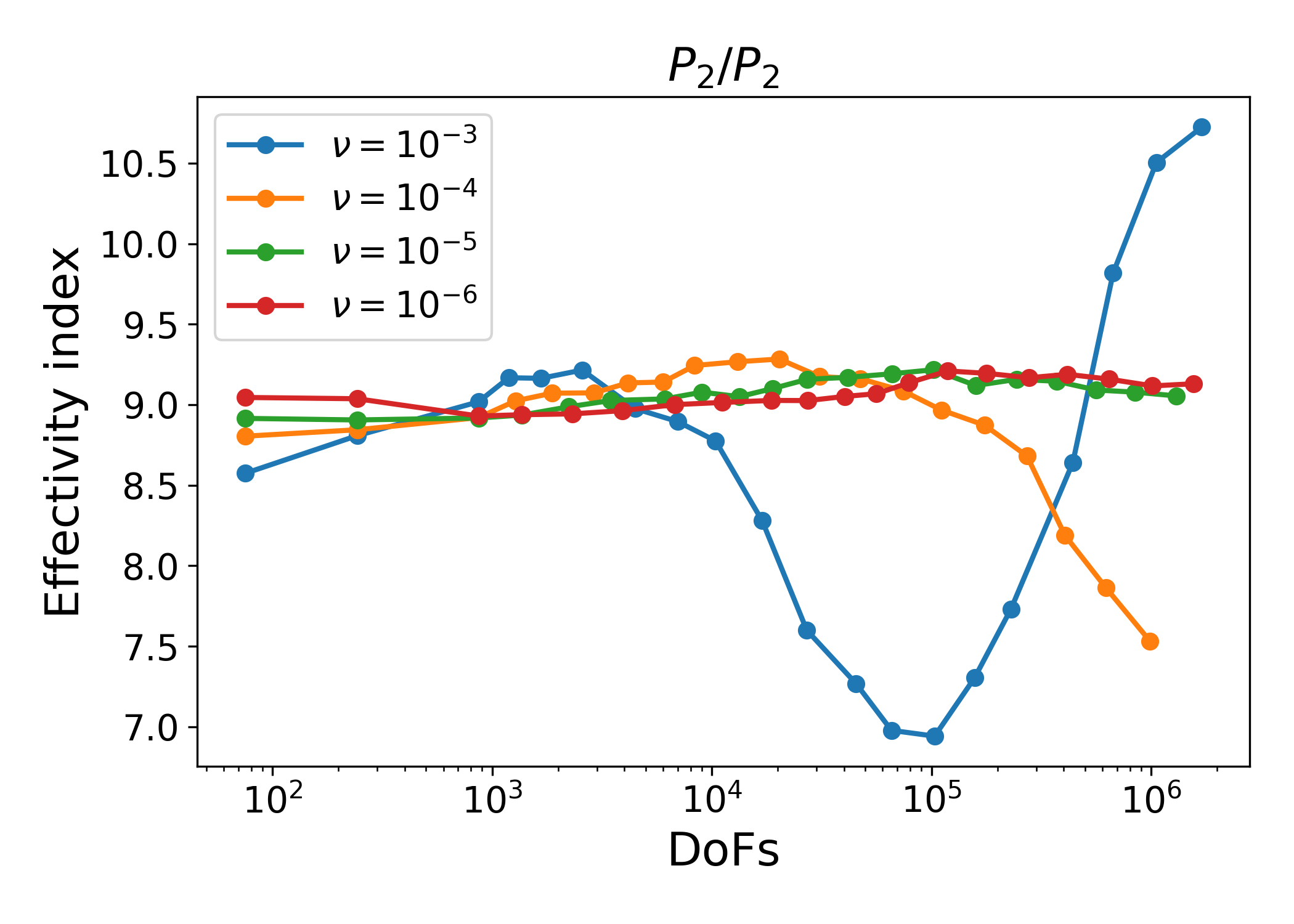}}
  \centerline{
  \includegraphics[width=0.32\textwidth]{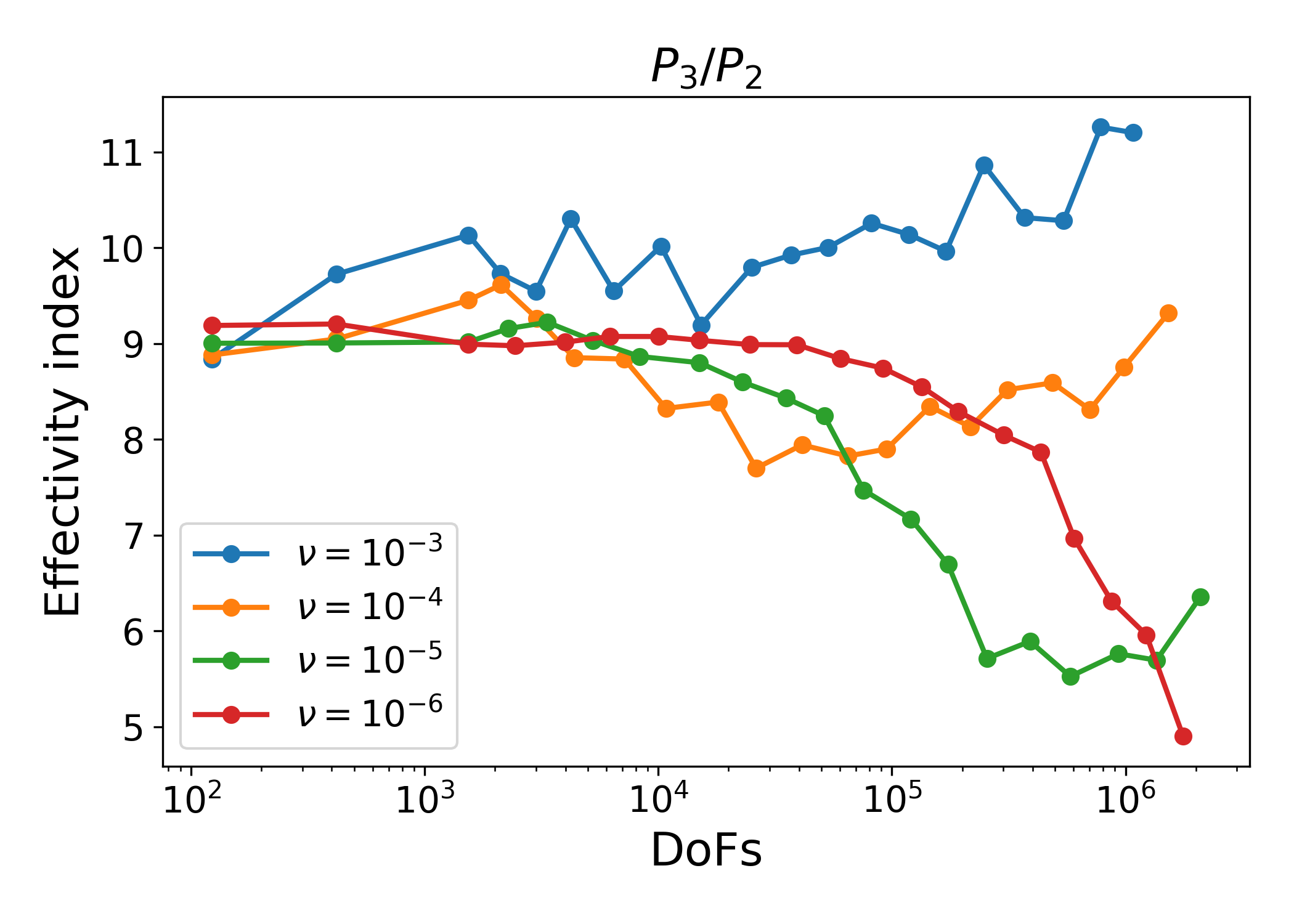}
  \hspace*{0.5em}
  \includegraphics[width=0.32\textwidth]{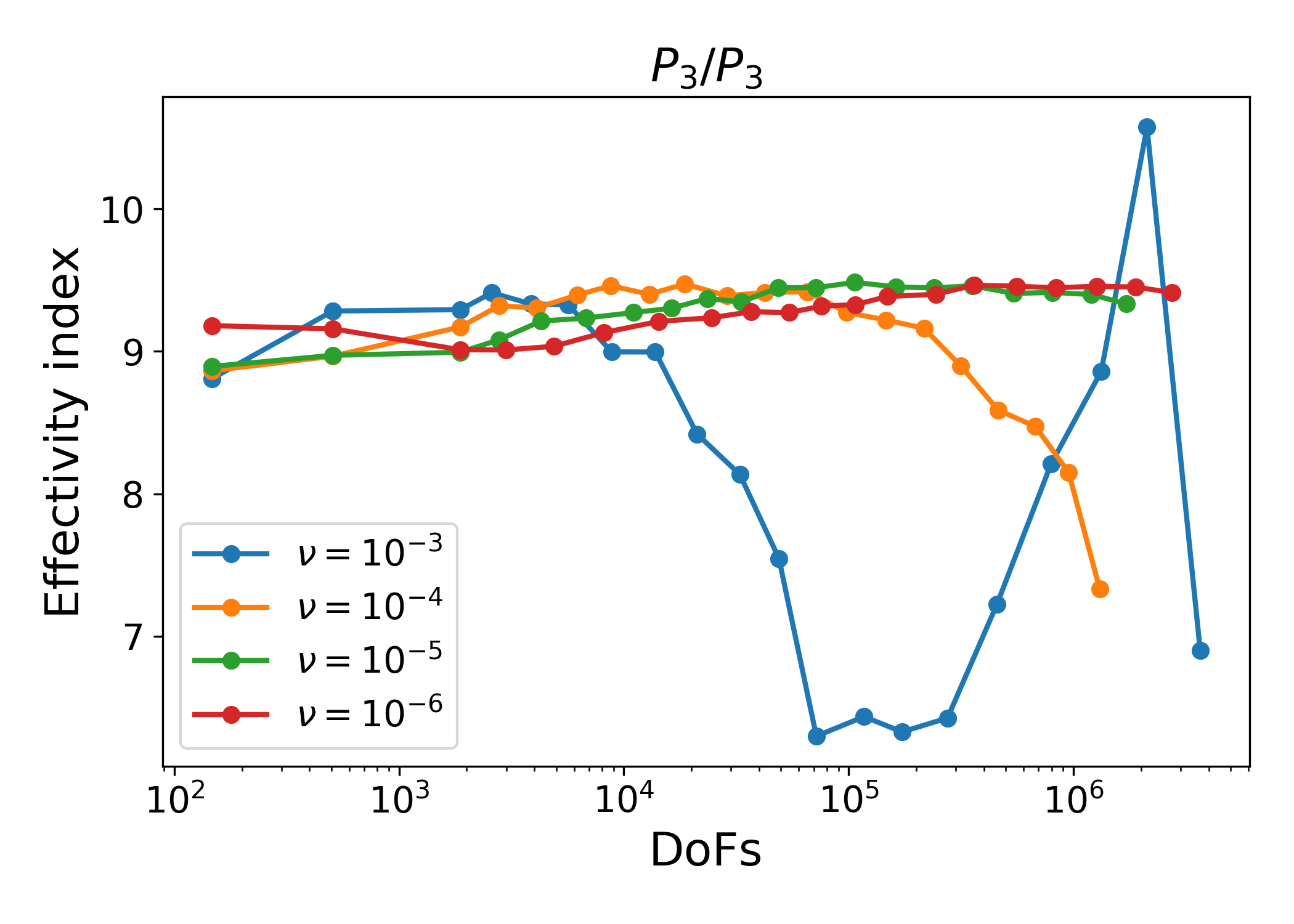}}
  \caption{Example~\ref{sec:layer}. Effectivity indices for different pairs of finite element spaces and different values of the viscosity coefficient.}
  \label{fig:bl-all-effectivity}
\end{figure}

Finally, Figure~\ref{fig:bl-all-indicators} presents exemplarily an investigation
of the importance of the individual terms of the error estimator $\eta$. One can see that 
again four out of the five terms are usually of the same order, whereas $\eta_F$ is considerably 
smaller. 

\begin{figure}[t!]
  \centerline{
  \includegraphics[width=0.32\textwidth]{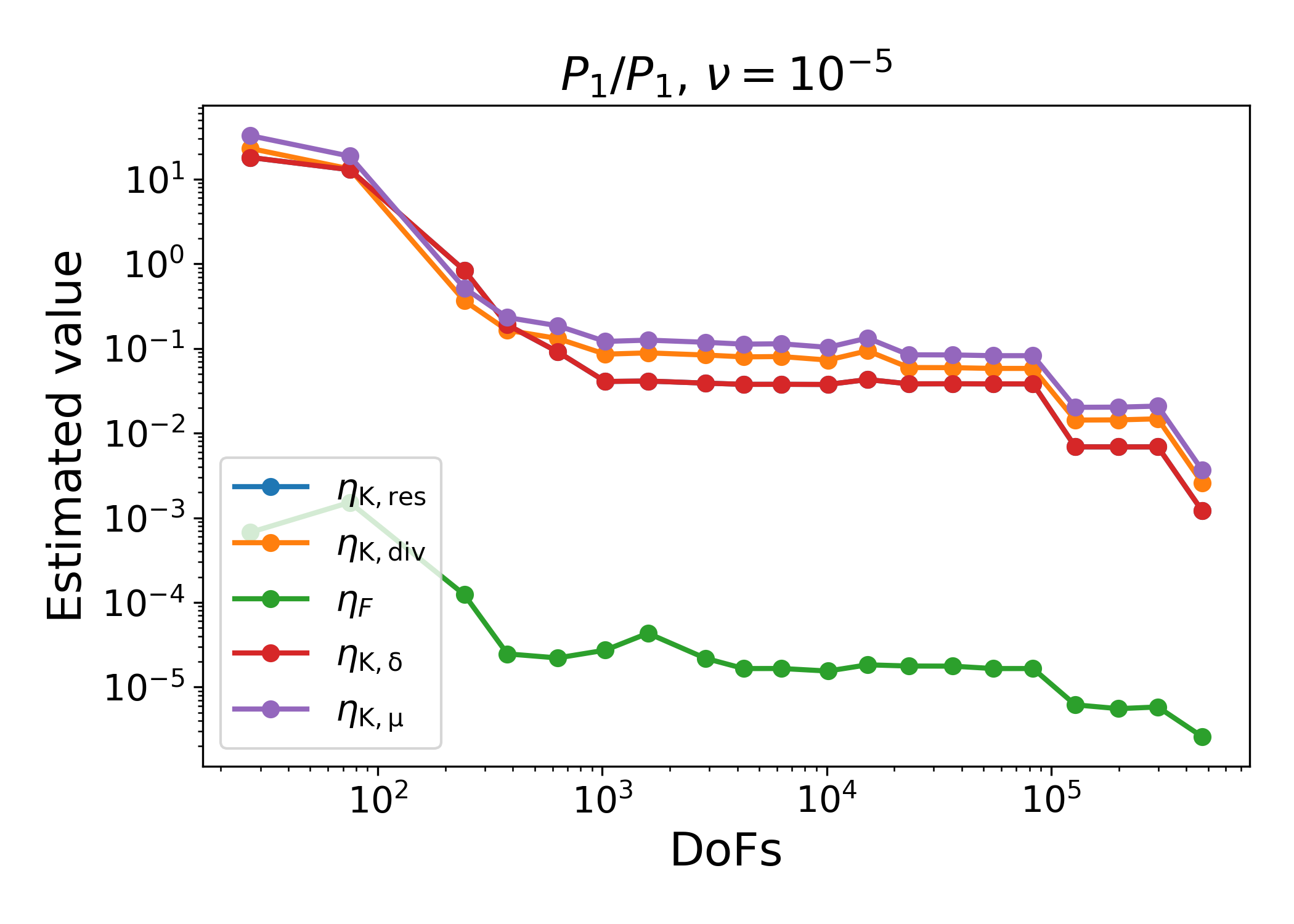}
  \hspace*{0.5em}
  \includegraphics[width=0.32\textwidth]{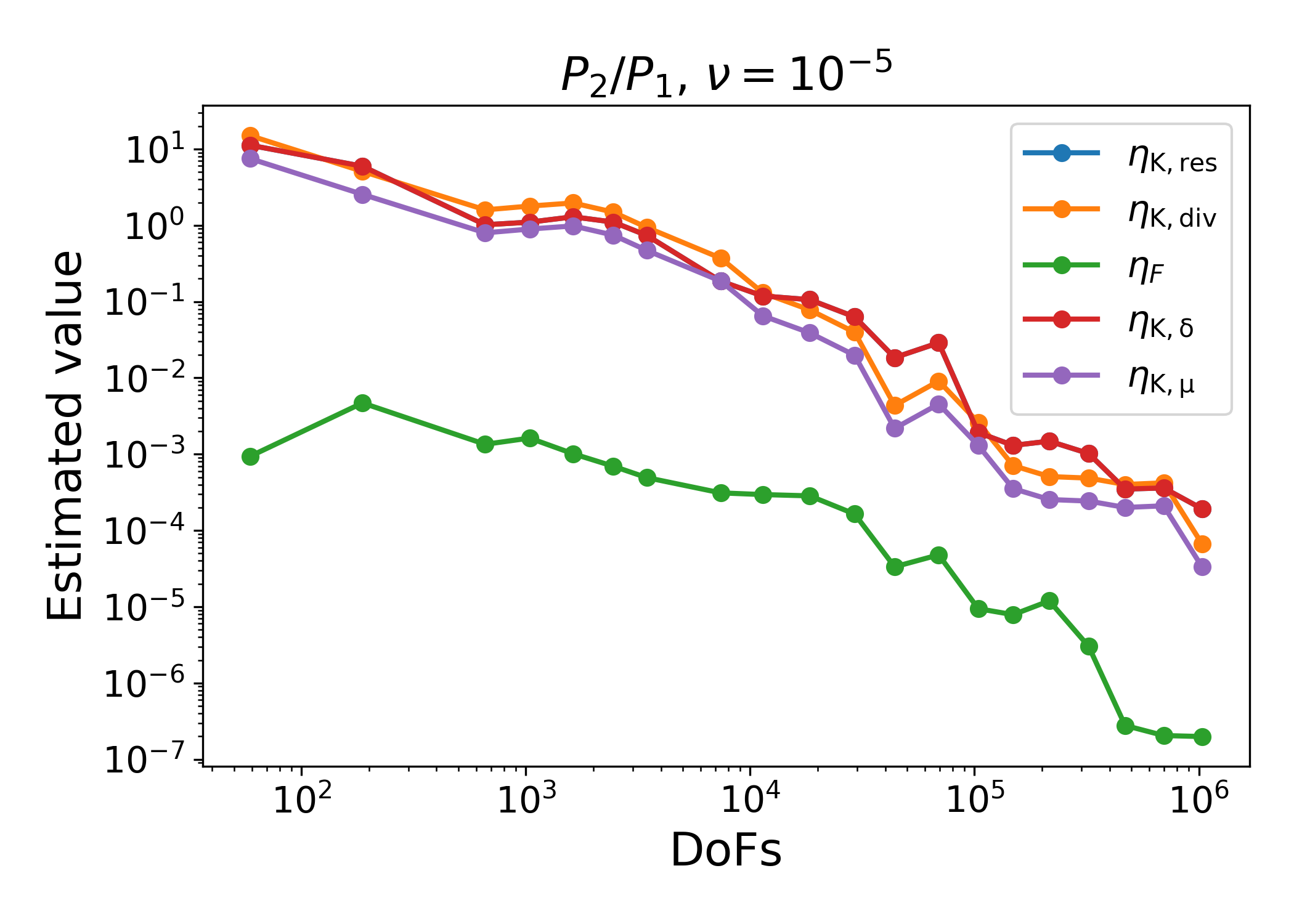}
  \hspace*{0.5em}
  \includegraphics[width=0.32\textwidth]{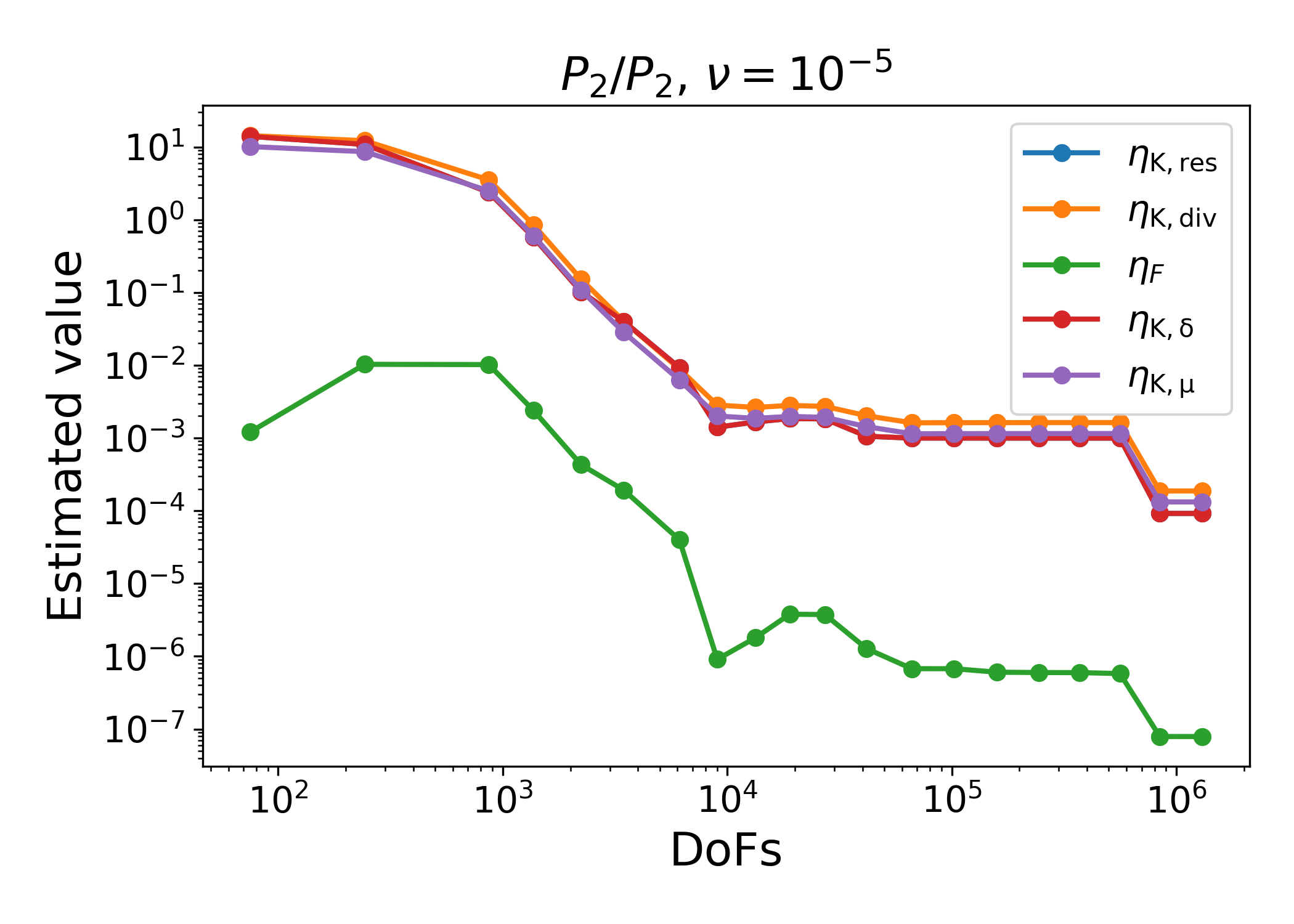}}
  \centerline{\includegraphics[width=0.32\textwidth]{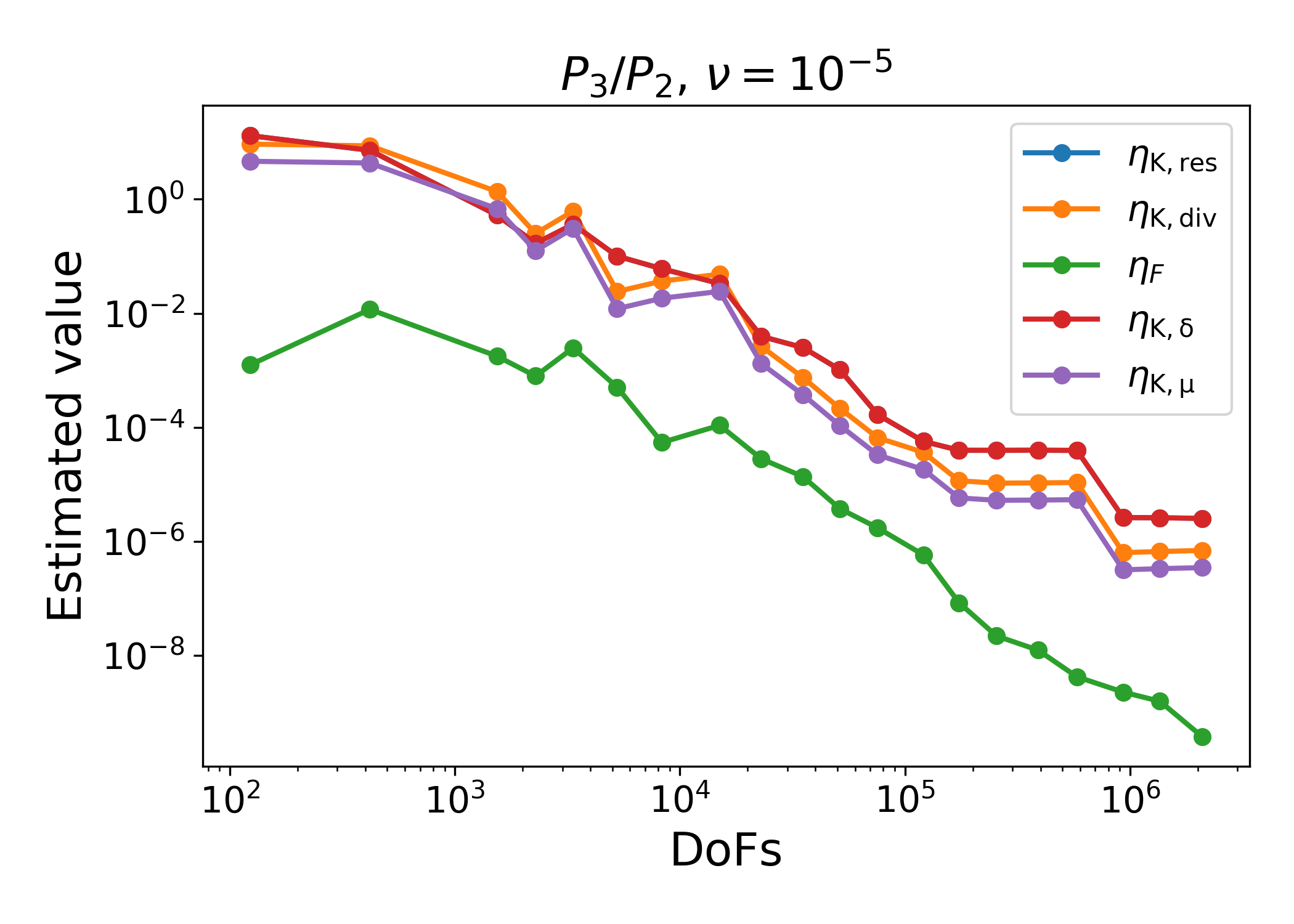}
  \hspace*{0.5em}
  \includegraphics[width=0.32\textwidth]{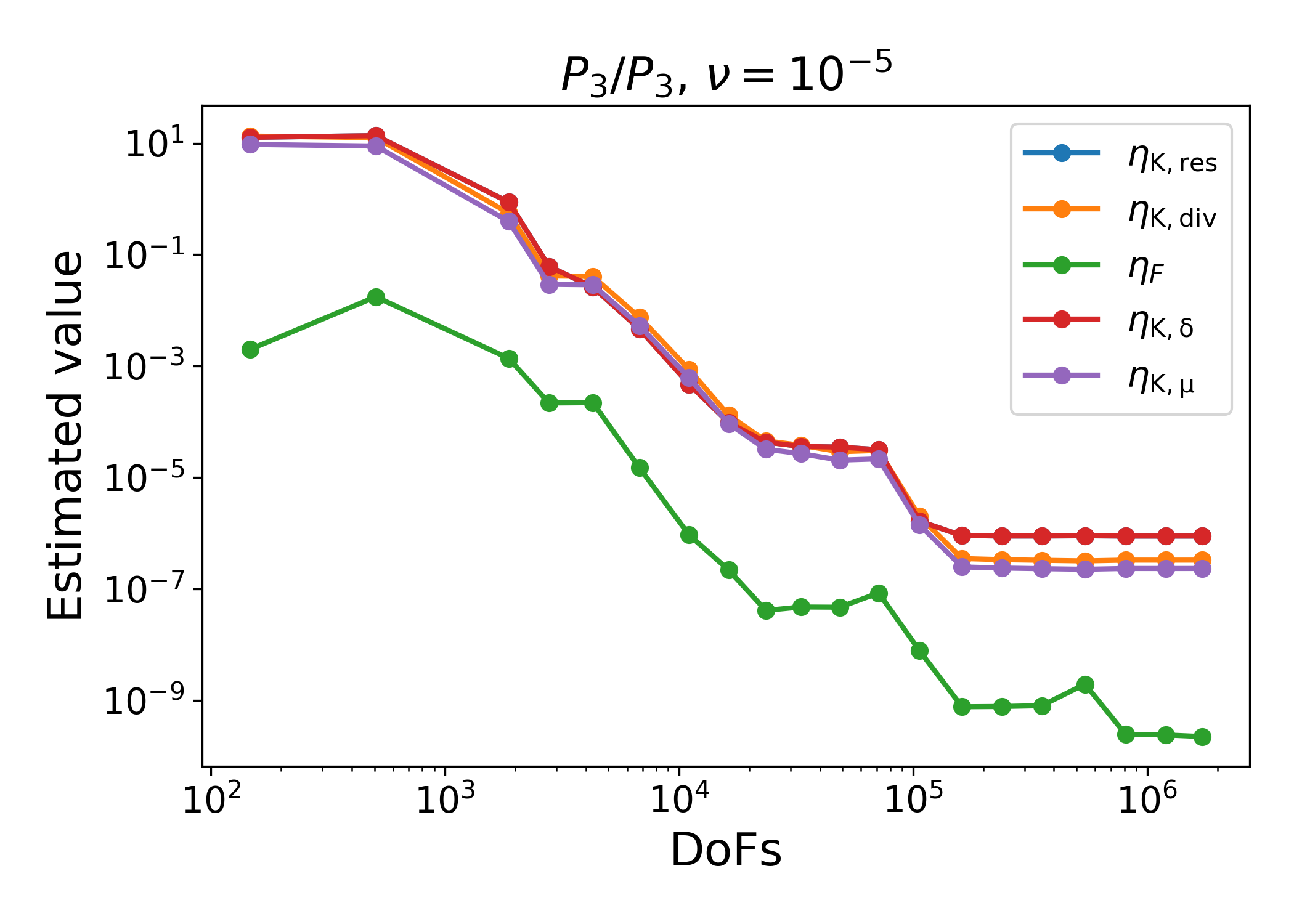}}
  \caption{Example~\ref{sec:layer}. Contributions of the individual parts of $\eta$, simulations with  $\nu=10^{-5}$.}
  \label{fig:bl-all-indicators}
\end{figure}

\subsection{Summary of the Numerical Studies}

This section studied two examples with different properties and setup:
\blist
\item[$\bullet$] Smooth solution with $\no{\(\bu-\bu_h,p-p_h\)}{\mathrm{spg}}$ independent of $\nu$ vs. 
solution with boundary layers and $\no{\(\bu-\bu_h,p-p_h\)}{\mathrm{spg}}$ increasing if $\nu$ 
becomes smaller,
\item[$\bullet$] uniform vs. adaptive grid refinement.
\end{list}
In all cases, independent of the pair of finite element spaces and the size of the viscosity coefficient in 
the convection-dominated regime, 
the effectivity index took values of around $9$. These results support the claimed robustness of the 
error estimator \eqref{eq:eta}.

\section{Some Considerations on the Navier--Stokes Equations}\label{sec:nse}

This section presents a numerical study that investigates the behavior of a modified a posteriori error 
estimator applied to the SUPG/PSPG/grad-div method of the incompressible Navier--Stokes problem. 

The Oseen problem \eqref{eq:weak-1}--\eqref{eq:weak-2} is a popular model problem for the numerical 
analysis of convection-dominated incompressible flow problems since it is a linear problem. The actual 
model for steady-state incompressible flows is the set of the incompressible Navier--Stokes equations 
given, already in weak form, by: 
Find $(\bu,p) \in \bV \times Q$ such that 
\begin{eqnarray}
\nu(\nabla\bu,\nabla\bv) + ((\bu\cdot \nabla)\bu,\bv)  - (p,\nabla \cdot \bv)
 &=& (\bff,\bv) + (\bg,\bv)_{\gN}
 \quad \foralls \bv\in \bV, \label{eq:nse-1} \\
(\nabla\cdot\bu,q) & = & 0 \quad \foralls q\in Q. \label{eq:nse-2}
\end{eqnarray}

An a priori error analysis of a stabilized SUPG/PSPG/grad-div finite element method for 
problem \eqref{eq:nse-1}--\eqref{eq:nse-2} is presented in \cite{TV96}. This analysis does not 
require that the velocity and pressure finite element spaces satisfy a discrete inf-sup condition. 
It is proposed to choose the stabilization parameters for all pairs of spaces to be 
$\delta_K \sim h_K^2$ and $\mu_K \sim 1$. While these are the standard asymptotic choices for 
inf-sup stable pairs of finite element spaces for the Oseen problem, compare \eqref{eq:para_inf_sup}, they correspond 
for equal order pairs to a recommendation in the situation where the problem is not strongly 
convection-dominated, see \eqref{eq:para_equal}.
In \cite{TV96}, the existence and uniqueness of a solution to the 
discrete problem is proved for small data (or large viscosity coefficients). Then, an a priori
error analysis is performed for the norm 
\begin{equation}\label{eq:norm_TV96}
\left(\no{\nabla\bv}{\bL^2(\Omega)}^2+ \no{q}{L^2(\Omega)}^2  
+\sum_{K\in\mathcal T_h}\delta_K \no{\nabla q}{\bL^2(K)}^2\right)^{1/2},
\end{equation}
where here the case of continuous pressure finite element spaces is considered. 
Comparing this norm with $\no{\(\bv,q\)}{\mathrm{spg}}$ defined in \eqref{eq:norm_spg}, one can see 
first that \eqref{eq:norm_TV96} does not contain the viscosity coefficient. The motivation in \cite{TV96} is that the 
case $\nu\to0$ is not of interest for simulations of the steady-state Navier--Stokes equations, since the 
solution is not unique for small viscosity coefficients. Second, the grad-div term is missing 
in \eqref{eq:norm_TV96}. And last, the convective term of the SUPG stabilization is missing in 
\eqref{eq:norm_TV96}, since including this term with $\bb$ replaced by $\bu_h$ would not define a norm. 
The a priori error analysis presented in \cite{TV96} leads for smooth solutions and  
$P_{k}/P_{k-1}$, $k\ge 2$, and $P_k/P_k$, $k\ge 1$, 
to the order of convergence $k$ in the norm \eqref{eq:norm_TV96}.

Having an a posteriori error estimator for the SUPG/PSPG/grad-div method for the Oseen problem at hand, 
a natural question is how it performs, after appropriate modifications, for the same discretization 
of the Navier--Stokes problem \eqref{eq:nse-1}--\eqref{eq:nse-2}. These modifications comprise:
\begin{list}{}{\itemsep0.0ex\parsep0.1ex\topsep0.2ex\leftmargin1.6em
\labelwidth1.3em}
\item[$\bullet$] to define the mesh cell residual by 
\[
\br_K(\bu_h,p_h) = \left.(\bff + \nu \Delta \bu_h - (\bu_h \cdot \nabla)\bu_h - \nabla p_h )\right|_K
\quad \foralls K \in \mathcal T_h,
\]
\item[$\bullet$] to take $\sigma_0 = 0$ in $\eta_{K,\mathrm{res}}$ and $\eta_F$, so that the corresponding terms
can be removed from the minimum, 
\item[$\bullet$] to replace  $\no{\bb}{L^\infty(F)}$ in $\eta_F$ by $\no{\bu_h}{L^\infty(F)}$.
\end{list}
Since the error estimate with this estimator will potentially depend on the value of the 
viscosity coefficient and on the contribution from the grad-div term, we think that \eqref{eq:norm_TV96}
is not an appropriate norm to compare the a posteriori estimates with. The definition of such a norm 
should resemble the norms  $\no{\(\bv,q\)}{\mathrm{spg}}$ or $\no{\(\bv,q\)}{\mathrm{spg},p}$ as much as possible
and we think that an appropriate norm is 
\begin{equation*}
\no{\(\bv,q\)}{\mathrm{spg,nse}}=\left(\nu 
\no{\nabla\bv}{\bL^2(\Omega)}^2+  \nu \no{q}{L^2(\Omega)}^2  + 
\sum_{K\in\mathcal T_h} \mu_K \no{\dive \bv}{L^2(K)}^2 
+\sum_{K\in\mathcal T_h}\delta_K \no{\nabla q}{\bL^2(K)}^2\right)^{1/2}.
\end{equation*}

The numerical studies considered the Navier--Stokes problem with the right-hand side and the Dirichlet 
boundary condition prescribed such that \eqref{eq:smooth_u}--\eqref{eq:smooth_p} is the solution of \eqref{eq:nse-1}--\eqref{eq:nse-2}.

As can be seen in \eqref{eq:para_equal}, there is not a unique optimal asymptotic choice of the grad-div 
stabilization parameter for equal order pairs of spaces if the problem is not strongly dominated by 
convection. We observed in our numerical studies that with $\mu = 0.5 h_K$ one obtains notably more 
accurate solutions than with $\mu= 0.5$, while both choices fit into \eqref{eq:para_equal}. For this 
reason and the for sake of brevity, here only the former results will be presented. To be concrete, 
the SUPG stabilization parameter was chosen in all simulations to be $\delta_K = 0.5 h_K^2$ and the grad-div 
parameter to be $\mu=0.5$ for inf-sup stable pairs of spaces and $\mu = 0.5 h_K$ for equal order 
pairs. 

The  SUPG/PSPG/grad-div method for the Navier--Stokes equations \eqref{eq:nse-1}--\eqref{eq:nse-2} has in principle the 
form \eqref{eq:oseen_SUPG_PSPG_graddiv}--\eqref{eq:l_spg} with the obvious modification of the convective term. For the velocity
field in the SUPG test function, always the current iterate of a Picard iteration to solve the 
nonlinear problem was taken.

\begin{figure}[t!]
\centerline{\includegraphics[width=0.32\textwidth]{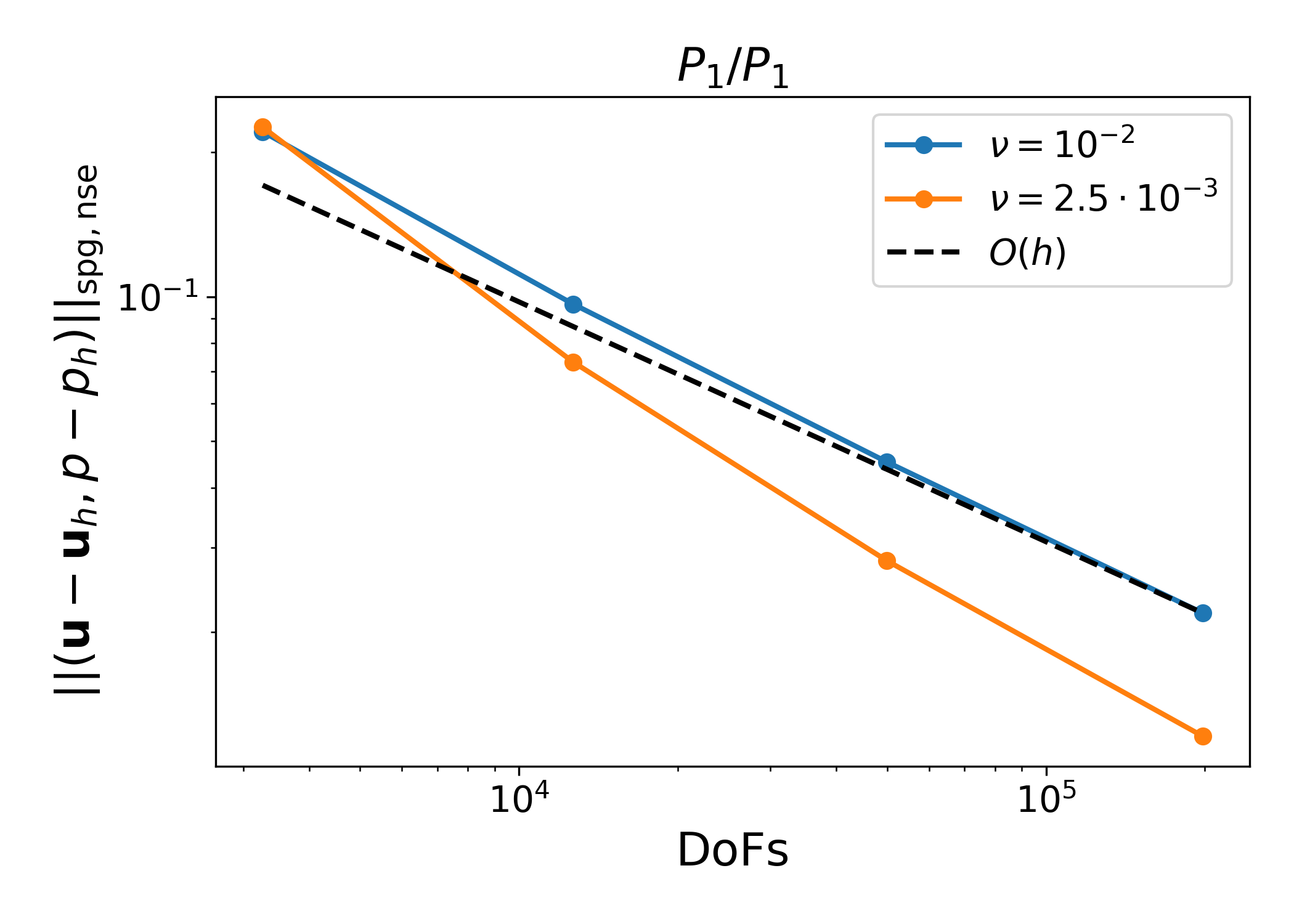}
  \hspace*{0.5em}
  \includegraphics[width=0.32\textwidth]{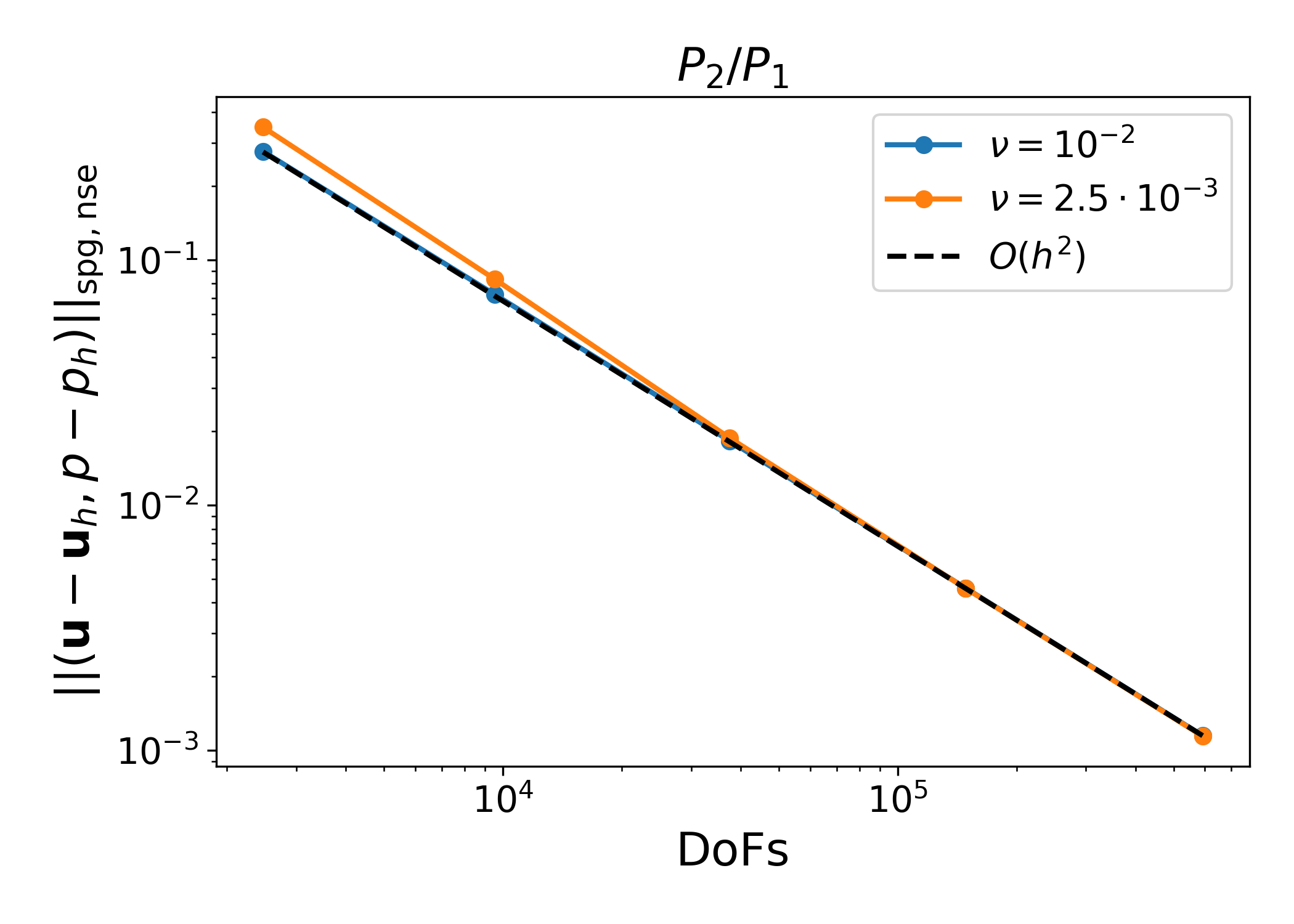}
  \hspace*{0.5em}
  \includegraphics[width=0.32\textwidth]{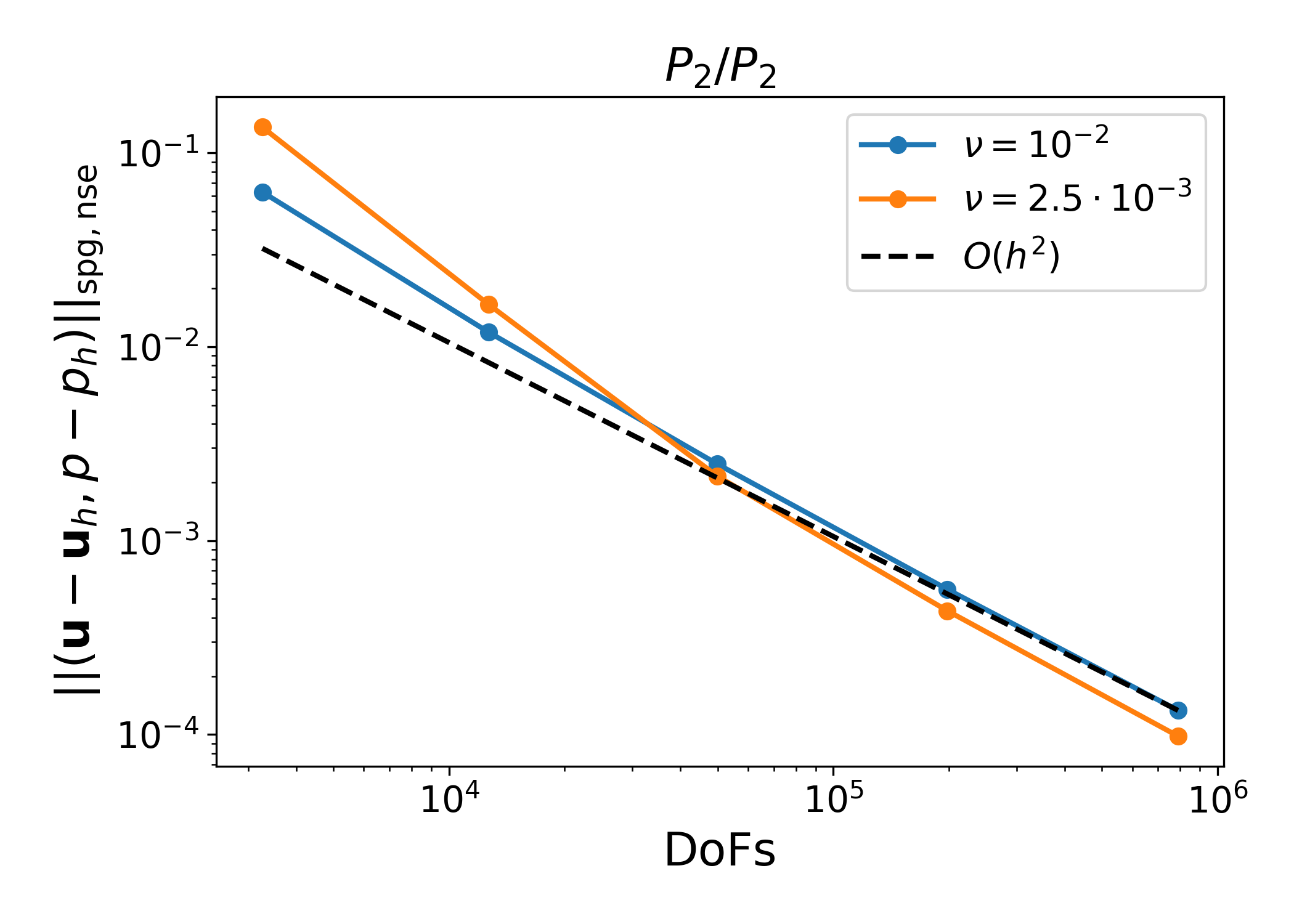}}
\centerline{\includegraphics[width=0.32\textwidth]{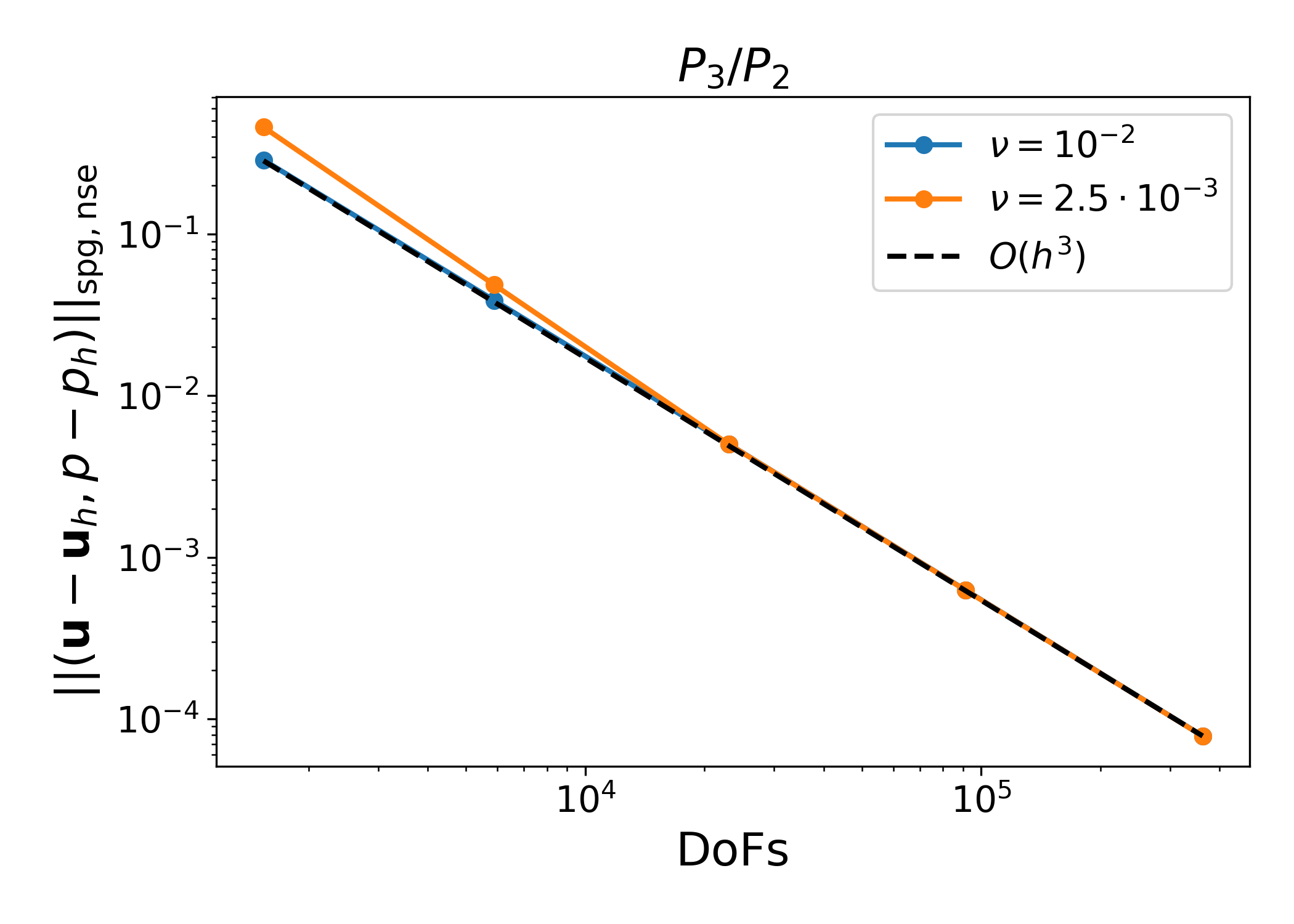}
\hspace*{0.5em}
\includegraphics[width=0.32\textwidth]{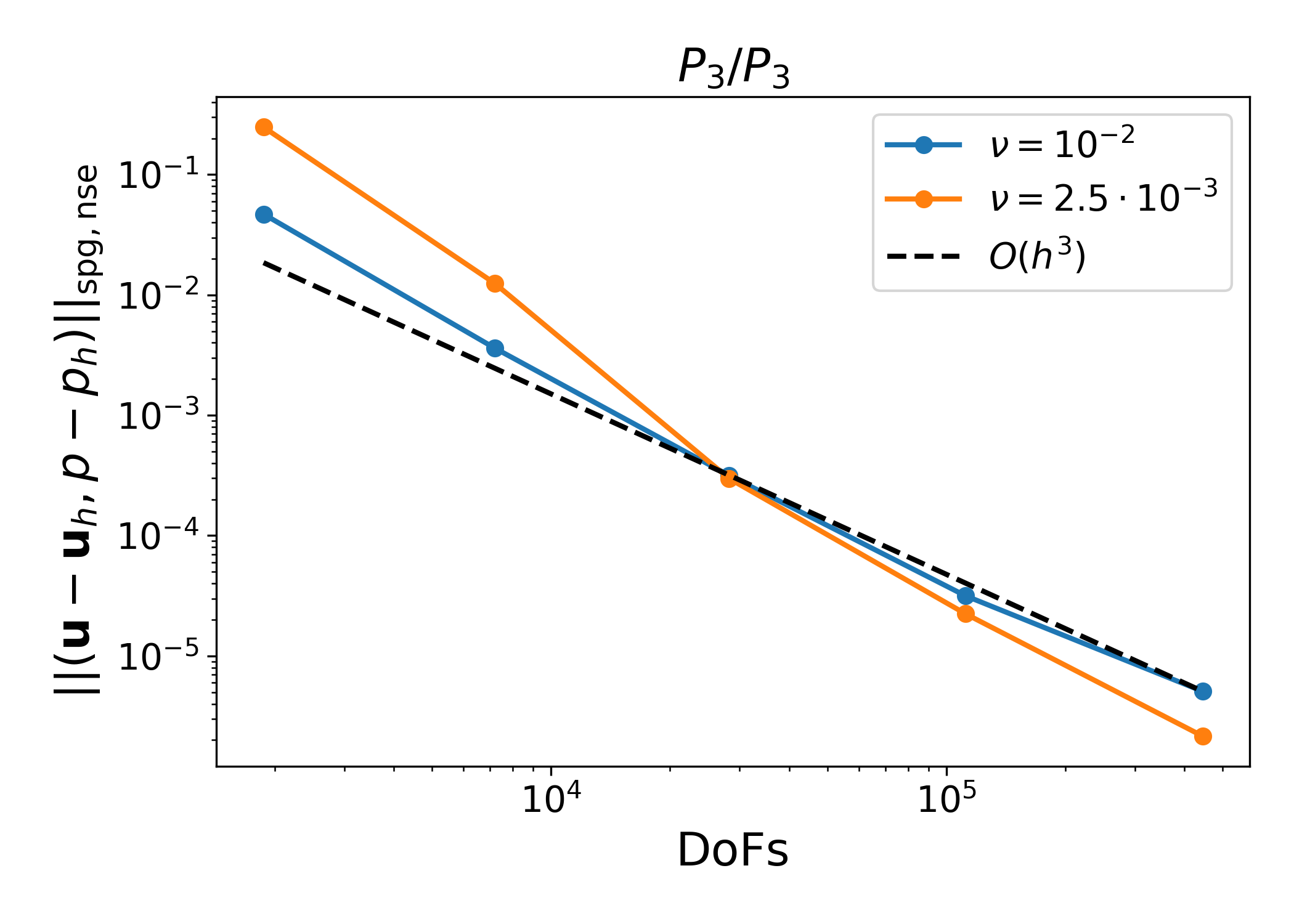}}
  \caption{Navier--Stokes problem. Errors $\no{\(\bu-\bu_h,p-p_h\)}{\mathrm{spg,nse}}$ 
  for different pairs of finite element spaces and different values of the viscosity coefficient.}
  \label{fig:nse-smooth-all-error}
\end{figure}

\begin{figure}[t!]
\centerline{\includegraphics[width=0.32\textwidth]{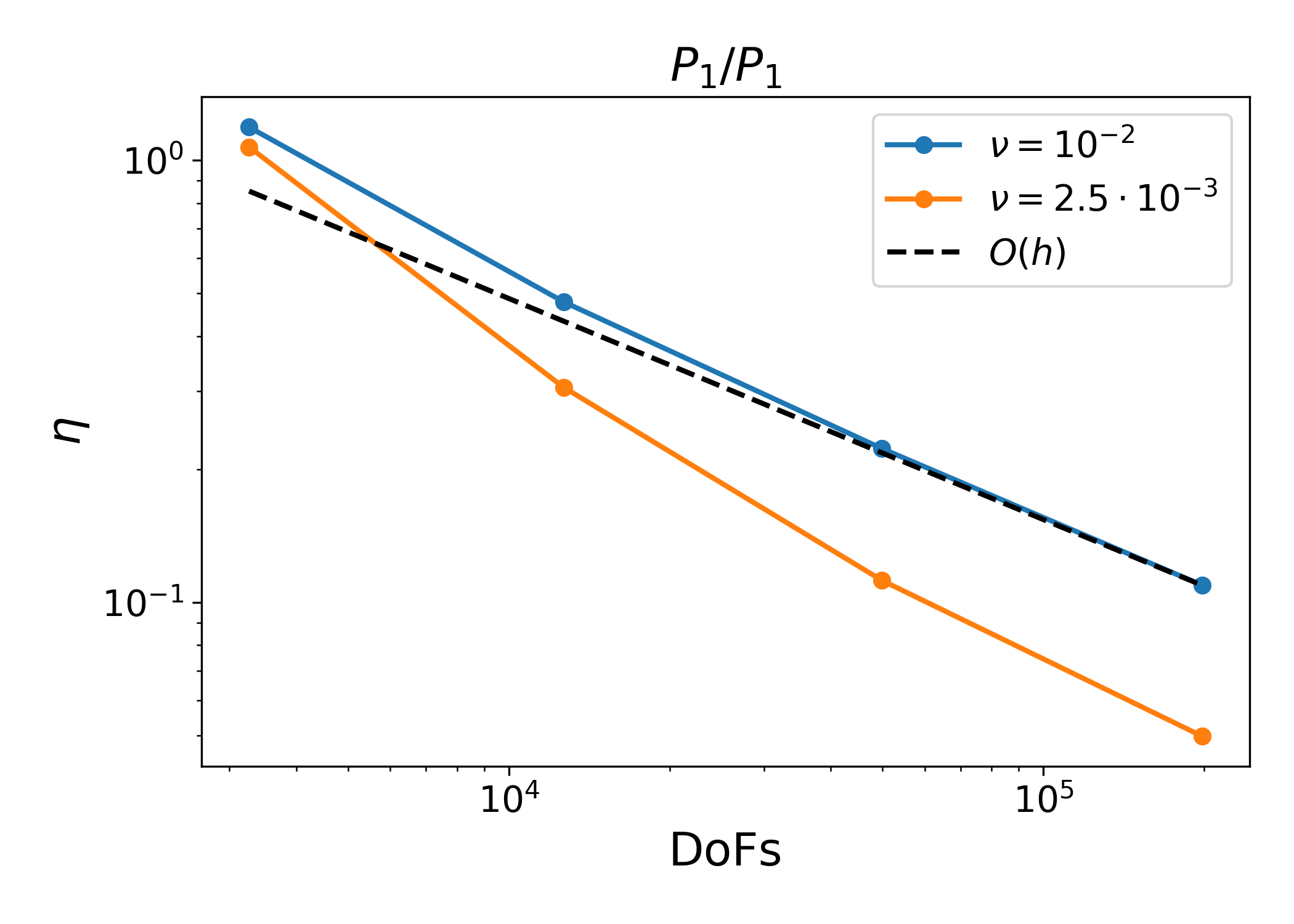}
\hspace*{0.5em}
\includegraphics[width=0.32\textwidth]{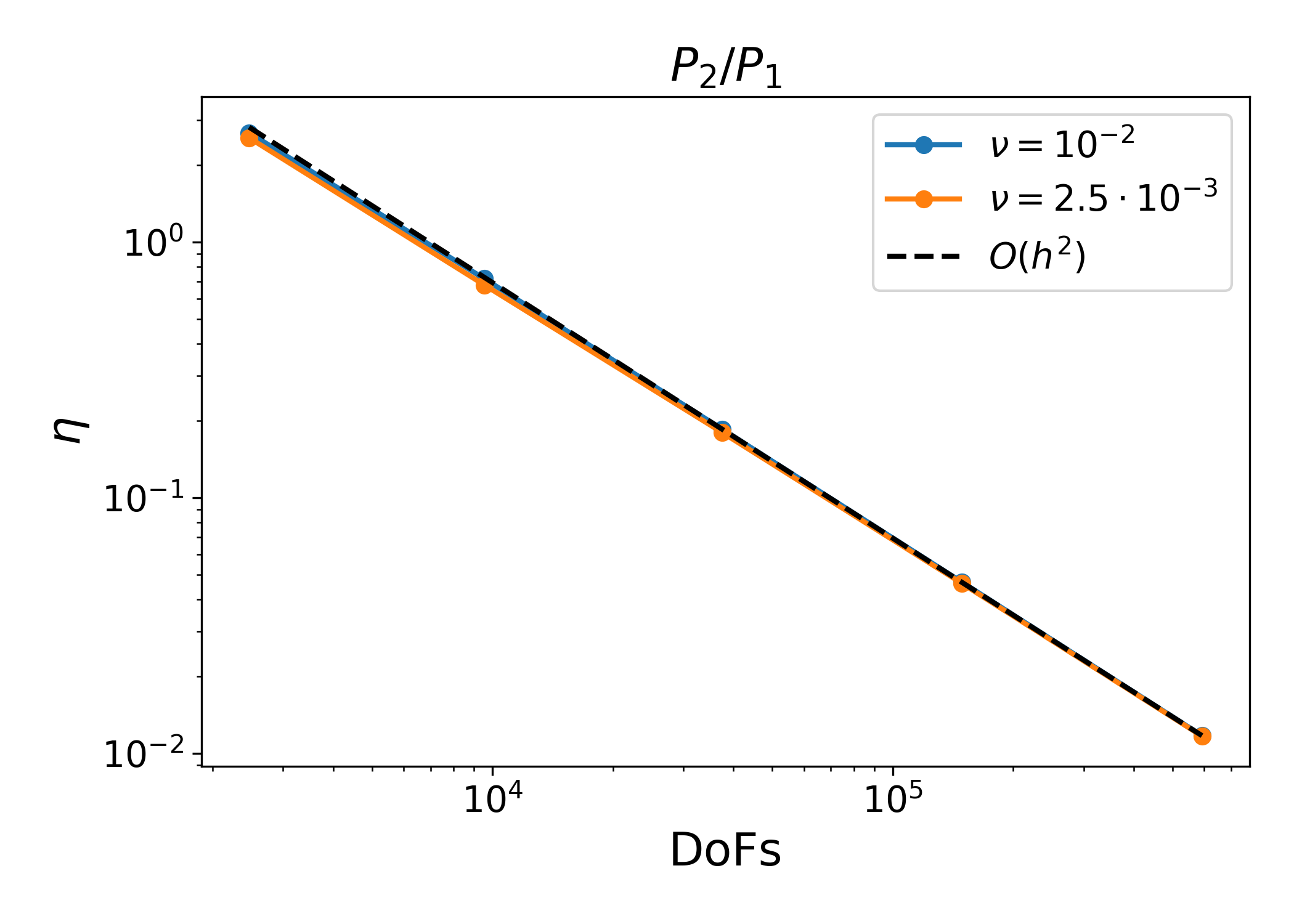}
\hspace*{0.5em}\includegraphics[width=0.32\textwidth]{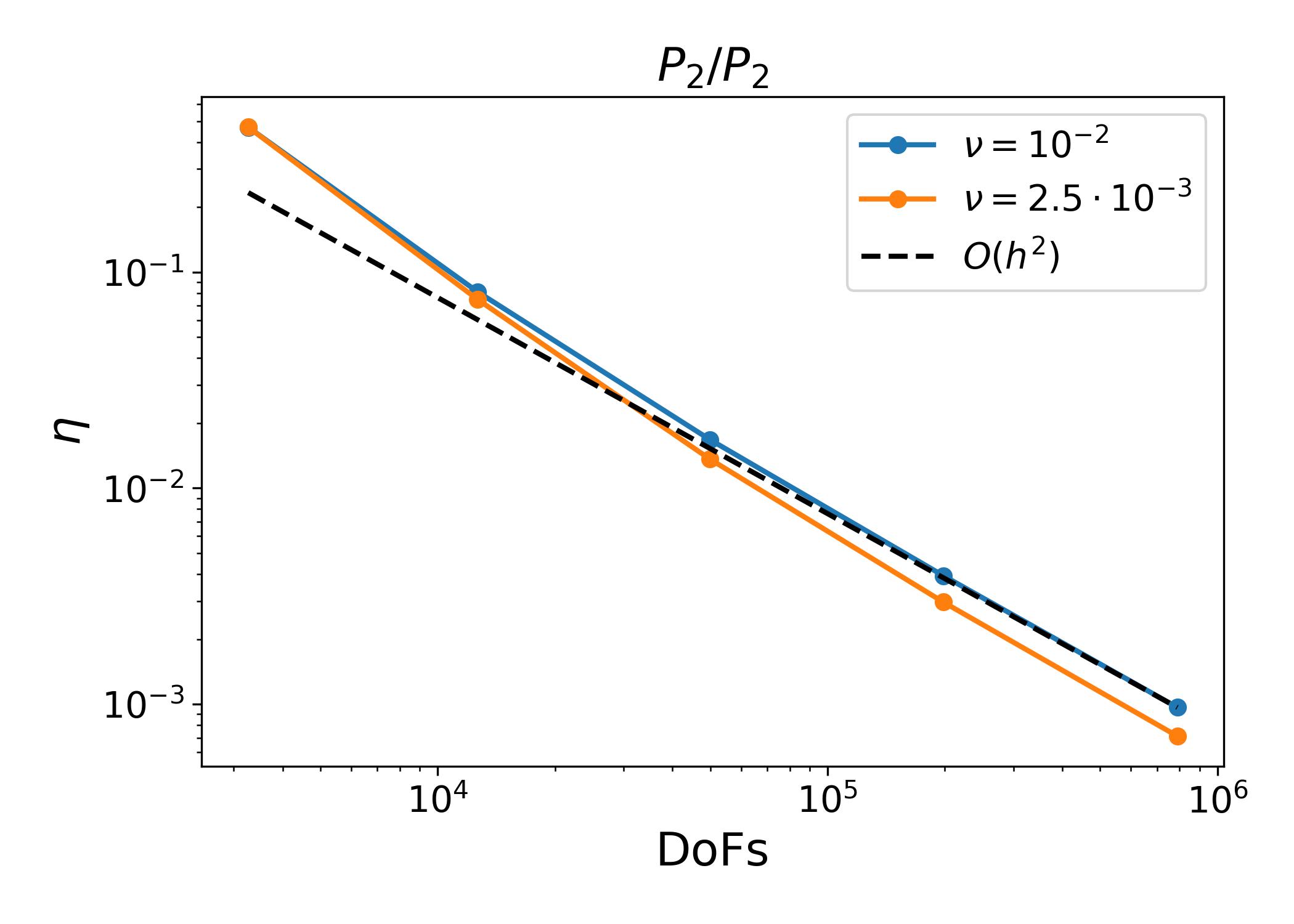}}
\centerline{\includegraphics[width=0.32\textwidth]{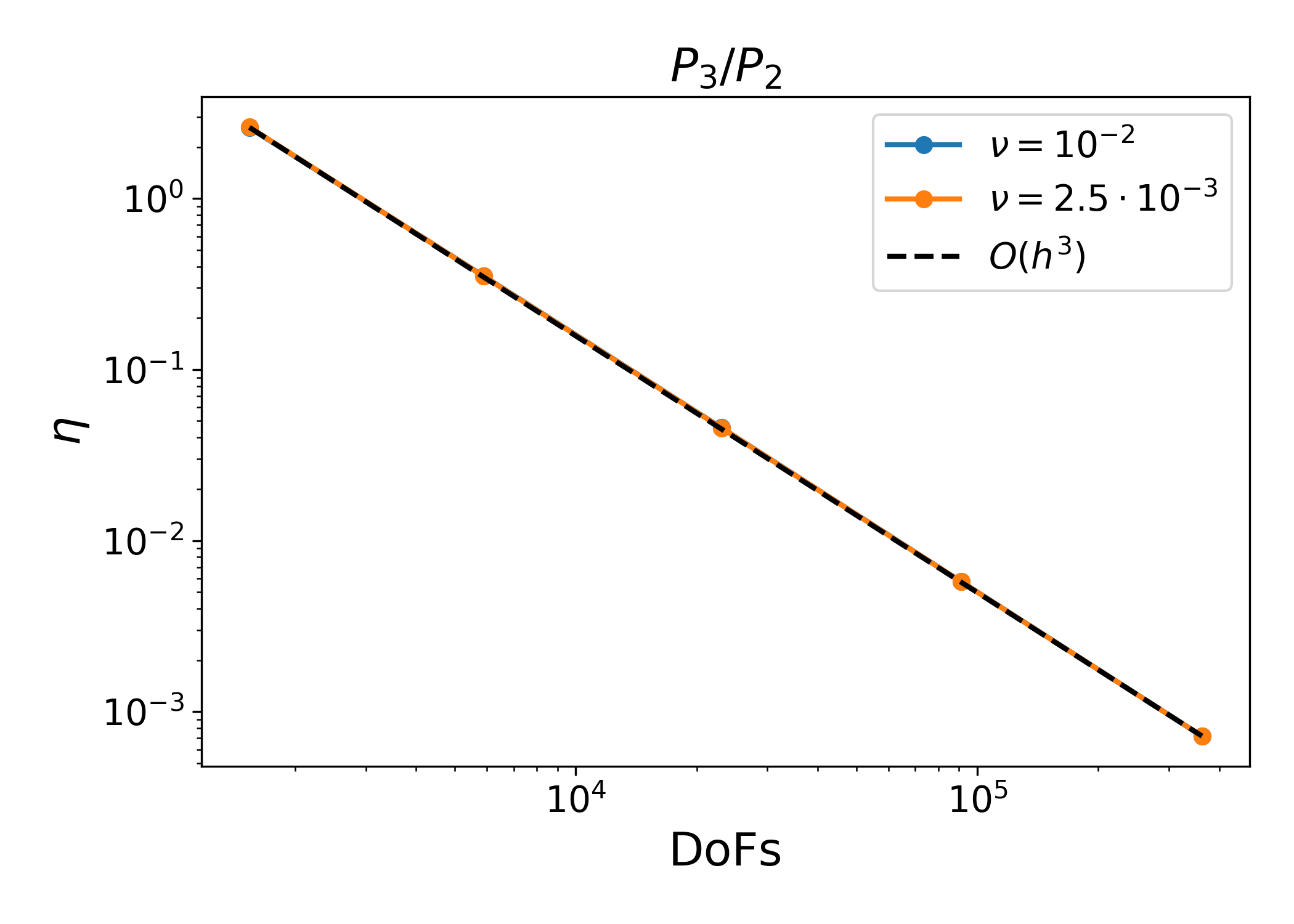}
\hspace*{0.5em}
\includegraphics[width=0.32\textwidth]{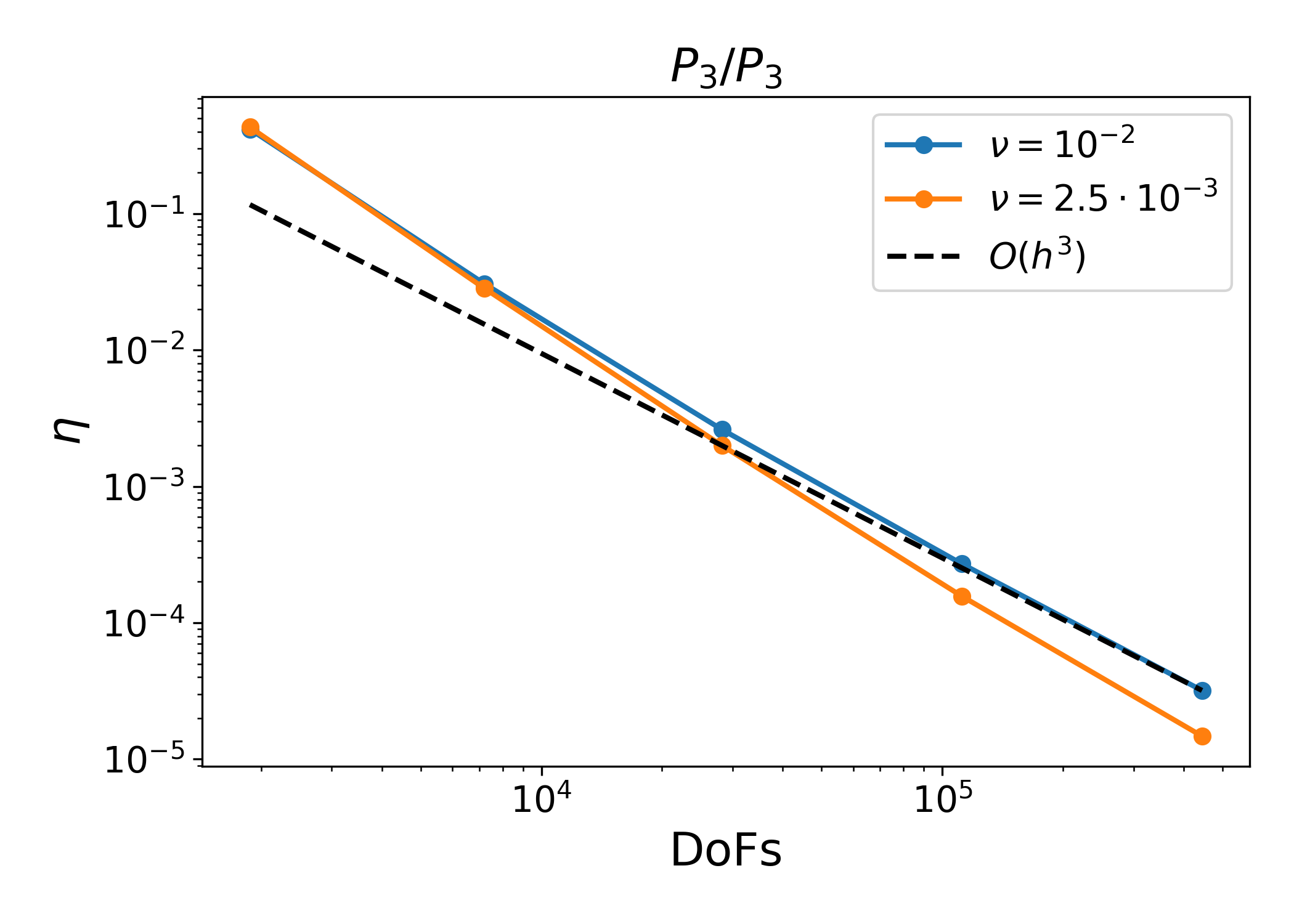}}
  \caption{Navier--Stokes problem. The a posteriori error estimator $\eta$ 
  for different pairs of finite element spaces and different values of the viscosity coefficient.}
  \label{fig:nse-smooth-all-estimate}
\end{figure}

\begin{figure}[t!]
\centerline{
  \includegraphics[width=0.32\textwidth]{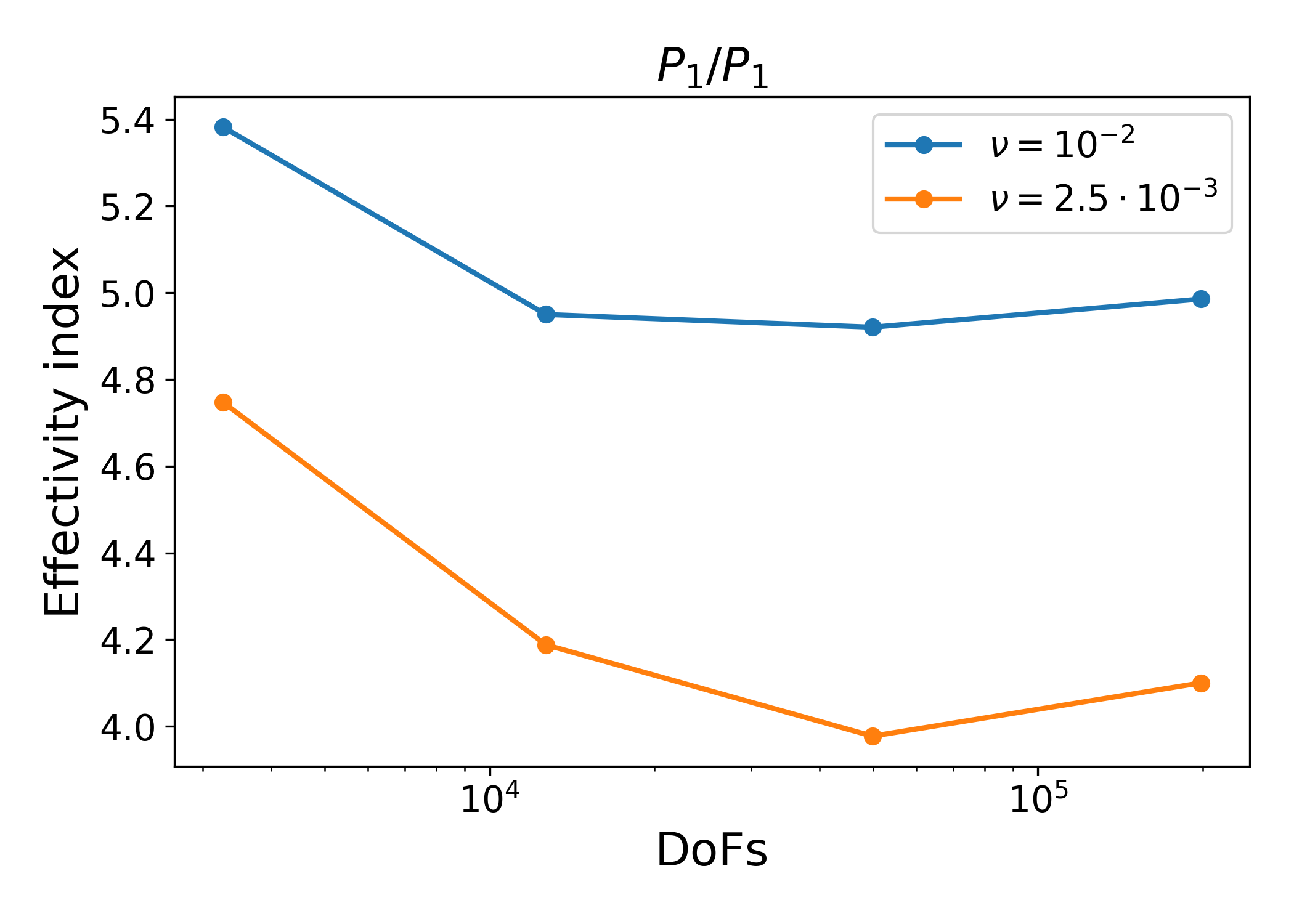}
  \hspace*{0.5em}
  \includegraphics[width=0.32\textwidth]{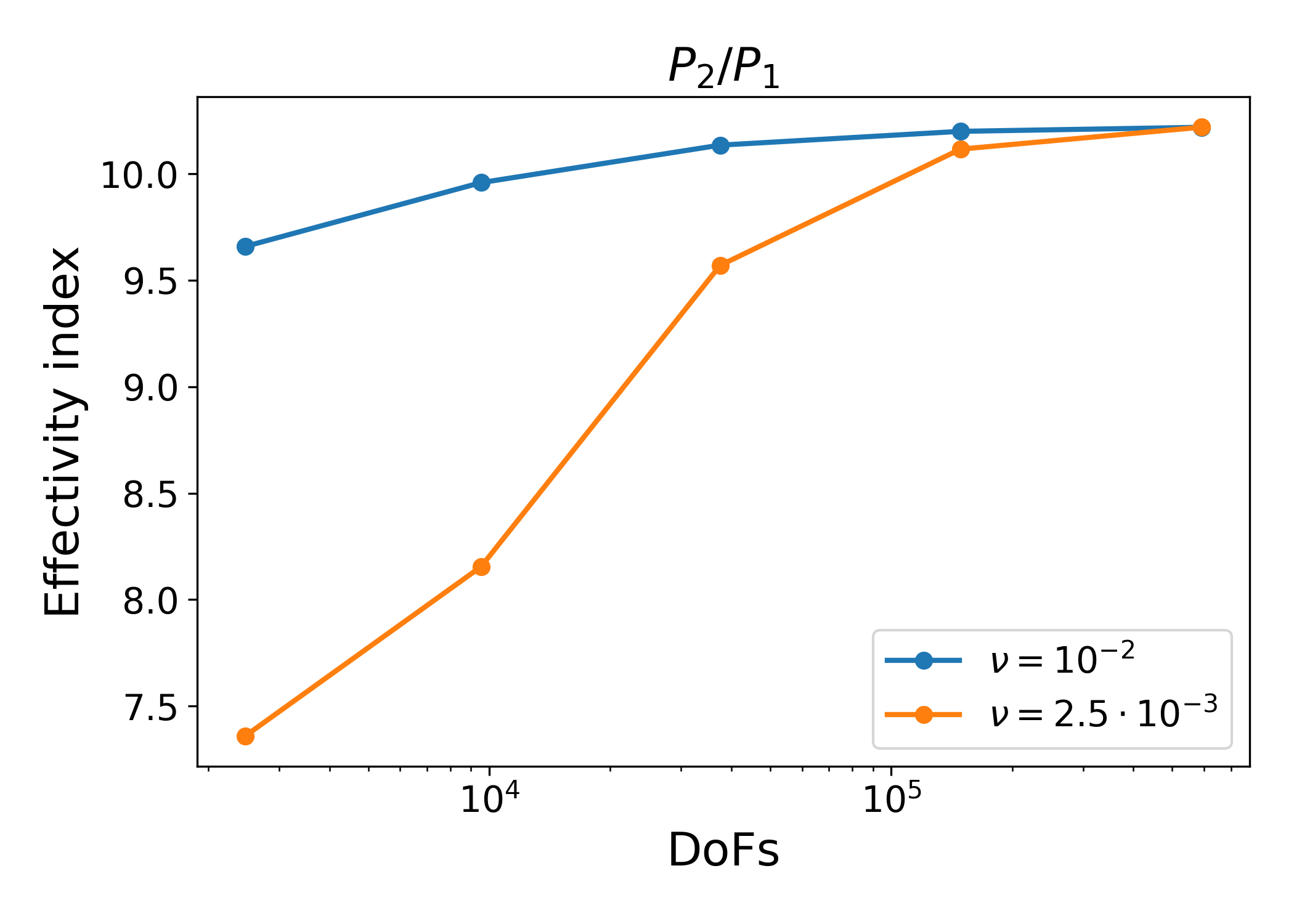}
  \hspace*{0.5em}
  \includegraphics[width=0.32\textwidth]{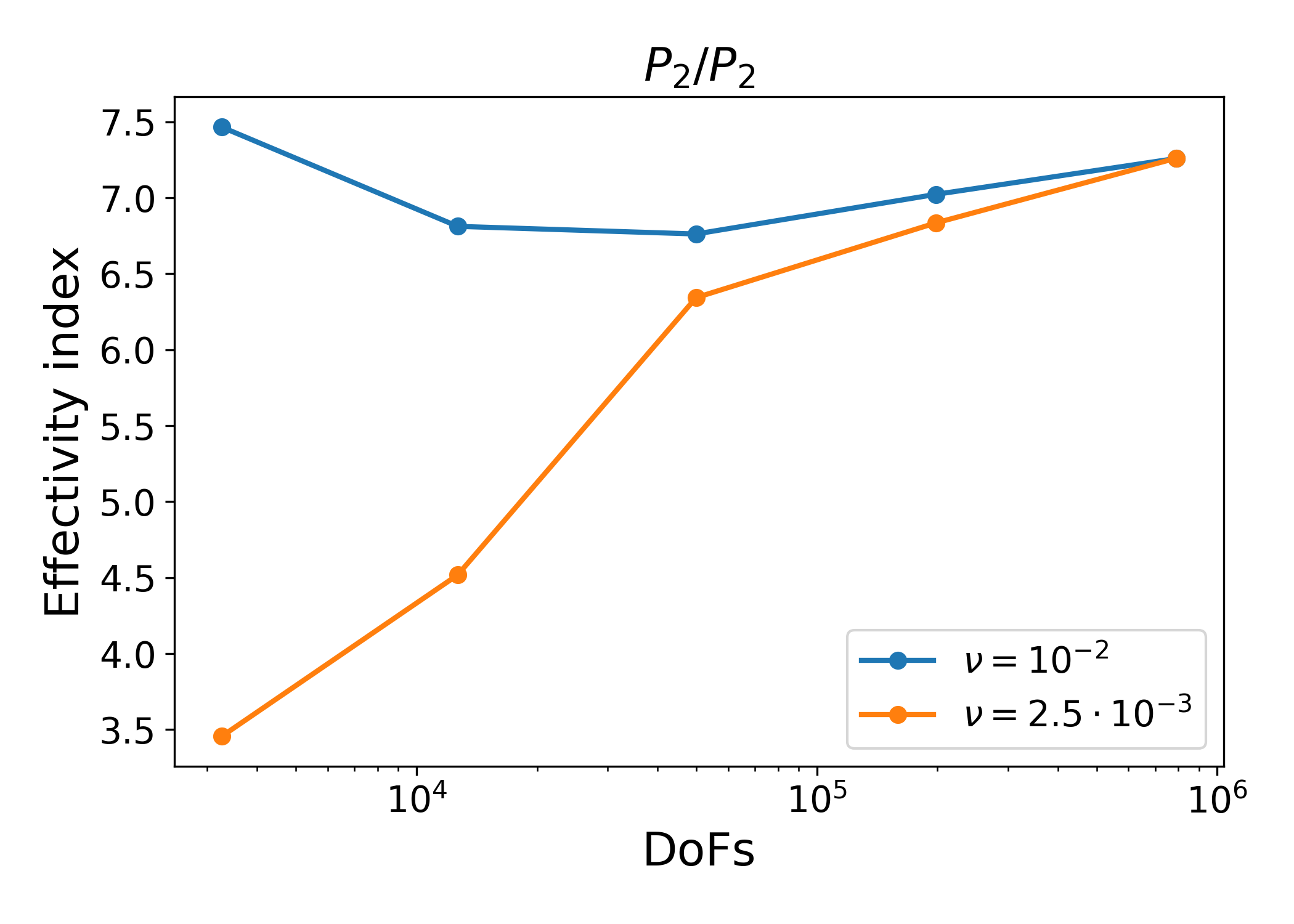}}
\centerline{\includegraphics[width=0.32\textwidth]{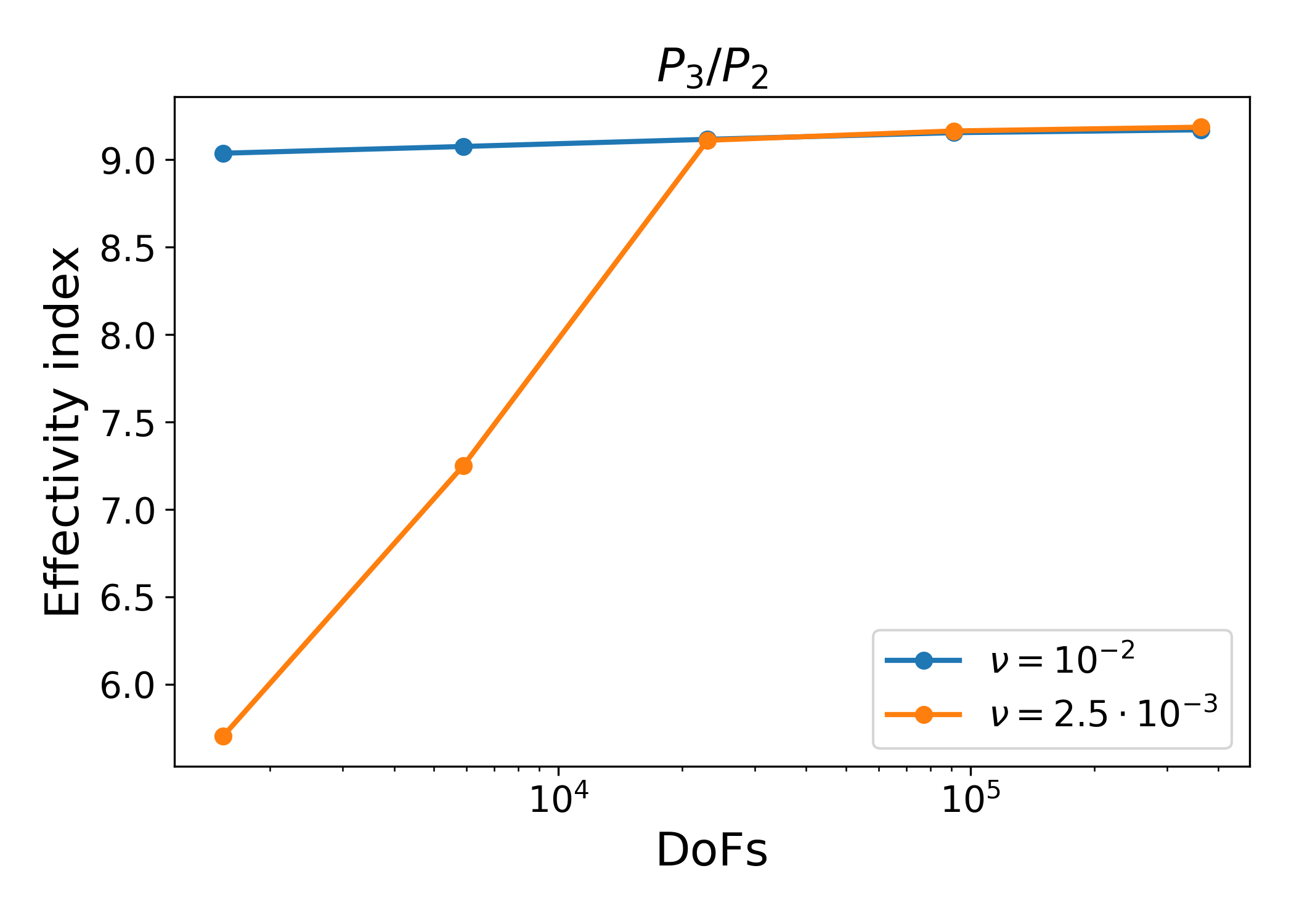}
\hspace*{0.5em}
  \includegraphics[width=0.32\textwidth]{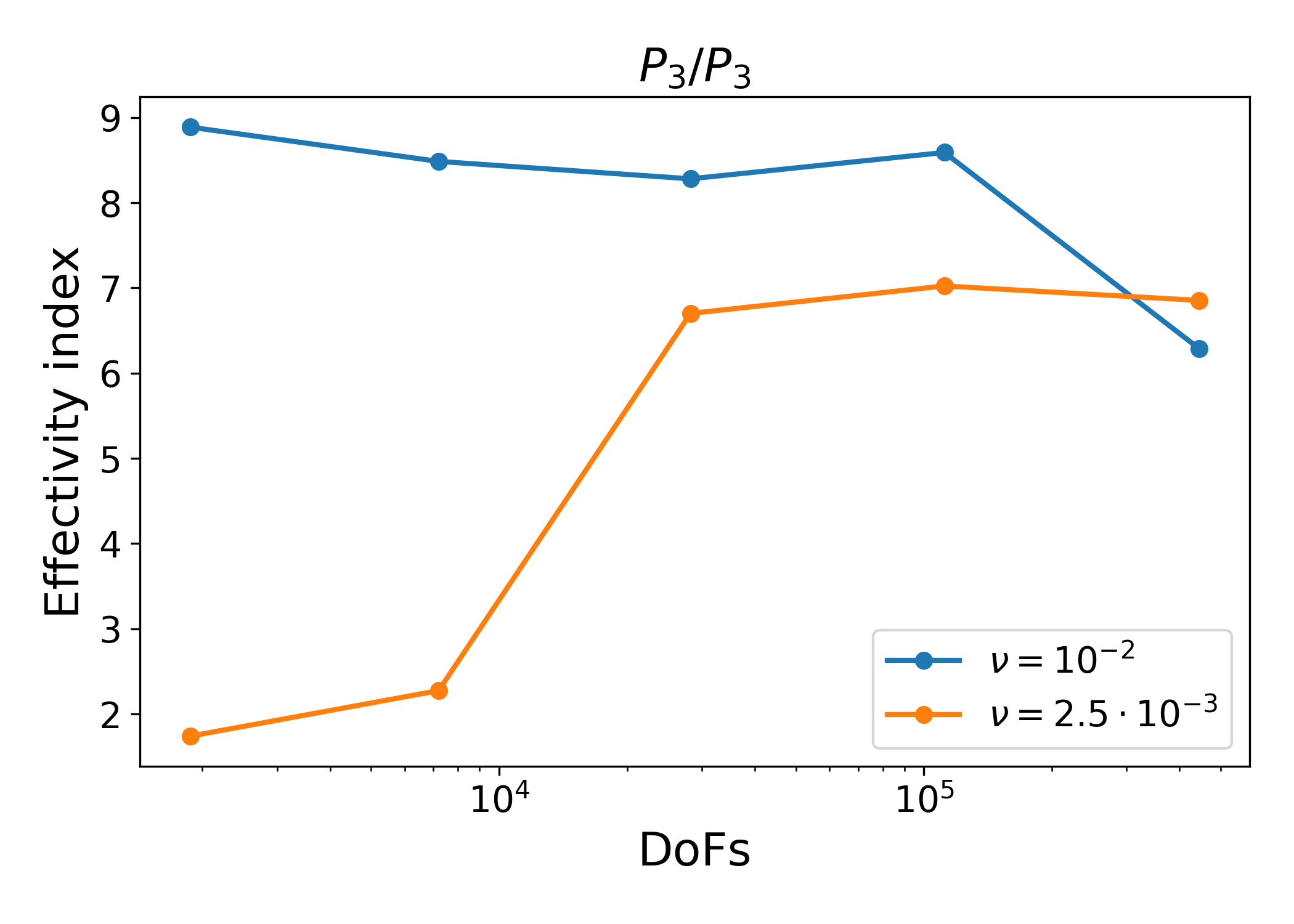}}
  \caption{Navier--Stokes problem. Effectivity indices for different pairs of finite element spaces and different values of the viscosity coefficient.}
  \label{fig:nse-smooth-all-effectivity}
\end{figure}

Figures~\ref{fig:nse-smooth-all-error}--\ref{fig:nse-smooth-all-effectivity} present results for 
two values of the viscosity coefficient, $\nu = 1/100$ and $\nu=1/400$. For smaller values of this 
coefficients, there were pairs of spaces where it was not possible to solve the nonlinear problem 
to the required accuracy (Euclidean norm of the residual vector smaller than $10^{-10}$). 
It can be observed that both the error in the norm $\no{\(\cdot,\cdot\)}{\mathrm{spg,nse}}$ and 
its estimator show the expected order of convergence or even a pre-asymptotic higher order of error 
reduction for the smaller viscosity coefficient. The effectivity indices are 
for all pairs of spaces, all levels, and both viscosity coefficients, in the same range as for 
the corresponding Oseen problem, see Figure~\ref{fig:smooth-all-effectivity}.

In summary, the presented numerical study shows that the proposed estimator gives good results 
for estimating the error in a norm that can be used for the a priori error analysis of the 
SUPG/PSPG/grad-div method for the steady-state incompressible Navier--Stokes problem. 

\section*{Acknowledgments} The work of M.~Afzal was supported by Gulf University for Science and Technology through GUST internal seed grant case no 145 and by the 
WIAS. The research was performed while M.~Afzal visited WIAS. 

\bibliographystyle{plainnat}
\bibliography{ref}

\end{document}